\documentclass[11pt]{article}

\usepackage[T1]{fontenc}
\usepackage{lmodern}
\usepackage{amssymb,amsmath,amsthm}
\usepackage{mathtools}
\usepackage{bm, bbm}
\usepackage{enumitem}
\usepackage{hyperref}
\usepackage{fullpage}
\usepackage[dvipsnames,svgnames,table]{xcolor}
\usepackage{graphicx}
\usepackage{cases}
\usepackage{multirow, multicol, adjustbox, makecell}
\usepackage{array, booktabs}
\usepackage{caption}
\usepackage{tikz-cd}
\usepackage{algpseudocode, algorithm, algorithmicx}
\algrenewcommand\algorithmicrequire{\textbf{Input:}}
\algrenewcommand\algorithmicensure{\textbf{Output:}}

\newcommand*{\supp}{\mathrm{supp}}

\newcommand*{\sgn}{\mathrm{sign}}

\newcommand*{\argmin}[1]{{\underset{#1}{\operatorname{argmin}}}}

\newcommand*{\spn}{\mathrm{span}}
\newcommand*{\tr}{\mathrm{tr}}  
\newcommand*{\Lip}{\mathrm{Lip}}  

\newcommand{\ind}[1]{\bm{1}_{#1}}

\def\D{\,\mathrm{d}}

\newcommand*{\nm}[1]{{\left\|#1\right \|}}
\newcommand*{\nmd}[1]{{\|#1\|}}    

\newcommand{\abs}[1]{\left|#1\right|}
\newcommand{\absd}[1]{|#1|}    

\newcommand*{\ip}[2]{\left\langle #1, #2 \right\rangle}
\newcommand*{\ipd}[2]{\langle #1, #2 \rangle}    

\newcounter{counter}
\numberwithin{counter}{section}
\numberwithin{equation}{section}

\theoremstyle{plain}
\newtheorem{theorem}[counter]{Theorem}
\newtheorem{lemma}[counter]{Lemma}

\newtheorem{proposition}[counter]{Proposition}
\newtheorem{assumption}[counter]{Assumption}

\newtheorem{remark}[counter]{Remark}

\newtheorem*{theorem*}{Theorem}
\newtheorem*{lemma*}{Lemma}
\newtheorem*{corollary*}{Corollary}
\newtheorem*{conjecture*}{Conjecture}
\newtheorem*{proposition*}{Proposition}
\newtheorem*{assumption*}{Assumption}
\newtheorem*{remark*}{Remark}
\newtheorem*{notation*}{Notation}

\theoremstyle{definition}
\newtheorem{definition}[counter]{Definition}
\newtheorem{example}[counter]{Example}

\newtheorem*{definition*}{Definition}
\newtheorem*{example*}{Example}
\newtheorem*{discussion*}{Discussion}

\newcommand*{\cA}{\mathcal{A}}

\newcommand*{\cD}{\mathcal{D}}
\newcommand*{\cE}{\mathcal{E}}
\newcommand*{\cF}{\mathcal{F}}

\newcommand*{\cH}{\mathcal{H}}
\newcommand*{\cI}{\mathcal{I}}

\newcommand*{\cK}{\mathcal{K}}
\newcommand*{\cL}{\mathcal{L}}

\newcommand*{\cN}{\mathcal{N}}

\newcommand*{\cO}{\mathcal{O}}
\newcommand*{\cP}{\mathcal{P}}

\newcommand*{\cT}{\mathcal{T}}
\newcommand*{\cV}{\mathcal{V}}

\newcommand*{\cX}{\mathcal{X}}
\newcommand*{\cY}{\mathcal{Y}}

\newcommand*{\bbE}{\mathbb{E}}

\newcommand*{\bbN}{\mathbb{N}}

\newcommand*{\bbR}{\mathbb{R}}

\newcommand*{\bbP}{\mathbb{P}}

\newcommand*{\wtC}{\widetilde{C}}

\newcommand*{\wtF}{\widetilde{F}}

\newcommand*{\wtH}{\widetilde{H}}

\newcommand*{\wtY}{\widetilde{Y}}

\newcommand*{\whC}{\widehat{C}}

\newcommand*{\whE}{\widehat{E}}
\newcommand*{\whF}{\widehat{F}}
\newcommand*{\whG}{\widehat{G}}

\newcommand*{\whX}{\widehat{X}}
\newcommand*{\whY}{\widehat{Y}}

\newcommand*{\whf}{\widehat{f}}

\DefineNamedColor{named}{Purple}{cmyk}{0.45,0.86,0,0}

\DefineNamedColor{named}{JungleGreen} {cmyk}{0.99,0,0.52,0}

\DefineNamedColor{named}{orange} {rgb}{1,0.55,0}

\usepackage{authblk}
\usepackage{subcaption}

\usepackage{pgfplots}
\usepgfplotslibrary{fillbetween}
\pgfplotsset{compat=1.18}

\newcommand*{\EX}{\cE_{\cX}} 
\newcommand*{\DX}{\cD_{\cX}} 
\newcommand*{\EY}{\cE_{\cY}} 
\newcommand*{\DY}{\cD_{\cY}} 

\newcommand*{\hEX}{\widehat{\cE}_{\cX}} 
\newcommand*{\hDX}{\widehat{\cD}_{\cX}} 
\newcommand*{\hEY}{\widehat{\cE}_{\cY}} 
\newcommand*{\hDY}{\widehat{\cD}_{\cY}} 

\newcommand*{\Smu}{\Sigma} 
\newcommand*{\hSmu}{\widehat{\Sigma}} 
\newcommand*{\Dmu}{\Delta}

\newcommand*{\HS}{\textup{HS}}

\theoremstyle{plain}
\newtheorem*{mainresult}{Main result}
\newtheorem*{mainproblem}{Main problem}

\usepackage[ 
    backend = biber,
    doi = false,
    maxbibnames = 99,
    giveninits = true,
    urldate = long,
    style = numeric]{biblatex}

\renewbibmacro{in:}{}
\DeclareFieldFormat[article,inbook,incollection,inproceedings,patent,thesis,unpublished]{title}
    {#1\isdot}
\DeclareFieldFormat[article, unpublished, inproceedings]{titlecase}{\MakeSentenceCase*{#1}}
\DeclareFieldFormat{date}{#1}
\DeclareFieldFormat{urldate}{\mkbibparens{accessed on #1}}

\AtEveryBibitem{\clearfield{month}}
\AtEveryBibitem{\clearfield{day}}
\AtEveryBibitem{\clearfield{pagetotal}}

\AtEveryBibitem{\ifentrytype{unpublished}{}{\clearfield{url}}}

\AtBeginBibliography{\small}

\title{Near-Optimal Learning of Gaussian Sobolev Operators}

\date{}

\author[1]{Ben Adcock}
\author[2,3]{Michael Griebel}
\author[2,3]{Gregor Maier}

\affil[1]{Department of Mathematics, Simon Fraser University, Burnaby~BC, Canada}
\affil[2]{Institute for Numerical Simulation, University of Bonn, Bonn, Germany}
\affil[3]{Fraunhofer Institute for Algorithms and Scientific Computing (SCAI), Sankt Augustin, Germany}

\begin{document}

\maketitle

\begin{abstract}
    A key question in operator learning is how to design surrogate operators with provable approximation guarantees in reasonable computational time. Whereas holomorphic operators can be approximated efficiently, i.e., with at least algebraic convergence in the number of training samples, learning finitely regular operators is known to be less efficient. The reason is an intrinsic curse of sample complexity, which allows only subalgebraic convergence rates. This fact makes it all the more important to develop algorithms which provably achieve these rates. In this work, we present a fully data-driven algorithm, termed Hermite-PCA approximation, for learning Gaussian Sobolev operators with near-optimal sample complexity. It employs principal component analysis and weighted least-squares methods and is therefore computationally efficient. Moreover, it is spectral, in the sense that it achieves faster (and near-optimal) convergence as the Sobolev regularity increases. We provide a full error analysis of this algorithm along with numerical experiments that match our theoretical results and empirically confirm the efficacy of Hermite-PCA approximation for learning Sobolev operators.
\end{abstract}


\vspace{1pc}
\noindent \textbf{Corresponding author:} \url{maier@ins.uni-bonn.de} (Gregor Maier)

\vspace{1pc}


\section{Introduction}

Many natural phenomena, despite their diverse individual nature, can be modeled at an abstract level by operators
\[
F : \cX \to \cY,
\]
which map between infinite-dimensional spaces \( \cX \) and \( \cY \). In computational science and engineering (CSE), for example, one is often interested in the solution operator of a continuum model which is governed by partial differential equations (PDEs) or variational inequalities. In this context, \( F \) typically maps inputs, such as parameters, boundary conditions, or coefficient functions, to outputs, such as states, fields, or observables. Prototypical examples with high and low operator regularity include holomorphic parameter-to-solution mappings in parametric PDEs and obstacle-to-solution mappings in obstacle problems, respectively. The latter exhibit only finite Sobolev regularity with respect to Gaussian measures. While there are many results on the approximation of smooth operators in the literature, comparatively few works discuss both theoretical and practical approximation properties for operators with only finite regularity. To bridge this gap, we present in this work a practical, fully data-driven algorithm, termed \emph{Hermite-PCA approximation}, for learning Gaussian Sobolev operators from noisy, pointwise samples. It combines empirical principal component analysis (PCA) for dimension reduction with weighted least-squares approximation based on Hermite polynomials. Our main contribution is a full error analysis of this approximation --- taking into account all sources of error in the problem --- for the learning of operators with Gaussian Sobolev regularity. As we demonstrate, our algorithm is \textit{spectral} --- the smoother the operator, the faster the convergence --- and, in the absence of discretization errors, its convergence rate is \textit{near-optimal}. Moreover, we provide numerical experiments which empirically validate our theoretical results.

\subsection{Motivations}

A key question for practical applications is how to design suitable surrogate operators, i.e., approximations to \( F \), that achieve high accuracy with feasible computational costs. In~\emph{operator learning}, methods from machine learning, especially (deep) neural networks (NNs), are incorporated into the surrogate operator design in order to meet these demands.
Many different designs, typically referred to as~\emph{neural operators}, along with theoretical error analyses have been proposed in recent years, often achieving impressive performance in numerical experiments. Among those are, for example, PCA-Net~\cite{bhattacharya_ModelReductionNeuralNetworks_2021,lanthaler_OperatorLearningPCANetUpper_2023}, DeepONet~\cite{lu_LearningNonlinearOperatorsDeepONet_2021, lanthaler_ErrorEstimatesDeepONetsDeep_2022}, FNO~\cite{li_FourierNeuralOperatorParametric_2021,kovachki_UniversalApproximationErrorBounds_2021},
general neural operators~\cite{kovachki_NeuralOperatorLearningMaps_2023}, Poseidon~\cite{herde_PoseidonEfficientFoundationModels_2024}, and variants thereof. We also refer to the reviews~\cite{boulle_MathematicalGuideOperatorLearning_2024,kovachki_OperatorLearningAlgorithmsAnalysis_2024,subedi_OperatorLearningStatisticalPerspective_2026} and references therein. 

However, neural operators face two drawbacks which are critical for their reliable deployment in CSE applications. First, they lack interpretability.
Second, it is generally unclear whether a neural operator with desired performance guarantees can be obtained from practical optimization routines.
For this reason, it is important to compare neural operators to simpler NN-free operator surrogates that alleviate these shortcomings.
In~\cite{batlle_KernelMethodsAreCompetitive_2024}, the authors present a framework for learning operators based on the theory of operator-valued reproducing kernel Hilbert spaces and Gaussian processes along with convergence guarantees and rigorous a priori error bounds. They provide numerical experiments which show that their kernel method is competitive with neural operators across various benchmark problems.
In~\cite{westermann_PerformanceNeuralPolynomialOperator_2026}, the authors empirically compare the performance of neural and polynomial operator surrogates in extensive numerical experiments. Their key conclusion is that there is no universally best surrogate model.
Polynomial surrogates seemingly perform significantly better than neural operators for smooth inputs, whereas neural operators tend to be superior for rough input data. However, the former are typically computationally much cheaper to train, as they rely on simple procedures such as linear least-squares.

These findings motivate placing greater emphasis on the design of NN-free operator surrogates. This work contributes to this agenda, in that we provide a polynomial surrogate model based on generalized Wiener--Hermite polynomial chaos expansions together with a non-intrusive training procedure via weighted least-squares approximation with carefully chosen pointwise training samples.

\subsection{Contributions}

We consider a Lipschitz continuous operator \( F : \cX \to \cY \) between real separable Hilbert spaces \( \cX, \cY \).
We equip \( \cX \) with an unknown Gaussian measure \( \mu \) and denote the unknown PCA eigenvalues of its covariance operator by \( \lambda_1 \geq \lambda_ 2 \geq \cdots > 0 \). We assume pointwise access to \( F \), with operator evaluations corrupted by additive noise with noise level \( \sigma \geq 0 \).
Our surrogate operator design has the standard encoder-decoder structure, \( \whF = \hDY \circ \whf \circ \hEX \), consisting of a linear encoder \( \hEX : \cX \to \bbR^{d_{\cX}} \), a linear decoder \( \hDY : \bbR^{d_{\cY}} \to \cY \), and a latent function \( \whf : \bbR^{d_{\cX}} \to \bbR^{d_{\cY}} \), see Figure~\ref{fig: surrogate operator design}.

\begin{figure}[h]
\centering
    \begin{tikzcd}
        \cX \arrow[rr, " F "] \arrow[dd , bend right, rightarrow, "\hEX"{pos=0.55}, swap, bend right=40, blue]
        && \cY \arrow[dd, bend right, rightarrow, "\hEY"{pos=0.4}, dashed, swap, bend right=40] \\
        && \\
        \bbR^{d_\cX} \arrow[rr, rightarrow, "\whf"{swap}, blue, "\whF"{yshift=2ex}]
        \arrow[uu, bend right, rightarrow, "\hDX"{pos=0.6}, dashed, swap, bend right=40]
        && \bbR^{d_\cY} \arrow[uu, bend right, rightarrow, "\hDY"{pos=0.45}, swap, bend right=40, blue]
    \end{tikzcd}
    \vspace{-2ex}
    \caption{Encoder-decoder surrogate operator design, \( \whF := \hDY \circ \whf \circ \hEX \approx F \).}
    \label{fig: surrogate operator design}
\end{figure}

We present an algorithm, called \emph{Hermite-PCA approximation}, to construct \( \whF \). It has two main parts. First, we learn the encoder and decoder by empirical PCA based on \( N_{\cX} \) and \( N_{\cY} \) samples, respectively. Second, we construct a suitable approximation space of vector-valued linear combinations of \( s \) Hermite polynomials, from which to select \( \whf \).
The Hermite coefficients of \( \whf \) are then learned by weighted least-squares approximation based on \( M \) labeled training samples. For this, we employ Christoffel sampling methods to identify an optimal distribution of the input training data. 

We denote by \( \hat{\varrho} \) a Gaussian measure on \( \bbR^{d_{\cX}} \) that is induced by the encoder. As usual, we write \( x \lesssim y \) for \( x, y \in \bbR \) if there exists a constant \( C > 0 \) such that \( x \leq C y \). We write \( x \lesssim_{a,b,c,\dots} y \) if the constant \( C \) depends on further parameters \( a, b, c , \dots \). We use the notation \( y \gtrsim x \) if \( x \lesssim y \). We can now state an informal version of our main result.

\begin{mainresult}[Hermite-PCA approximation for Sobolev operators; Theorem~\ref{thm: Hermite-PCA Sobolev error bound}]
We provide lower bounds for the number of data samples \( N_{\cX}, N_{\cY}, M \) such that the following holds for all sufficiently large \(s\): Suppose that \( F \) is \(L\)-Lipschitz, and its latent space representation \( \hEY \circ F \circ \hDX \) belongs to the Gaussian Sobolev space \( H^k_{\hat{\varrho}}( \bbR^{d_{\cX}} ; \bbR^{d_{\cY}} ) \). Then, with high probability, the algorithm described above yields a surrogate operator \( \whF \) which satisfies the error bound
\begin{equation}
\label{eq: main result; informal}
\begin{split}
    \nmd{F - \whF}_{L_{\mu}^2(\cX; \cY)}
    & \lesssim L \sqrt{\sum^{\dim(\cX)}_{i = d_{\cX}+1} \lambda_i} + \sqrt{\sum^{\dim(\cY)}_{i = d_{\cY}+1} \lambda^{F \sharp \mu}_i} + (C + \sigma ) \left [ \left ( \frac{d_{\cX}}{N_{\cX}} \right )^{1/4} + \left ( \frac{d_{\cY}}{N_{\cY}} \right )^{1/4} \right ]
    \\
    &~~+ C \left [ \bigg(6 + 6 k \sum_{i=1}^{d_{\cX}} \lambda_i^{-1} \bigg)^{k/2} w_s^k \, \nmd{ \hEY \circ F \circ \hDX }_{H_{\hat{\varrho}}^{k}(\bbR^{d_{\cX}}; \bbR^{d_{\cY}})} + \left( \sqrt{\frac{s}{M} } + 1 \right) \sigma \right ]
\end{split}
\end{equation}
with some constant \( C = C( \nm{F(0)}_{\cY}, L ) \). The numbers \( \lambda^{F \sharp \mu}_i \) denote the PCA eigenvalues of the (uncentered) covariance operator of the pushforward measure \( F \sharp \mu \), and \( w_s \) denotes a weight decaying to zero as \( s \to \infty \), depending on the decay of the \( \lambda_i \). The parameter \( s \) is determined via the number of labeled training samples \( M \) through a log-linear relationship \( M \gtrsim s \log(s) \), see~\eqref{eq: main condition for M}.
\end{mainresult}

In fact, this result follows from a more general error bound which holds for arbitrary Lipschitz operators and does not require any additional Sobolev regularity, see Theorem~\ref{thm: Hermite-PCA error bound}.

We now summarize our main contributions:
\begin{itemize}[leftmargin=*]
    \item \textbf{Practicality:} We present a fully data-driven algorithm to construct a polynomial surrogate operator for Lipschitz operators with Sobolev regularity from pointwise, noisy samples. The only assumption in the problem setup is that \( \mu \) is a nondegenerate, centered Gaussian measure, but its covariance structure is unknown (up to a unit trace assumption) and needs to be learned from data. Our algorithm employs linear least-squares fitting, and is therefore computationally efficient.
    
    \item \textbf{Weak assumptions:} Our measure \( \mu \) has unbounded support. This complicates the analysis but is the natural support condition under minimal assumptions on the input data distribution. This situation has, to the best of our knowledge, not been adequately addressed in the literature so far, as existing results typically require bounded support assumptions. In addition, we also allow for arbitrary decay of the PCA eigenvalues \( \lambda_i \), not just algebraic as typically considered in the literature.

    \item \textbf{Full error analysis:} The bound~\eqref{eq: main result; informal} takes into account all sources of error originating from the model design and learning procedure. It bounds the overall approximation error by (empirical) PCA projection errors, an approximation error in latent space, and a noise error. It explicitly shows the effect of the (hyper-)parameters of the problem on each error term, thus enabling optimal design choices for practical applications.

    \item \textbf{Optimal approximation rates:} In the absence of the PCA error terms and noise, we show that the error bound~\eqref{eq: main result; informal} achieves optimal approximation rates. In fact, we argue that the approximation error admits a matching lower bound if we additionally assume that \(F\) is bounded. We provide explicit examples for the decay behavior of the weight \( w_s \) as \( s \to \infty \) in the case of an infinite-dimensional input space and algebraically and exponentially decaying \( \lambda_i \).

    \item \textbf{Spectral approximation property:} The bound~\eqref{eq: main result; informal} reveals that the approximation error is determined by the weight \( w_s^k \) in the case where the latent mapping \( \hEY \circ F \circ \hDX \) is \(k\)-Sobolev regular. As the algorithm itself is independent of \( k \), it thus automatically achieves faster convergence rates for higher-order Sobolev operators. In particular, for infinitely smooth operators the error decays faster than \( w_s^k \) as \( s \rightarrow \infty \) for any $k$.

    \item \textbf{Universality:} Our algorithm chooses a certain polynomial space for Sobolev regular operators.
    However, it can be easily adapted to other types of regularity, e.g., mixed regularity, as well. In fact, our main theoretical contributions in their most general form consider arbitrary polynomial spaces and impose no regularity assumption on the operator.

    \item \textbf{Numerical experiments:} We provide empirical validation of our theoretical results by applying our algorithm to learn (i) the obstacle-to-solution operator of an obstacle problem and (ii) functionals of varying Sobolev regularity.
\end{itemize}

\newpage

\subsection{Related work}

PCA methodologies for dimensionality reduction are prominent among classical reduced basis methods. In the context of operator learning, they are employed in~\cite{hesthaven_NonintrusiveReducedOrderModeling_2018,wang_NonintrusiveReducedOrderModeling_2019}, though only for the decoder in the output space. The work~\cite{bhattacharya_ModelReductionNeuralNetworks_2021} introduces PCA-Net, which applies empirical PCA for both the encoder and decoder in the input and output space, respectively, and is thus closest to the encoder-decoder scheme used in this work. Its theoretical properties are further analyzed in~\cite{lanthaler_OperatorLearningPCANetUpper_2023}. PCA-Net models the latent function \( \whf \) as a deep NN. In contrast, we use Hermite polynomials to obtain a fully NN-free approximation scheme.
We mention in passing that there are other operator surrogate designs as well, which are not based on NN approximations but, e.g., on the random feature model~\cite{lanthaler_ErrorBoundsLearningVectorValued_2023,liao_CauchyRandomFeaturesOperator_2025,nelsen_OperatorLearningUsingRandom_2024,nelsen_RandomFeatureModelInputOutput_2021}, or polynomial chaos expansions~\cite{sharma_PolynomialChaosExpansionOperator_2026}.

A long list of recent works has established that holomorphic (nonlinear) operators can be approximated efficiently, that is, with algebraic or even exponential convergence rates. This is both in terms of NN expression rate bounds~\cite{herrmann_NeuralSpectralOperatorSurrogates_2024,schwab_DeepOperatorNetworkApproximation_2026,schwab_DeepLearningHighDimension_2023} as well as sample complexity estimates~\cite{adcock_OptimalApproximationInfinitedimensionalHolomorphic_2024,adcock_OptimalApproximationInfinitedimensionalHolomorphic_2025,adcock_OptimalDeepLearningHolomorphicOperators_2024,bartel_SamplingRecoveryBochnerSpaces_2026,dung_AnalyticitySparsityUncertaintyQuantification_2023}. Such operators arise, e.g., as data-to-solution operators in parametric PDEs and in problems of uncertainty quantification~\cite[Chpt. 4]{adcock_SparsePolynomialApproximationHighDimensional_2022}. 
Algebraic convergence results are also available for the approximation of operators with special structure, such as linear operators~\cite{dehoop_ConvergenceRatesLearningLinear_2023,mollenhauer_LearningLinearOperatorsInfinitedimensional_2024,subedi_OnlineInfiniteDimensionalRegressionLearning_2024}, and pseudo-differential operators~\cite{chen_ConvergenceRatesLearningPseudoDifferential_2026}.
Complexity estimates for non-smooth operators, as we consider in this work, are given in~\cite{adcock_SampleComplexityLearningLipschitz_2026,kovachki_DataComplexityEstimatesOperator_2024,lanthaler_OperatorLearningLipschitzOperators_2024,lanthaler_ParametricComplexityOperatorLearning_2026}. They reveal intrinsic curses of parametric or sample complexities for non-smooth operators on infinite-dimensional domains: no approximation algorithm can achieve an algebraically decaying worst-case error in terms of NN parameters or employed data samples. The experiments in the present work are, to the best of our knowledge, the first numerical evidence of this phenomenon, and ours is the first fully data-driven algorithm that achieves such rates, while also tackling practical issues such as noise and unknown \( \mu \).
We also mention in passing the recent work~\cite{cheng_LearningFrechetDifferentiableOperators_2026a}, which establishes upper (subalgebraic) expression rate bounds with respect to a sub-Gaussian input measure for the approximation of Fréchet differentiable operators with PCA-Net. A concise overview of current results in operator learning theory is given in~\cite{brugiapaglia_ShortTourOperatorLearning_2026}.

The recovery of operators from finite noisy samples is studied in~\cite{adcock_SharpMinimaxRiskBounds_2026,liu_DeepNonparametricEstimationOperators_2024,reinhardt_StatisticalLearningTheoryNeural_2026}. While the last two works only derive upper bounds with a focus on Lipschitz and holomorphic operators, respectively, the authors of~\cite{adcock_SharpMinimaxRiskBounds_2026} prove information-theoretic upper and matching (or near-matching) lower bounds for the minimax risk of Lipschitz and Hölder operators. Their findings again confirm, in a minimax sense, a curse of sample complexity for finitely regular operators.
However, in contrast to our results, these works require uniform boundedness assumptions for the objective operator. Moreover, they do not establish practical approximation algorithms or provide numerical results, as we do in this~work.

Weighted least-squares approximation and Christoffel sampling techniques, originally developed for scalar-valued approximations to achieve near-optimal sample complexities~\cite{cohen_OptimalWeightedLeastsquaresMethods_2017}, have recently been generalized to the Hilbert-valued setting~\cite{adcock_OptimalSamplingLeastsquaresApproximation_2025,adcock_SparsePolynomialApproximationHighDimensional_2022,adcock_OptimalSamplingLearningSparse_2022,bartel_SamplingRecoveryBochnerSpaces_2026}. We also mention the work \cite{bohn_ErrorAnalysisRegularizedUnregularized_2017} which provides an error analysis of regularized and unregularized least-squares regression problems in the Hilbert-valued setting. 
An operator extension of Christoffel sampling is considered in \cite{turnage_OptimalWeightedLeastSquaresMethod_2025}. However, that work does not address the practical scenario in which the measure \( \mu \) is unknown, which is the setting we consider here. Nor does it provide algorithms that achieve concrete (and, as we establish, near-optimal) approximation rates for Sobolev operators --- a central focus of our analysis.

The Gaussian setting in this paper is also used in~\cite{schwab_DeepLearningHighDimension_2023} and~\cite{adcock_SampleComplexityLearningLipschitz_2026, kovachki_DataComplexityEstimatesOperator_2024} to derive, respectively, expression rate and sample complexity estimates for operator learning, and in~\cite{dung_AnalyticitySparsityUncertaintyQuantification_2023} to study problems in uncertainty quantification with Gaussian random field inputs. 
However, those works are theoretical studies that provide no numerical evidence.
In~\cite{luo_DimensionReductionDerivativeInformedOperator_2026}, the authors use the Gaussian setting to analyze approximation errors in derivative-informed operator learning strategies and perform numerical experiments. However, in all the works listed above, the Gaussian measure \( \mu \) is assumed to be known. In contrast, we assume \( \mu \) to be unknown (up to centering, nondegeneracy, and trace normalization) and approximate it from data --- a major asset of our Hermite-PCA method, making it fully data-driven.

\subsection{Outline}

The rest of the paper is organized as follows. Section~\ref{sec: setup and main problem} introduces our setup and main problem in detail. In Section~\ref{sec: construction}, we describe the surrogate operator construction and discuss the resulting Hermite-PCA algorithm. Section~\ref{sec: Hermite-PCA approximation} contains our main result, Theorem~\ref{thm: Hermite-PCA Sobolev error bound}, which is a detailed error analysis of the Hermite-PCA approximation for Gaussian Sobolev operators and 
which we empirically validate in numerical experiments in Section~\ref{sec: numerical}. In Section~\ref{sec: general error bound}, we present a general error bound for the Hermite-PCA method, see Theorem~\ref{thm: Hermite-PCA error bound}. Together with an \(\ell^2\)-characterization of Gaussian Sobolev spaces in Section~\ref{sec: l2-characterization}, we use this result to prove Theorem~\ref{thm: Hermite-PCA Sobolev error bound} in Section~\ref{sec: main res proof}. We finally conclude in Section~\ref{sec: conclusions} with limitations and future work. The appendix provides a number of further definitions and results which are used in the main part of the paper.


\section{Setup and main problem}
\label{sec: setup and main problem}

We begin by presenting our setup in detail and stating the main problem.

\subsection{Setup}\label{sec: setup}

Let \( (\cX, \ip{\cdot}{\cdot}_{\cX}) \), \( (\cY, \ip{\cdot}{\cdot}_{\cY}) \) be real separable Hilbert spaces with dimensions \( \dim(\cX), \dim(\cY) \in \bbN \cup \{\infty\} \).
Let \( F : \cX \to \cY \) be an \(L\)-Lipschitz continuous operator, that is, 
\[
\nm{F(X) - F(X')}_{\cY} \leq L \nm{X - X'}_{\cX}, \quad \forall X, X' \in \cX,
\]
for some Lipschitz constant \(L > 0\).
We assume that we can access \(F\) only through noisy, pointwise evaluations,
\[
Y = F(X) + \sigma E,
\]
where the scalar \(\sigma \geq 0\) denotes the noise level, and \(E\) is a \(\cY\)-valued random noise variable.

\begin{assumption}[Input distribution]
\label{ass: input distribution}
    The input data \(X\) is drawn from an unknown data distribution which we model as a centered, nondegenerate Gaussian measure \(\mu\) on \(\cX\). We write \(\Sigma = \Sigma_{\mu} := \bbE_{X \sim \mu} [X \otimes X]\) for the covariance operator of $\mu$ and assume for simplicity that \(\tr(\Sigma) = 1\).
\end{assumption}

We remark that all computations can be adapted to nondegenerate Gaussian measures with arbitrary trace, which are not necessarily centered.

\begin{assumption}[Noise model]
\label{assumption : noise model}
We denote the distribution of \(E\) by \(\rho\) and assume that \(E\) is a centered, sub-Gaussian random variable with parameter \(K_{\rho} = 1\) (see~\eqref{eq: sub-Gaussian moment condition}), which is independent of \( X \sim \mu \).
\end{assumption}

Recall that, in general, a random vector \(Z\) in a separable Hilbert space \(\cH\) is sub-Gaussian with parameter \(K > 0\) if \(\nm{Z}_{\cH}\) is sub-Gaussian in the conventional sense, that is,
\begin{equation}
\label{eq: sub-Gaussian moment condition}
    \bbE[ \nm{Z}_{\cH}^p ]^{1/p} \leq K \sqrt{p}, \quad \forall p \geq 1.
\end{equation}
We call the distribution of a sub-Gaussian random variable with parameter \(K\) a sub-Gaussian distribution with parameter \(K\). As an example, the measure \(\mu\) on \(\cX\) is sub-Gaussian. Indeed, by~\cite[eq. (3.5)]{ledoux_ProbabilityBanachSpaces_1991} and \( \tr(\Sigma) = 1 \), we have
\begin{equation}
\label{eq: tail bound for mu}
    \bbP_{X \sim \mu} [\nm{X}_{\cX} > t] 
    = \mu\left( \left\{X \in \cX : \nm{X}_{\cX} > t \right\} \right)
    \leq 4 \exp\left(- \frac{t^2}{8} \right), \quad \forall t \geq 0.
\end{equation}
A standard computation (see e.g. the proof of~\cite[Prop. 2.6.1]{vershynin_HighDimensionalProbability_2025}) shows that this is equivalent to~\eqref{eq: sub-Gaussian moment condition} for some absolute constant \(K_{\mu} > 0\).

\begin{remark}[Covariance operator of sub-Gaussian distribution]
\label{rmk: covariance operator of sub-Gaussian distribution}
    From~\eqref{eq: sub-Gaussian moment condition} it follows in particular that any sub-Gaussian distribution on a separable Hilbert space \( \cH \) has finite second moment. This in turn implies that its uncentered covariance operator \( \Sigma = \bbE[Z \otimes Z] \) is a self-adjoint, positive semi-definite trace class operator. Consequently, by the spectral theorem, there exists an orthonormal basis of \( \overline{\textup{ran}(\Sigma)} \) consisting of eigenvectors of \( \Sigma \) with corresponding eigenvalues \( \lambda_1 \geq \lambda_2 \geq \cdots > 0 \).
\end{remark}

We denote the distribution of $Y = F(X) + \sigma E$ by \( \nu \). It follows directly from the moment condition~\eqref{eq: sub-Gaussian moment condition} 
that \(\nu\) is sub-Gaussian with parameter \(K_{\nu} = \nm{F(0)}_{\cY} + L K_{\mu} + \sigma\). Notice that in the noiseless case with \( \sigma = 0 \) we have \(\nu = F \sharp \mu\). Furthermore,~\eqref{eq: sub-Gaussian moment condition} and Assumption~\ref{assumption : noise model} imply that
\begin{equation} 
\label{eq: expectation of E}
    \bbE[\nm{E}_{\cY}^2] \leq 2,
\end{equation}
which we will repeatedly use in our estimates.

\subsection{Main problem}\label{ss:main-prob}

We now consider the approximation of \(F\). We follow the standard operator learning approach and construct the surrogate operator \(\whF\) from three parts: 
\begin{enumerate}[label=(\roman*)]
    \item an \textit{encoder} for the input space, \( \hEX : \cX \to \bbR^{d_{\cX}} \), \item a \textit{decoder} for the output space, \( \hDY : \bbR^{d_{\cY}} \to \cY \), 
    \item and a \textit{latent space function} \( \whf : \bbR^{d_{\cX}} \to \bbR^{d_{\cY}} \).  
\end{enumerate}
Here, \(d_{\cX}, d_{\cY} \in \bbN\) with \(d_{\cX} \leq \dim(\cX)\), \(d_{\cY} \leq \dim(\cY)\), denote the \emph{encoding and decoding dimension}, respectively. We then define the surrogate operator 
\[
\whF := \hDY \circ \whf \circ \hEX,
\]
as an approximation to \( F \), see Figure~\ref{fig: surrogate operator design}.
For its computation, we require three datasets:
\begin{enumerate}[label=(\alph*)]
    \item To compute \(\hEX\), we require an additional amount of \(N_{\cX} \geq d_{\cX}\) unlabeled data points
    \[
    \whX_i \sim_{\mathrm{i.i.d.}} \mu, \qquad i = 1, \dots, N_{\cX}.
    \]
    \item To compute \(\hDY\), we require \(d_{\cY} \leq N_{\cY} \leq N_{\cX}\) noisy labels 
    \[
    \whY_i = F(\whX_i) + \sigma \whE_i, \quad \whE_i \sim_{\mathrm{i.i.d.}} \rho, \qquad i = 1, \dots, N_{\cY}.
    \]
    We assume that the \(\whE_i\) and \(\whX_j\) are mutually independent so that the \(\whY_i\) are i.i.d. with respect to $\nu$. 
    \item To compute \(\whF\), we generate \(M\) labeled data samples \((X_i, Y_i) \in \cX \times \cY\) with 
    \begin{equation}
    \label{eq: X_i}
        X_i \sim_{\mathrm{i.i.d.}} \mu_{\textup{samp}}, \qquad i = 1, \dots, M,
    \end{equation}
    where \(\mu_{\textup{samp}}\) is a custom sampling measure (depending on the samples \( \whX_i \)), which we will specify later on, and 
    \begin{equation}
    \label{eq: Y_i}
        Y_i = F(X_i) + \sigma E_i, \quad E_i \sim_{\mathrm{i.i.d.}} \rho, \qquad i = 1, \dots, M,
    \end{equation}
    are the corresponding noisy operator evaluations.
    We assume that, conditionally on the PCA data \( \whX_1, \dots, \whX_{N_{\cX}} \) and \( \whE_1, \dots, \whE_{N_{\cY}} \), the \( X_i \) are independent with law \(\mu_{\textup{samp}}\). The \( E_i \) have law~\( \rho \) and are independent of the PCA data and the \( X_i \).
\end{enumerate}
Throughout this paper, we measure the accuracy of $\whF$ in the $L^2_{\mu}(\cX ; \cY)$-norm. We write \(L_{\mu}^2(\cX; \cY)\) for the Lebesgue--Bochner space of (equivalence classes of) strongly measurable operators \(F: \cX \to \cY\) with finite Bochner norm
\[
\nm{F}_{L_{\mu}^2(\cX; \cY)} := \left( \int_{\cX} \nm{F(X)}_{\cY}^2 \D \mu(X) \right)^{1/2}.
\]
\begin{mainproblem} 
    Assuming suitable (e.g., Sobolev) regularity of $F$, determine how large $M$, $N_{\cX}$, and $N_{\cY}$ should be in order to compute a surrogate $\whF$ to a given accuracy in the $L^2_{\mu}(\cX;\cY)$-norm.
\end{mainproblem}


\section{Construction of the approximation}
\label{sec: construction}

We now describe the construction of \( \whF \) based on the three components (i)--(iii) from Section~\ref{ss:main-prob}.

\subsection{Construction of encoders and decoders: PCA and empirical PCA}

We compute $\hEX$ and $\hDY$ using empirical PCA.
To this end, let us first describe the constructions based on (true) PCA. Recall that \(\Sigma = \Sigma_{\mu}\) is the covariance operator of $\mu$. Since \( \mu \) is nondegenerate, \( \Sigma \) is positive definite. We write \(\lambda_1^{\mu} \geq \lambda_2^{\mu} \geq \cdots > 0\) for the PCA eigenvalues and \(\{\phi^{\mu}_i\}_{i = 1}^{\dim(\cX)}\) for corresponding PCA eigenvectors which form an orthonormal basis of \(\cX\), see Remark~\ref{rmk: covariance operator of sub-Gaussian distribution}. To simplify notation, we write \( \lambda_i = \lambda_i^{\mu} \) and \( \phi_i = \phi^{\mu}_i \) if not specified otherwise. Given an encoding dimension~$d_{\cX}$, we define the \textit{true PCA encoder and decoder} for \(\cX\) as
\begin{equation}
\label{eq: true PCA encoder decoder for X}
    \EX(X) := (\ip{X}{\phi_1}_{\cX}, \dots, \ip{X}{\phi_{d_{\cX}}}_{\cX}), \quad \forall X \in \cX, 
    \quad \textup{ and } \quad
    \DX(\bm{x}) 
    := \sum_{i = 1}^{d_{\cX}} x_i \phi_i, \quad \forall \bm{x} \in \bbR^{d_{\cX}}.
\end{equation}
The construction for the output space \(\cY\) is based on the distribution \(\nu\) of the random variable \(Y = F(X) + \sigma E\) with \(X \sim \mu\), \(E \sim \rho\). Its uncentered covariance operator is given by
\[
\Sigma_{\nu} := \bbE[Y \otimes Y],
\]
with eigenvalues \(\lambda_1^{\nu} \geq \lambda_2^{\nu} \geq \dots > 0\) and corresponding eigenvectors \( \{\psi_i\}_{i = 1}^{\dim(\overline{\textup{ran}(\Sigma_{\nu})})} \), which form an orthonormal basis of \( \textup{ran}(\Sigma_{\nu}) \). 
The associated true encoder and decoder mappings are given by
\[
\EY(Y) := (\ip{Y}{\psi_1}_{\cY}, \dots, \ip{Y}{\psi_{d_{\cY}}}_{\cY}), \quad \forall Y \in \cY, 
\quad \textup{ and } \quad
\DY(\bm{y}) := \sum_{i = 1}^{d_{\cY}} y_i \psi_i, \quad \forall \bm{y} \in \bbR^{d_{\cY}}.
\]
If \( \dim(\overline{\textup{ran}(\Sigma_{\nu})}) < d_{\cY} \), we complement \( \{\psi_i\}_{i = 1}^{\dim(\overline{\textup{ran}(\Sigma_{\nu})})} \) in the definition of \( \cE_{\cY}, \cD_{\cY} \) by orthonormal basis vectors \( \psi_i \) of \( \textup{ker}(\Sigma_{\nu}) \).

Note, however, that \(\EX\) and \(\DY\) cannot be used for practical computations, as the measure \(\mu\) (and hence also \(\nu\)) is assumed to be unknown and so are the basis functions \(\phi_i\) and \(\psi_i\). Instead, we use empirical PCA: Given $N_{\cX}$ unlabeled data points $\whX_i$ as in (a) above, we first define the empirical measure and the associated empirical covariance operator by
\[
\hat{\mu} := \frac{1}{N_{\cX}} \sum_{i = 1}^{N_{\cX}} \delta_{\whX_i} 
\quad \textup{ and } \quad
\Sigma_{\hat{\mu}} := \frac{1}{N_{\cX}} \sum_{i = 1}^{N_{\cX}} \whX_i \otimes \whX_i.
\]
The latter operator has eigenvalues \(\lambda_1^{\hat{\mu}} \geq \lambda_2^{\hat{\mu}} \geq \dots \geq 0\), with \(\lambda_i^{\hat{\mu}} = 0\) for all \(N_{\cX} < i \leq \dim(\cX)\), and corresponding eigenvectors \(\{\widehat{\phi}_i\}_{i = 1}^{\dim(\cX)}\) which form an orthonormal basis of \(\cX\). If not specified otherwise, we write \( \widehat{\Sigma} = \Sigma_{\hat{\mu}} \) and \( \hat{\lambda}_i = \lambda_i^{\hat{\mu}} \) in what follows. Given this, we define the \emph{empirical PCA encoder and decoder} for \(\cX\) as
\begin{equation}
\label{eq: empirical PCA encoder decoder for X}
    \hEX(X) 
    := (\ipd{X}{\widehat{\phi}_1}_{\cX}, \dots, \ipd{X}{\widehat{\phi}_{d_{\cX}}}_{\cX}), \quad \forall X \in \cX,
    \quad \textup{ and } \quad
    \hDX(\bm{x}) 
    := \sum_{i = 1}^{d_{\cX}} x_i \widehat{\phi}_i, \quad \forall \bm{x} \in \bbR^{d_{\cX}}.
\end{equation}
The construction for \(\cY\) is analogous: Given $N_{\cY}$ data points $\whY_i$ as in (b) above, we define the empirical (uncentered) covariance operator
\[
\Sigma_{\hat{\nu}} 
:= \frac{1}{N_{\cY}} \sum_{i = 1}^{N_{\cY}} \whY_i \otimes \whY_i,
\]
with eigenvalues \(\lambda_1^{\hat{\nu}} \geq \lambda_2^{\hat{\nu}} \geq \dots \geq 0\), with \(\lambda_i^{\hat{\nu}} = 0\) for all \(N_{\cY} < i \leq \dim(\cY)\), and corresponding eigenvectors \(\{\widehat{\psi}_i\}_{i = 1}^{\dim(\cY)}\), which form an orthonormal basis of \(\cY\).
The empirical PCA encoder and decoder for \(\cY\) are defined as
\[
\hEY(Y) 
:= (\ipd{Y}{\widehat{\psi}_1}_{\cY}, \dots, \ipd{Y}{\widehat{\psi}_{d_{\cY}}}_{\cY}), \quad \forall Y \in \cY,
\quad \textup{ and } \quad
\hDY(\bm{y}) 
:= \sum_{i = 1}^{d_{\cY}} y_i \widehat{\psi}_i, \quad \forall \bm{y} \in \bbR^{d_{\cY}}.
\]

\subsection{Construction of latent space approximation: Least-squares approximation with Hermite polynomials}

Our latent space approximation $\whf$ is based on certain orthogonal polynomials. We now describe this construction, along with the definition of the sampling measure $\mu_{\mathrm{samp}}$.
For \(n \in \bbN_0\), the \(n\)-th normalized (probabilist's) Hermite polynomial on \(\bbR\) is given by
\begin{equation}
\label{eq: Hermite polynomials 1D}
    H_n: \bbR \to \bbR, \hspace{2ex} H_n(x) := \frac{(-1)^n}{\sqrt{n!}} \exp \left( \frac{x^2}{2} \right) \frac{d^n}{d x^n} \exp \left( -\frac{x^2}{2} \right).
\end{equation}
The higher-dimensional Hermite polynomials are defined as products of the one-dimensional ones. Given \(d \in \bbN\), a vector \(\bm{\lambda} = (\lambda_i)^{d}_{i=1}\) with \(\lambda_i > 0\), and a multi-index \(\bm{\gamma} = (\gamma_i)^{d}_{i=1} \in \bbN^d_0 \), we define
\begin{equation}
\label{eq: Hermite polynomials multi-D}
    H_{\bm{\gamma},\bm{\lambda}} : \bbR^d \to \bbR,\quad H_{\bm{\gamma},\bm{\lambda}} (\bm{x}) := \prod_{i = 1}^{d} H_{\gamma_i} \left( \frac{x_i}{\sqrt{\lambda_i}} \right).
\end{equation}
Notice that the family \(\{ H_{\bm{\gamma},\bm{\lambda} } \}_{\bm{\gamma} \in \bbN^d_0 }\) forms an orthonormal basis of \(L^2_{\varrho}(\bbR^d)\), where \( \varrho = \cN(0,\bm{\lambda}) \) is the Gaussian measure with mean zero and diagonal covariance with $i$-th diagonal entry \( \lambda_i \).

Let \(S \subset \bbN^{d_{\cX}}_0\) be a set of size \(|S| = s\) with \( \bm{0} \in S \), which we will specify later. The latent space approximation \(\whf\) has the form
\[
\whf = \sum_{\bm{\gamma} \in S} \bm{c}_{\bm{\gamma}} H_{\bm{\gamma},\bm{\hat{\lambda}}},
\]
where \(\bm{\hat{\lambda}} = (\hat{\lambda}_i)^{d_{\cX}}_{i=1}\) is the vector of empirical PCA eigenvalues and the \(\bm{c}_{\bm{\gamma}} \in \bbR^{d_{\cY}}\) are vector-valued coefficients, which we compute via (weighted) least-squares. With the above notation, define the Gaussian measure \(\hat{\varrho} := \cN(0,\bm{\hat{\lambda}}) \) and the subspace
\begin{equation}
\label{eq: hat-f subspace def}
    \cP_{\bbR^{d_{\cY}}} := \left \{ \sum_{\bm{\gamma} \in S} \bm{c}_{\bm{\gamma}} H_{\bm{\gamma},\bm{\hat{\lambda}}} : \bm{c}_{\bm{\gamma}} \in \bbR^{d_{\cY}} \right \} \subseteq L^2_{\hat{\varrho} } (\bbR^{d_{\cX} } ; \bbR^{d_{\cY}} ) \cap C(\bbR^{d_{\cX} } ; \bbR^{d_{\cY}}).
\end{equation}
The sampling measure on \( \bbR^{d_{\cX}} \) is now defined as
\begin{equation}
\label{eq: upsilon-samp def}
    \D \varrho_{\mathrm{samp}} : = w^{-1} \D \hat{\varrho},
    ~~\text{where}~~
    w(\bm{x}) := \left ( \frac1s \sum_{\bm{\gamma} \in S} H_{\bm{\gamma},\bm{\hat{\lambda}}}(\bm{x})^2 \right )^{-1},\quad \bm{x} \in \bbR^{d_{\cX}},
\end{equation}
and the corresponding sampling measure on \( \cX \) as
\[
\mu_{\mathrm{samp}} := \hDX \sharp \varrho_{\mathrm{samp}}.
\]
Note that \( w > 0 \) because \( \bm{0} \in S \).
Now consider $M$ labeled data samples \( (X_i,Y_i) \) as in (a), given by \eqref{eq: X_i}--\eqref{eq: Y_i}. We then define \( \whf \) as a weighted least-squares fit
\begin{equation}
\label{eq: f-hat def}
    \whf \in \argmin{p \in \cP_{\bbR^{d_{\cY}}}} \frac1M \sum^{M}_{i=1} w(\hEX(X_i)) \nmd{\hEY(Y_i) - p(\hEX(X_i))}^2_2.
\end{equation}

\begin{remark}[Definition of $\mu_{\mathrm{samp}}$]
\label{rmk: sampling measure}
    The sampling measure \( \varrho_{\mathrm{samp}} \) is precisely the \textit{Christoffel sampling} measure for the subspace \(\cP = \spn\{H_{\bm{\gamma}, \bm{\hat{\lambda}}} : \bm{\gamma} \in S \}\). Hence, drawing \( X_1,\ldots,X_M \sim_{\mathrm{i.i.d.}} \mu_{\mathrm{samp}} \) is a type of Christoffel sampling \cite{adcock_OptimalSamplingLeastsquaresApproximation_2025}. Christoffel sampling is a near-optimal random sampling strategy for least-squares approximation in a given subspace.  By this, we ensure that the weighted least-squares fit \( \whf \) is a \textit{quasi-best} approximation from the subspace \(\cP_{\bbR^{d_{\cY}}}\) whenever \(M\) scales log-linearly in \(s = |S|\). Another key aspect of this choice is that we can efficiently draw samples from \(\mu_{\mathrm{samp}}\), as it is a pushforward through the map \( \hDX \) of a measure \(\varrho_{\mathrm{samp}}\) that is an additive mixture of tensor-product probability measures. We discuss this further in Section \ref{subsec: algorithm and practical}.
\end{remark}

In practice, a minimizer of \eqref{eq: f-hat def} can be computed by solving a set of $d_{\cY}$ algebraic least-squares problems of size $M \times s$. Indeed, it is a short exercise to show that
\begin{equation}
\label{eq: L-S approximation rewritten}
    \whf(\bm{x}) =  \left( \sum_{\bm{\gamma} \in S} c^{(k)}_{\bm{\gamma}}  H_{\bm{\gamma}, \bm{\hat{\lambda}} }(\bm{x}) \right )^{d_{\cY}}_{k=1} ,
    \qquad
    \bm{c^{(k)}} : = (c^{(k)}_{\bm{\gamma}} )_{\bm{\gamma} \in S} \in \argmin{\bm{c} \in \bbR^s}{\nmd{\bm{A} \bm{c} - \bm{b^{(k)}}}_2^2}, \qquad k =1, \dots, d_{\cY},
\end{equation}
where the weighted measurement matrix and vectors are given by
\begin{equation}
\label{eq: A, b^k}
    \bm{A} = \left( \sqrt{\frac{w(\bm{x_i})}{M}} H_{\bm{\gamma}, \bm{\hat{\lambda}} }(\bm{x_i}) \right)_{i \in [M],\bm{\gamma} \in S} \in \bbR^{M \times s}, \qquad \bm{b^{(k)}} = \left( \sqrt{\frac{w(\bm{x_i})}{M}} \ip{Y_i}{\widehat{\psi}_k}_{\cY}  \right)_{i \in [M]} \in \bbR^M,
\end{equation}
with \( \bm{x_i} = \hEX(X_i) \) for \( i \in [M] \). Here and throughout the text, we use the notation \([d] = \{1, \dots, d\}\) for \(d \in \bbN \) and \([\infty] = \bbN\).

It remains to specify the multi-index set \( S \). It should be chosen to give as good an approximation as possible from the resulting subspace \( \cP_{\bbR^{d_{\cY}}} \), which naturally depends on the regularity assumptions placed on \( F \). In this work, we assume \( F \) has Sobolev regularity. As we shall show, a suitable choice of \( S \) in this case is defined as follows: 
Suppose that \( \hat{\lambda}_{d_{\cX}} > 0 \) and introduce the weights
\begin{equation}
\label{eq: hat(v)_gamma}
    \hat{v}_{\bm{\gamma}} := \left( 1 + \sum_{i = 1}^{d_{\cX}} \frac{\gamma_i}{\hat{\lambda}_{i}}\right)^{-1/2}, \quad \bm{\gamma} \in \bbN_0^{d_{\cX}},
\end{equation}
and let 
\begin{equation}
\label{eq: hat(tau)}
    \hat{\tau} : \bbN \to \bbN_0^{d_\cX}
\end{equation}
 be a bijection that gives a nonincreasing rearrangement, i.e., \( \hat{v}_{\bm{\hat{\tau}(1)}} \geq  \hat{v}_{\bm{\hat{\tau}(2)}} \geq \cdots > 0 \). Then we set
\begin{equation}
\label{eq: S hat-tau def}
    S = \{ \bm{\hat{\tau}(1)},\ldots,\bm{\hat{\tau}(s)} \}.
\end{equation}
Note that \( \bm{\hat{\tau}(1)} = \bm{0} \).
This specific choice for $S$ comes from the analysis of the approximation error. In fact, later in Section \ref{subsec: optimal rates} we explain that it leads to an approximation error that decays with optimal rates in \( s \) under the stipulated regularity assumptions.

In light of Remark~\ref{rmk: sampling measure}, we call the samples \((X_i, Y_i)\) as in (c) \emph{Christoffel samples}. We further call the (unlabeled) data points \( \whX_i \) as in (a) \emph{PCA samples} and the corresponding labels \(\whY_i\) as in (b) \emph{PCA labels}.

\subsection{Algorithm and practical aspects}\label{subsec: algorithm and practical}

With all the pieces in place, we summarize the computation of \( \whF \) in Algorithm \ref{alg: Hermite-PCA approximation}. A key facet of this algorithm is that the main steps can also be performed numerically. We now discuss these steps in more detail. 

\begin{algorithm}[h]
\caption{Hermite-PCA approximation}
\label{alg: Hermite-PCA approximation}
\begin{algorithmic}[1]
    \Require 
        \Statex \(s \in \bbN \)    \Comment{Approximation dimension}
        \Statex \(d_{\cX}, d_{\cY} \in \bbN\)    \Comment{Encoding and decoding dimensions}
        \Statex \( M \in \bbN \) \Comment{Number of labeled Christoffel samples}
        \Statex \( N_{\cX} \in \bbN,\ N_{\cX} \geq d_{\cX} \) \Comment{Number of unlabeled PCA samples}
        \Statex \(N_{\cY} \in \bbN,\ d_{\cY} \leq N_{\cY} \leq N_{\cX} \)  \Comment{Number of corresponding PCA labels}
    
    \Ensure Surrogate operator \(\whF\) approximating \(F\).
    
    \State Draw \(N_{\cX}\) points \(\whX_1, \dots, \whX_{N_{\cX}} \sim_{\mathrm{i.i.d.}} \mu \).
    \State Compute \(N_{\cY}\) noisy operator evaluations \(\whY_i = F(\whX_i) + \sigma \whE_i\), \(i \in [N_{\cY}]\).
    \State Compute \(\hEX,\hDX\) and \(\hDY,\hEY\). 
        \Statex This requires the computation of \(\hat{\lambda}_1, \dots, \hat{\lambda}_{d_{\cX}}\), \(\widehat{\phi}_1, \dots, \widehat{\phi}_{d_{\cX}}\), and \(\hat{\lambda}_1^{\nu}, \dots, \hat{\lambda}_{d_{\cY}}^{\nu}\), \(\widehat{\psi}_1, \dots, \widehat{\psi}_{d_{\cY}}\).
    \State Compute the multi-index set \( S \) defined in \eqref{eq: S hat-tau def} corresponding to the bijection \( \hat{\tau} \) in \eqref{eq: hat(tau)}.
    \State Draw \(M\) points \(\bm{x_1}, \dots, \bm{x_M} \sim_{\mathrm{i.i.d.}} \varrho_{\mathrm{samp}}\), with \(\varrho_{\mathrm{samp}}\) as in \eqref{eq: upsilon-samp def}, and set \( X_i = \hDX(\bm{x_i}) \) for \( i \in [M] \).
    \State Compute \(M\) noisy operator evaluations \(Y_i = F(X_i) + \sigma E_i\), \(i \in [M]\).
    \State Compute \(\whf\) via~\eqref{eq: L-S approximation rewritten}--\eqref{eq: A, b^k} and set \( \whF = \hDY \circ \whf \circ \hEX\).
\end{algorithmic}
\end{algorithm}

\textbf{Step 1} assumes we are given \( N_{\cX} \) unlabeled samples from the unknown measure \( \mu \). This is a reasonable assumption in practice. 

In \textbf{Step 2}, we generate \( N_{\cY} \leq N_{\cX} \) noisy evaluations of the target operator \( F \). 

In \textbf{Step 3}, we use the results of Steps 1 and 2 to compute the empirical encoders and decoders. This is done by forming the empirical covariance operators \( \Sigma_{\hat{\mu}} \) and \( \Sigma_{\hat{\nu}} \) and computing, respectively, their \( d_{\cX} \) and \( d_{\cY} \) largest eigenvalues and corresponding orthonormal eigenvectors. These are operators on the Hilbert spaces \( \cX \) and \( \cY \), respectively. In practice, these spaces are normally discretized first using, for example, polynomials or wavelets. The computation involves computing the largest \( d_{\cX} \) and \( d_{\cY} \) eigenvalues and corresponding orthonormal eigenvectors of matrices of sizes \( D_{\cX} \times D_{\cX} \) and \( D_{\cY} \times D_{\cY} \), respectively, where \(D_{\cX},D_{\cY} \in \bbN\) are the numbers of degrees of freedom in the discretizations.

\textbf{Step 4} requires the computation of the nonincreasing rearrangement \( \hat{\tau} \) in \eqref{eq: hat(tau)}. This can be done efficiently by using a Dijkstra-like algorithm on the weighted lattice \( \bbN_0^{d_{\cX}} \), where all edges in direction \(i \in [d_{\cX}] \) carry the weight \( 1 / \hat{\lambda}_i \), to iteratively select those nodes which can be reached from the origin with the smallest costs.

\textbf{Step 5} involves sampling \( M \) points i.i.d.\ from \( \varrho_{\mathrm{samp}} \). This measure can be expressed as an additive mixture (with weights \( 1/s \)) of the probability measures 
\[
H_{\bm{\gamma},\bm{\hat{\lambda}}}(\bm{x})^2 \D \hat{\varrho}(\bm{x}) = \bigotimes^{d_{\cX}}_{i=1} H_{\gamma_i} \Bigg ( \frac{x_i}{\sqrt{\hat{\lambda}_i} } \Bigg )^2 \frac{1}{\sqrt{2 \pi \hat{\lambda}_i}} e^{-x^2_i / (2 \hat{\lambda}_i)} \D x_i,
\]
which are tensor products of one-dimensional probability measures.
Thus, to sample  \( \bm{x} \sim \varrho_{\mathrm{samp}} \), one proceeds as follows: First, draw an index \( \bm{\gamma} \in S \) uniformly at random from all possible \( |S| = s \) indices. Then, draw \( \bm{x} \sim H_{\bm{\gamma},\bm{\hat{\lambda}}}(\bm{x})^2 \D \hat{\varrho}(\bm{x}) \) by drawing its components \( x_i \) independently from the corresponding univariate measures. Note that this latter step can be efficiently carried out via inverse transform sampling \cite{narayan_ComputationInducedOrthogonalPolynomial_2018}, for example.

Having done this, \textbf{Step 6} just involves sampling \( F \) at the sample points \( X_1,\ldots, X_M \). 

Finally, in \textbf{Step 7} we compute \( \whf \) by solving the algebraic least-squares problems \eqref{eq: L-S approximation rewritten}, \eqref{eq: A, b^k}. These can be solved efficiently, e.g., via conjugate gradients (notice that our theoretical results also guarantee that \( \bm{A} \) is well-conditioned) in \( \mathcal{O}(s M d_{\cY}) \) flops per iteration.


\section{Hermite-PCA approximation of Sobolev operators}
\label{sec: Hermite-PCA approximation}

We now state and discuss our main result on the approximation of Sobolev operators via the Hermite-PCA method, Algorithm \ref{alg: Hermite-PCA approximation}.

\subsection{Main result}

First, we need several further concepts. Let
\(
\Gamma := \{ \bm{\gamma} \in \bbN^{\bbN}_0 : |\mathrm{supp}(\bm{\gamma}) |< \infty \}
\)
be the set of infinite multi-indices with finite support, where \( \mathrm{supp}(\bm{\gamma}) = \{ i \in \bbN : \gamma_i \neq 0 \} \). Define the weights

\begin{equation}
\label{eq: v_gamma}
    v_{\bm{\gamma}} := \left( 1 + \sum_{i = 1}^{\dim(\cX)} \frac{\gamma_i}{\lambda_{i}}\right)^{-1/2}, \quad \bm{\gamma} \in \bbN_0^{\dim(\cX)},
\end{equation}
and let 
\begin{equation}
\label{eq: tau}
    \tau : \bbN \to \bbN_0^{\dim(\cX)}
\end{equation}
be a bijection that gives a nonincreasing rearrangement, i.e., \( v_{\bm{\tau(1)}} \geq  v_{\bm{\tau(2)}} \geq \cdots > 0\). If \mbox{\( \dim(\cX) = \infty \)}, we replace \( \bbN_0^{\dim(\cX)} \) by \( \Gamma \). As we will see in a moment, this rearranged sequence determines the error due to approximating \( F \) in a finite-dimensional subspace of Hermite polynomials. Note that \( \tau \) is well-defined since \( \lambda_i \to 0 \), which implies that all the superlevel sets \( \{ \bm{\gamma} : v_{\bm{\gamma}} > \varepsilon \} \), \( \varepsilon > 0 \), are~finite.

\begin{theorem}[Error bound for Hermite-PCA approximation of Sobolev operators]
\label{thm: Hermite-PCA Sobolev error bound}
    There exist constants \( c_1,c_2,c_3 \geq 1 \) such that the following holds:
    Let \( s \geq \binom{d_{\cX} + 2}{2} + 1 \) and \(0 <  \epsilon < 1\) be fixed. Suppose that 
    \begin{equation}
    \label{eq: main condition for N_X}
        N_{\cX} \geq c_1 \max \left \{ d^2_{\cX} (\lambda_{d_{\cX}})^{-2} ,  (\lambda_{d_{\cX}})^{-4} (s \log(s) + |\log(N_{\cX})| + s | \log(\lambda_{d_{\cX}}) | )^2 \right \}  \log(12/\epsilon), 
    \end{equation}
    \begin{equation}
    \label{eq: main condition for N_Y}
        N_{\cX} \geq N_{\cY} \geq c_2 d_{\cY} \log(12/\epsilon),
    \end{equation}
    \begin{equation}
    \label{eq: main condition for M}
        M \geq c_3 s \log(12 s / \epsilon).
    \end{equation}
    Let \( F \) be \( L \)-Lipschitz and \( k \in \bbN \). 
    Then, with probability at least \(1 - \epsilon\) in the draw of \( (\whX_i)_{i \in [N_{\cX}]} \), \( (\whE_i)_{ i \in [N_{\cY}] } \), and, conditionally on the PCA data, \( (X_i, E_i)_{i \in [M]} \), the surrogate \(\whF\) defined by Algorithm \ref{alg: Hermite-PCA approximation} uniquely exists and satisfies
    \begin{equation*}
        \begin{split}
            \nm{F - \whF}_{L_{\mu}^2(\cX; \cY)}
            & \lesssim L \sqrt{\sum^{\dim(\cX)}_{i = d_{\cX}+1} \lambda_i} + \sqrt{\sum^{\dim(\cY)}_{i = d_{\cY}+1} \lambda^{F \sharp \mu}_i}
            \\
            & \quad + (\nm{F(0)}_{\cY} + L (1+K_{\mu}) + \sigma ) \left [ \left ( \frac{d_{\cX} \log(12/\epsilon)}{N_{\cX}} \right )^{1/4} + \left ( \frac{d_{\cY} \log(12/\epsilon)}{N_{\cY}} \right )^{1/4} \right ]
            \\
            & \quad + \frac{1}{\sqrt{\epsilon}} \bigg(6 + 6 k \sum_{i = 1}^{d_{\cX}} \lambda_i^{-1} \bigg)^{k/2} v^{k}_{\bm{\tau(s+1)}} \nm{ \hEY \circ F \circ \hDX }_{H_{\hat{\varrho}}^{k}(\bbR^{d_{\cX}}; \bbR^{d_{\cY}})} \\
            & \quad + \left( \sqrt{\frac{s}{\epsilon M} } + 1 \right) \sigma.
        \end{split}
    \end{equation*}
\end{theorem}
Note that~\eqref{eq: main condition for N_X} implies the required condition \( N_{\cX} \geq d_{\cX} \) for the number of PCA samples because~\( \lambda_{d_{\cX}} \leq 1 \).

\subsection{Discussion of the error bound}
\label{subsec: discussion of the error bound}

Theorem \ref{thm: Hermite-PCA Sobolev error bound} decomposes the total Hermite-PCA approximation error $\nmd{F - \whF}_{L_{\mu}^2(\cX; \cY)}$ into four main terms
\begin{equation}
\label{eq: error decomposition}
    \nm{F - \whF}_{L_{\mu}^2(\cX; \cY)} \lesssim \mathrm{Err}_{\textup{true-PCA}} + \mathrm{Err}_{\textup{emp-PCA}}  + \mathrm{Err}_{\textup{approx}} +\mathrm{Err}_{\textup{noise}}.
\end{equation}
These are a \textit{true PCA projection error}
\begin{equation}
\label{eq: true PCA projection error}
    \mathrm{Err}_{\textup{true-PCA}} : = L \sqrt{\sum^{\dim(\cX)}_{i = d_{\cX}+1} \lambda_i} + \sqrt{\sum^{\dim(\cY)}_{i = d_{\cY}+1} \lambda^{F \sharp \mu}_i},
\end{equation}
an \textit{empirical PCA projection error}
\begin{equation}
\label{eq: empirical PCA projection error}
    \mathrm{Err}_{\textup{emp-PCA}} : = (\nm{F(0)}_{\cY} + L (1+K_{\mu}) + \sigma )  \left [ \left ( \frac{d_{\cX} \log(12/\epsilon)}{N_{\cX}} \right )^{1/4} + \left ( \frac{d_{\cY} \log(12/\epsilon)}{N_{\cY}} \right )^{1/4} \right ],
\end{equation}
an \textit{approximation error} (in latent space)
\begin{equation}
\label{eq: approx error}
    \mathrm{Err}_{\textup{approx}} : = \frac{1}{\sqrt{\epsilon}} \bigg(6 + 6 k \sum_{i = 1}^{d_{\cX}} \lambda_i^{-1} \bigg)^{k/2} v^{k}_{\bm{\tau(s+1)}} \nm{ \hEY \circ F \circ \hDX }_{H_{\hat{\varrho}}^{k}(\bbR^{d_{\cX}}; \bbR^{d_{\cY}})},
\end{equation}
and a \textit{noise error}
\begin{equation}
\label{eq: noise error}
    \mathrm{Err}_{\textup{noise}} : = \left( \sqrt{\frac{s}{\epsilon M}} + 1 \right) \sigma.
\end{equation}

\paragraph{The term \( \mathrm{Err}_{\textup{true-PCA}} \).} 
This is the error due to the truncation \( d_{\cX} \leq \mathrm{dim}(\cX) \) and \( d_{\cY} \leq \mathrm{dim}(\cY) \) incurred by considering finitely many PCA terms. As expected from standard results \cite{lanthaler_ErrorEstimatesDeepONetsDeep_2022}, it behaves like the square root of the sum of the omitted (uncentered) PCA eigenvalues. 
Note that the eigenvalues \(\lambda_i^{F \sharp \mu}\) might exhibit very slow decay. However, if we know that \(F\) maps to some function space of higher regularity (e.g., \(H^r\) with \(r > 0\)), then it can be shown that the PCA truncation error decays at least polynomially in \( d_{\cY} \), see~\cite[Prop. 15]{lanthaler_OperatorLearningPCANetUpper_2023}.

\paragraph{The term \( \mathrm{Err}_{\textup{emp-PCA}}\).} 

This term accounts for computing the PCA eigenvalues and eigenfunctions from samples and scales like \( (N_{\cX})^{-1/4} \) and \( (N_{\cY})^{-1/4} \). Recall that \( N_{\cX} \) is the number of unlabeled samples \( \whX_i \sim \mu \) used to compute the PCA basis of \( \cX \) and \( N_{\cY} \) is the number of labels \( \whY_i = F(\whX_i) + \sigma \whE_i \) used to compute the PCA basis of \( \cY \). This scaling is expected from general bounds for empirical PCA, see~\cite{lanthaler_OperatorLearningPCANetUpper_2023}.  

Notice that \( N_{\cX} \) and \( N_{\cY} \) must also satisfy \eqref{eq: main condition for N_X} and \eqref{eq: main condition for N_Y}, respectively. The latter is quite reasonable, as it states that \( N_{\cY} \) scales linearly in \( d_{\cY} \), the dimension of the PCA truncation. The former condition is more stringent. In particular, it suggests a scaling of \( N_{\cX} \) in \( s \) that is at least quadratic. However, in practice (see Section~\ref{sec: numerical}) a log-linear scaling appears to suffice. We consequently believe the quadratic scaling is an artifact of our analysis. Improving it is a topic for future work.

\paragraph{The term \( \mathrm{Err}_{\textup{approx}} \).} 

This term is controlled by the parameter \( s \) and the weight \( v_{\bm{\tau(s+1)}} \), the former being related log-linearly to \( M \) via \eqref{eq: main condition for M}. Put another way, with log-linear oversampling of \( s \) in the number of labeled samples, one achieves an approximation rate specified by the behavior of \( v_{\bm{\tau(s+1)}} \) as \( s \rightarrow \infty \).

The precise decay rate of \( v_{\bm{\tau(s+1)}} \) is determined by the decay of the PCA eigenvalues \( \lambda_i \). We elaborate more on this behavior in the next section. A key point is that the approximation error behaves like \( v_{\bm{\tau(s+1)}} \) raised to the power \( k \), where \( k \) denotes the Sobolev regularity. This holds for any~\( k \), thereby demonstrating the spectral approximation properties of Hermite-PCA approximation. Notice in particular that Algorithm \ref{alg: Hermite-PCA approximation} is independent of \( k \).

The sequence \( v_{\bm{\gamma}} \) and rearrangement \( \tau \) are closely related to the corresponding empirical quantities \( \hat{v}_{\bm{\gamma}} \) and \( \hat{\tau} \) defined in \eqref{eq: hat(v)_gamma} and \eqref{eq: hat(tau)} that are used to construct the approximation \( \whF \) via \eqref{eq: S hat-tau def}. One can consider \( \hat{v}_{\bm{\gamma}} \) and \( \hat{\tau} \) as approximations to the `true' terms  \( v_{\bm{\gamma}} \) and \( \tau \) stemming from, firstly, the dimension truncation \( d_{\cX} \leq \mathrm{dim}(\cX) \) and, secondly, the approximate computation of the true PCA eigenvalues \( \lambda_i \) via the empirical PCA eigenvalues \( \hat{\lambda}_i \). In particular, these quantities coincide when \( d_{\cX} = \mathrm{dim}(\cX) \) and \( \hat{\lambda}_i = \lambda_i \) for all \(i\). Crucially, Theorem \ref{thm: Hermite-PCA Sobolev error bound} estimates the error in terms of the `true' weight \( v_{\bm{\tau(s+1)}} \), with an additional additive term that accounts for the empirical PCA error. As a result, our method achieves near-optimal approximation rates for any level of empirical PCA error. We discuss the near-optimality of these rates further in Section~\ref{subsec: optimal rates} below.

\paragraph{The term \( \mathrm{Err}_{\textup{noise}} \).} This term scales like the noise level \( \sigma \) multiplied by the factor \( \sqrt{s/(\epsilon M)} + 1 \). In particular, it tends to \( \sigma \) as \( M \rightarrow \infty \) for fixed \( s \), thereby demonstrating the denoising properties of Hermite-PCA approximation in the setting of statistical noise considered in this work.
However, the error bound generally cannot avoid an \( \cO(\sigma) \) term, no matter how many samples are drawn. Indeed, this term results from bounding the true PCA projection error in \(\cY\). The latter is independent of the number of drawn samples but depends on the measure \( \nu \), which carries noise, see Lemma~\ref{lem: true PCA decoding, noisy}. For other results addressing statistical noise in operator learning, see \cite{adcock_SharpMinimaxRiskBounds_2026,reinhardt_StatisticalLearningTheoryNeural_2026}.

\subsection{Optimal approximation of Sobolev operators}
\label{subsec: optimal rates}

As noted, in the absence of noise and PCA projection errors, the error \( F - \whF \) is determined by the quantity \( v^{k}_{\bm{\tau(s+1)}} \), which is itself controlled by the Sobolev regularity \( k \) and the true PCA eigenvalues~\( \lambda_i \). It transpires that this quantity is fundamental: no \textit{algorithm} that uses \( s \) linear measurements (which may or may not be pointwise evaluations) can achieve a faster rate of approximation in the \( L^2_{\mu} \)-norm uniformly for all Sobolev operators. More specifically, \( v^{k}_{\bm{\tau(s+1)}} \) is a lower bound (up to \(s\)-independent constants) for the \emph{(adaptive) \( s \)-width} of the class of bounded, continuous \( H^k_{\mu}(\cX ; \cY) \)-operators. We present this result in Appendix~\ref{app: lower bounds}. It follows from adapting arguments in \cite{adcock_SampleComplexityLearningLipschitz_2026}, which considered the case \( k = 1 \) only. We stress that requiring boundedness of the operators is one way to infer lower recovery bounds from pointwise information, see Remark~\ref{rmk: boundedness of F}.

A key consequence is that Hermite-PCA approximation achieves the optimal worst-case rates for learning Sobolev operators from \( M \) samples, up to the log-linear oversampling amount \eqref{eq: main condition for M}. This also shows that pointwise samples, sampled randomly from an appropriate distribution, constitute \textit{near-optimal information} for this problem with high probability. Moreover, as noted above, Hermite-PCA achieves this in a \emph{spectral} fashion: it attains these near-optimal rates for \emph{any} \( k \) and without any knowledge of \( k \).

\subsection{Concrete rates of approximation}
\label{subsec: concrete rates of approximation}

In applications, for example, in the context of functional PCA, the eigenvalues \( \lambda_i \) are typically assumed to exhibit algebraic or exponential decay, see e.g.~\cite{milbradt_HighprobabilityBoundsReconstructionError_2020,reiss_NonasymptoticUpperBoundsReconstruction_2020} and references therein. 
The following bounds~\eqref{eq: error decay; alg EVs},~\eqref{eq: error decay; exp EVs} follow from~\cite[Thm. 3.6]{adcock_SampleComplexityLearningLipschitz_2026} in the case of infinite-dimensional domains. By Lemma~\ref{lem: weight compatibility across dimensions}, they hold for finite-dimensional domains as well.

If there exists a constant \( c > 0 \) such that \( \lambda_i \leq c \, i^{- \alpha} \) for every \( i \in \bbN \) and \( \alpha > 1 \), then we have for every \( \varepsilon > 0 \)
\begin{equation}
\label{eq: error decay; alg EVs}
    v_{\bm{\tau(s+1)}} \lesssim \log(s)^{-\frac{\alpha}{2} + \varepsilon}, \quad s \to \infty.
\end{equation}
If there exists a constant \( c > 0 \) such that \( \lambda_i \leq c \, e^{-\alpha i^{\beta}} \) for every \( i \in \bbN \) and \( \alpha, \beta > 0 \), then we have for every \( \varepsilon > 0 \)
\begin{equation}
\label{eq: error decay; exp EVs}
    v_{\bm{\tau(s+1)}} \lesssim e^{(-\frac{1}{2} \alpha^{\frac{1}{\beta + 1}} (\beta')^{1 / \beta'} + \varepsilon) \log(s)^{1 / \beta'}},  \quad s \to \infty \quad \text{ with } \quad \beta' := 1 + \frac{1}{\beta}.
\end{equation}
These bounds yield estimates for the approximation error in terms of the approximation dimension~\( s \). 
More concretely, if \( \lambda_i \leq c \, i^{-\alpha} \), we have with high probability
\begin{equation}
\label{eq: approx error; alg EVs}
    \mathrm{Err}_{\mathrm{approx}} \lesssim_{F,\hDX,\hEY,k} \log(s)^{-\frac{k \alpha}{2} + \varepsilon}, \quad s \to \infty.
\end{equation}
The underlying logarithmic decay is imposed by the algebraic decay of the eigenvalues, while the decay rate \( \alpha \) of the eigenvalues and the Sobolev regularity \( k \) of the target operator determine the precise exponent: faster decay and/or higher Sobolev regularity ensure faster convergence. Similarly, if \( \lambda_i \leq c \, e^{-\alpha i^{\beta}} \), we have with high probability
\begin{equation}
\label{eq: approx error; exp EVs}
    \mathrm{Err}_{\mathrm{approx}} \lesssim_{F,\hDX,\hEY,k} e^{(-\frac{k}{2} \alpha^{\frac{1}{\beta + 1}} (\beta')^{1 / \beta'} + \varepsilon) \log(s)^{1 / \beta'}}, \quad s \to \infty.
\end{equation}
Note that in both cases, the decay of the upper bounds is only subalgebraic in \( s \) --- and via the log-linear scaling~\eqref{eq: main condition for M} also subalgebraic in \(M\). In fact, this phenomenon is independent of the decay of the \( \lambda_i \) and constitutes an intrinsic~\emph{curse of sample complexity} on infinite-dimensional domains: Regardless of the decay rate of the PCA eigenvalues \( \lambda_i \), the quantity \( v_{\bm{\tau(s+1)}} \) can only decay subalgebraically as \(s \to \infty\) (and hence \( M \to \infty \)). For further details, we refer to~\cite{adcock_SampleComplexityLearningLipschitz_2026}.

\subsection{The regularity condition}

In order for the error bound in Theorem \ref{thm: Hermite-PCA Sobolev error bound} to be non-vacuous, we require the latent Sobolev regularity  \( \hEY \circ F \circ \hDX \in H^k_{\hat{\varrho}}(\bbR^{d_{\cX}} ; \bbR^{d_{\cY}} ) \). This may be less desirable, as it pertains to the latent space function \( \hEY \circ F \circ \hDX  \) and the measure \( \hat{\varrho} = \cN(0,\bm{\hat{\lambda}}) \) is defined in terms of the empirical PCA eigenvalues $\hat{\lambda}_i$, as opposed to their true counterparts $\lambda_i$.
Ideally, one would consider Sobolev regularity \( F \in H^k_{\mu}(\cX ; \cY) \) of the underlying operator \( F \) with respect to the underlying Gaussian measure \( \mu \). However, this is arguably not straightforward within this framework, where regularity of the encoded function appears critically important to address the errors stemming from the approximation of the true PCA eigenvalues and eigenfunctions. It is a short argument to show that \( F \in H^k_{\mu}(\cX ; \cY) \not\Rightarrow \hEY \circ F \circ \hDX \in H^k_{\hat{\varrho}}(\bbR^{d_{\cX}} ; \bbR^{d_{\cY}} ) \), in general. 

This raises the question of what kinds of regularity assumptions on \( F \) imply that \( \hEY \circ F \circ \hDX \in H^k_{\hat{\varrho}}(\bbR^{d_{\cX}} ; \bbR^{d_{\cY}} ) \). Fortunately, assuming a small amount of additional regularity is sufficient. In Appendix~\ref{app: regularity} we show that elements of the class \( C^k_{(\mu,c)\text{-adm}}(\cX ; \cY) \), consisting of \( C^k \)-operators whose derivatives do not grow too fast at infinity, automatically satisfy \( \hEY \circ F \circ \hDX \in H^k_{\hat{\varrho}}(\bbR^{d_{\cX}} ; \bbR^{d_{\cY}} ) \) with high probability. As a special case, this includes the class \( C^k_{\mathrm{Lip}}(\cX ; \cY) \) of \( C^k \)-operators whose derivatives are Lipschitz continuous.

On the other hand, it remains an open problem to determine whether \( H^k_{\mu}(\cX ; \cY) \)-regularity is sufficient to achieve optimal approximation rates in the practical setting studied in this paper: namely, where \( \mu \) is Gaussian but its covariance structure is unknown and therefore can only be estimated from samples.

\subsection{Further discussion}

Several other aspects of Theorem \ref{thm: Hermite-PCA Sobolev error bound} warrant further discussion. First, notice that the result requires \( F \) to be \( L \)-Lipschitz. This is a standard assumption used in encoding-decoding approaches to operator learning to bound errors stemming from approximate encoders and decoders. Second, Theorem \ref{thm: Hermite-PCA Sobolev error bound} presents an error bound holding with high probability. One could also derive error bounds in expectation using standard techniques from the analysis of least-squares approximation from random samples, see \cite{adcock_OptimalSamplingLeastsquaresApproximation_2025}. Third, notice that the
approximation and noise error terms scale like \( 1/\sqrt{\epsilon} \) in the failure probability \( \epsilon \). This stems from the application of Markov's inequality in the analysis of the least-squares problem \eqref{eq: f-hat def}. While there are ways to mitigate this scaling \cite{adcock_OptimalSamplingLeastsquaresApproximation_2025}, these generally require faster, algebraic and/or \( L^{\infty} \)-norm convergence of best approximations in \( \cP_{\bbR^{d_{\cY}}} \) to the target operator. Neither of these generally holds in the case of Sobolev operators, as elaborated previously. Fourth, we mention that there are ways to reduce the log-linear \(s\)-scaling of \(M\) to a linear scaling~\cite{adcock_OptimalSamplingLeastsquaresApproximation_2025}. Incorporating them into the Hermite-PCA method is left to future work.


\section{Numerical experiments}
\label{sec: numerical}

We present two different experiments which illustrate the validity of the theoretical error bound in Theorem~\ref{thm: Hermite-PCA Sobolev error bound} in practice. In the first experiment, we consider an obstacle problem on the one-dimensional unit interval. More specifically, we approximate the obstacle-to-solution operator associated with minimizing the Dirichlet energy of functions under a unilateral side constraint. We show how each of the error terms~\eqref{eq: true PCA projection error}--\eqref{eq: noise error} directly influences the overall approximation error via~\eqref{eq: error decomposition}, empirically validating our theoretical findings. In the second experiment, we demonstrate the spectral property of our algorithm by approximating Gaussian Sobolev functionals of varying regularity.

In each experiment, we use \( M = \lceil s \log(s) \rceil \) labeled training samples \((X_i, Y_i) \in \cX \times \cY\) as in~\eqref{eq: X_i}, \eqref{eq: Y_i} for the least-squares fit, and \( 1000 \) unseen labeled test samples to obtain a Monte Carlo estimate for the relative test error
\(
\nmd{ F - \whF }_{L_{\mu}^2(\cX; \cY)} / \nm{ F }_{L_{\mu}^2(\cX; \cY)}.
\)
We run each experiment \( 10 \) times with different random seeds, but the same test set, and plot the mean approximation error with its one-standard-deviation band (as shaded region) on a log-log scale.
In each experiment, we set \( \epsilon = 0.1 \) and \( \delta = 0.5 \), and the input space has dimension~\( \dim(\cX) = 10 \). With a \(d\)-dimensional input space, it is known that \( v_{\bm{\tau(s + 1)}} \asymp s^{-1/(2 d)} \), see~\cite[Thm. 3.5]{nguyen_WidthsEmbeddingsGaussianSobolev_2026}. Although this rate is algebraic in \( s \), it falls below the subalgebraic rates described in Section~\ref{subsec: concrete rates of approximation} only if \( s \) is very large (e.g. with \(\lambda_i \asymp i^{-2} \), we have \( s^{-1/20} < \log(s)^{-1} \) only if \( s \gtrsim 10^{39} \), cf.~\eqref{eq: error decay; alg EVs}). In our experiments, we thus expect the approximation error to be governed by the latter subalgebraic rates.

\subsection{Obstacle problem in 1D}
\label{subsec: obstacle problem}

We consider the problem of minimizing the Dirichlet energy functional \( I(u) := \frac{1}{2} \int_0^1 \absd{u'(x)}^2 \D x \) among all functions \( u \) which belong to the set \( \cA_v := \{ w \in H_0^1(0,1) : w \geq v \text{ a.e. in } (0,1) \} \), where \( v \in H_0^1(0,1) \) is a given obstacle function. Equivalently, we seek the solution \( u \in \cA_v \) of the variational inequality
\[
\int_0^1 u' (w' - u') \D x \geq 0, \quad \forall w \in \cA_v.
\]
It is easy to see that such a solution exists and is unique. Hence, the obstacle-to-solution operator
\[
F : H_0^1(0,1) \to H_0^1(0,1), \quad v \mapsto F(v) = u,
\]
is well-defined. It is known from the regularity theory of obstacle problems that \( F \) is Lipschitz continuous, but is in general not Fréchet differentiable, see~\cite{doktor_PerturbationsVariationalInequalitiesRate_1980,mignot_ControleDansInequationsVariationelles_1976}. Since \( \hEY, \hDX \) are Lipschitz continuous as well, it follows by Rademacher's theorem that the latent space function \( \hEY \circ F \circ \hDX \) belongs to \( H_{\hat{\varrho}}^1(\bbR^{d_{\cX}}; \bbR^{d_{\cY}}) \). In general, however, we cannot expect higher-order Sobolev regularity. 

Let us comment on the approximability of \( F \) in the context of optimal recovery, discussed in Section~\ref{subsec: optimal rates}. Note that \( F \) is unbounded. Capping \( F \) at a fixed threshold value \( T > 0 \) yields a mapping 
\[
\wtF(X) := 
\begin{cases}
    F(X) & \textup{ if } \nm{F(X)}_{\cY} \leq T, \\
    T \frac{F(X)}{\nm{F(X)}_{\cY}} &  \textup{ if } \nm{F(X)}_{\cY} > T,
\end{cases}
\]
which is bounded and Lipschitz continuous and therefore lies in \( H^1_{\mu} \cap C_b \), see~\cite[Thm. 2.11]{adcock_SampleComplexityLearningLipschitz_2026}\footnote{Strictly speaking, since the output space \( H^1_0(0,1) \) is infinite-dimensional, \( \wtF \) lies in a weighted Gaussian Sobolev space \( H^1_{\mu, \bm{b}}\) for some weight sequence \( \bm{b} \). We omit the details and refer the reader to~\cite{adcock_SampleComplexityLearningLipschitz_2026} for further information.}. The \( H^1_{\mu} \)-unit ball of this class satisfies the lower bound in Theorem~\ref{thm: lower width bound} subject to pointwise sampling, see Remark~\ref{rmk: boundedness of F}. As we evaluate \( F \) only finitely many times in our experiments, we may replace \( F \) by \( \wtF \) in them, by taking \( T \) sufficiently large so that it exceeds all observed values of \( \nm{F(X)}_{\cY} \). As such, we can generally expect that the approximation of \( F \) (resp. \( \wtF \)) via the Hermite-PCA approximation method yields convergence rates as predicted by Theorem~\ref{thm: lower width bound} and Proposition~\ref{prop: lower weight bound}.

Next, we describe our experimental setup. We equip \( H^1_0(0,1) \) with the inner product \( \ip{u}{v}_{H^1_0} := \int_0^1 u' v' \D x \). As input and output spaces, we choose the span of \( 10 \) orthonormal sine basis functions in \( H_0^1(0,1) \), given by \( \frac{\sqrt{2}}{i \pi} \sin(i \pi x) \), \( i \in [10] \), which we discretize by \(127\) piecewise linear finite element functions. We compute the solution to the obstacle problem (projected onto the \(10\)-dimensional output space) by applying the projected Gauss--Seidel method on a mesh with \( 257 \) nodes. 
We take \( \lambda_i \asymp i^{-2} \) and normalize such that \( \sum_{i=1}^{10} \lambda_i = 1 \). 

In all experiments in this subsection, we construct the empirical PCA encoder and decoder using \( N_{\cX} = \lceil d_{\cX} s \log(s) \rceil \) samples from \( \mu \) and \( N_{\cY} = \min\{ \lceil d_{\cY} \log(s)^4  \rceil, N_{\cX} \}\) corresponding labels from \( F \sharp \mu\). \\

\noindent
\textbf{Approximation and noise errors.} \quad We set \( d_{\cX} = d_{\cY} = 10 \). In this case, \( \hDX \circ \hEX = \cI_{\cX} \) and \( \hDY \circ \hEY = \cI_{\cY} \) are the identity operators on \( \cX\) and \(\cY\), respectively. As such, there are no (true or empirical) PCA projection errors, see~\eqref{eq: encoding error: general},~\eqref{eq: decoding error: general}. In particular, there is also no additive \( \cO(\sigma) \) noise term, which arises in bounding the true PCA projection error under noise, see Lemma~\ref{lem: true PCA decoding, noisy}. Consequently, the total Hermite-PCA approximation error is controlled by \( \mathrm{Err}_{\textup{approx}} + \widetilde{\mathrm{Err}}_{\textup{noise}} \) with \( \widetilde{\mathrm{Err}}_{\textup{noise}} = \sigma \sqrt{s / (\epsilon M)} \).

We first set \( \sigma = 0 \) so that \( \widetilde{\mathrm{Err}}_{\textup{noise}} = 0 \). By~\eqref{eq: approx error; alg EVs} we expect the error term \( \mathrm{Err}_{\textup{approx}} \) to be of order \( \cO(\log(s)^{-1}) \). This is consistent with Figure~\ref{fig:approx}. In particular, this empirically confirms our hypothesis that the at-least-quadratic scaling of \( N_{\cX} \) in \( s \), as suggested by~\eqref{eq: main condition for N_X}, is pessimistic in practice --- a log-linear scaling suffices.

Next, we consider \( \sigma > 0 \). Our choice of \( M \) yields \( \widetilde{\mathrm{Err}}_{\textup{noise}} \asymp \log(s)^{-1/2} \). The noise error thus dominates \( \mathrm{Err}_{\textup{approx}} \). This is consistent with Figure~\ref{fig:approx-noise}, where we set \( \sigma = 5 \). As the noise distribution, we choose \( \rho = \cN(0, \bm{C}) \) with a random covariance matrix of the form \( \bm{C} = \bm{\wtC} \bm{\wtC}^{\top} \). The entries of the matrix \( \bm{\wtC} \) are standard normally distributed and fixed for all runs. We rescale \( \bm{C} \) so that \( K_{\rho} = 1 \), see Assumption~\ref{assumption : noise model}. \\

\noindent
\textbf{PCA projection errors.} \quad We set \( d_{\cX} = d_{\cY} = 5 \). 
If \( d_{\cX} < \dim(\cX) \) and \( d_{\cY} < \dim(\cY) \), we need to take into account all error terms~\eqref{eq: true PCA projection error}--\eqref{eq: noise error}. We again first consider the noiseless case \( \sigma = 0 \), where \( \mathrm{Err}_{\textup{noise}} = 0 \).
From~\eqref{eq: empirical PCA projection error} and our choices of \( N_{\cX}, N_{\cY} \), we expect the term \( \mathrm{Err}_{\textup{emp-PCA}} \) to be of the same order \( \cO(\log(s)^{-1}) \) as \( \mathrm{Err}_{\textup{approx}} \).
As \( \mathrm{Err}_{\textup{true-PCA}} \) is independent of \( s \), it dominates the Hermite-PCA approximation error for sufficiently large \(s\). This can be observed in Figure~\ref{fig:all}, where the error curve first decays at a rate of order \( \cO(\log(s)^{-1}) \) before it levels off for large values of~\( s \).

We repeat the experiment with positive noise level \( \sigma = 5 \), so that the noise error \( \mathrm{Err}_{\textup{noise}} \) impacts the total error as well. We consequently expect an overall error decay of order \( \cO(\log(s)^{-1/2}) \) until the error reaches a plateau determined by \( \mathrm{Err}_{\textup{true-PCA}} \) and the additive noise term \( \cO(\sigma) \). This can be observed in Figure~\ref{fig:all-noise}.

\begin{figure}[h!]
\centering
    \begin{subfigure}{0.4\linewidth}
        \begin{tikzpicture}
        \begin{axis}[
            width=0.9\linewidth,
            height=1.5\linewidth,
            xmode=log, ymode=log,
            xlabel={s}, ylabel={Relative $L_{\mu}^2$ error},
            ytick={10^(-0.9), 10^(-0.7), 10^(-0.5), 10^(-0.3)},
            grid=both,
            legend style={
                at={(0.98,0.98)},
                anchor=north east,
                draw=black,
                fill=white,
                font=\scriptsize,
            },
        ]  
        \addplot[blue, mark=square*, line width=1.5pt,  mark size=1pt]
            table[x=s, y=mean_l2_at_s, col sep=comma]{csv/approx-no_noise.csv};
        \addlegendentry{Mean error};    
        \addplot[name path=upper, draw=none, forget plot]
            table[x=s, y expr=\thisrow{mean_l2_at_s} + \thisrow{std_l2_at_s},
                  col sep=comma]{csv/approx-no_noise.csv};
        \addplot[name path=lower, draw=none, forget plot]
            table[x=s, y expr=\thisrow{mean_l2_at_s} - \thisrow{std_l2_at_s},
                  col sep=comma]{csv/approx-no_noise.csv};
        \addplot[blue!30, fill opacity=0.5, forget plot]
            fill between[of=upper and lower];
        \addplot[orange, mark=none, dashed, line width=1.5pt,
                 domain=1e1:3e3, samples=200, on layer=axis foreground]
            {1.2*ln(x)^(-1)};
        \addlegendentry{$\asymp \log(s)^{-1}$};
        \end{axis}
        \end{tikzpicture}
        \caption{\centering Approximation error \\ $d_{\cX} = d_{\cY} = 10$, $\sigma = 0$}
        \label{fig:approx}
    \end{subfigure}
    \hspace{0.05\linewidth}
    \begin{subfigure}{0.4\linewidth}
        \begin{tikzpicture}
        \begin{axis}[
            width=0.9\linewidth,
            height=1.5\linewidth,
            xmode=log, ymode=log,
            xlabel={s}, ylabel={Relative $L_{\mu}^2$ error},
            grid=both,
            legend style={
                at={(0.98,0.98)},
                anchor=north east,
                draw=black,
                fill=white,
                font=\scriptsize,
            },
        ]  
        \addplot[blue, mark=square*, line width=1.5pt,  mark size=1pt]
            table[x=s, y=mean_l2_at_s, col sep=comma]{csv/approx-noise.csv};
        \addlegendentry{Mean error};    
        
        \addplot[name path=upper, draw=none, forget plot]
            table[x=s, y expr=\thisrow{mean_l2_at_s} + \thisrow{std_l2_at_s},
                  col sep=comma]{csv/approx-noise.csv};
        \addplot[name path=lower, draw=none, forget plot]
            table[x=s, y expr=\thisrow{mean_l2_at_s} - \thisrow{std_l2_at_s},
                  col sep=comma]{csv/approx-noise.csv};
        \addplot[blue!30, fill opacity=0.5, forget plot]
            fill between[of=upper and lower];
        
        \addplot[orange, mark=none, dashed, line width=1.5pt,
                 domain=1e1:3e3, samples=200, on layer=axis foreground]
            {2.5*ln(x)^(-1/2)};
        \addlegendentry{$\asymp \log(s)^{-1/2}$}
    
        \end{axis}
        \end{tikzpicture}
        \caption{\centering Noise error \\ $d_{\cX} = d_{\cY} = 10$, $\sigma = 5$} 
        \label{fig:approx-noise}
    \end{subfigure}
    \caption{\emph{\textbf{Approximation and noise errors.}} In the absence of true and empirical PCA projection and noise errors (left), the Hermite-PCA approximation error is bounded by the approximation error term \( \mathrm{Err}_{\textup{approx}} \) and decays with its predicted rate of order \( \cO(\log(s)^{-1}) \). A log-linear \(s\)-scaling of the number of PCA samples used for empirical PCA seemingly suffices for this. In the presence of noise (right), the Hermite-PCA approximation error converges at the rate of order \( \cO(\log(s)^{-1/2}) \) predicted by the dominating noise error term $\widetilde{\mathrm{Err}}_{\textup{noise}}$.}
    \label{fig:true-pca-approx-noise}
\end{figure}

\begin{figure}[h!]
\vspace{3ex}
\centering
    \begin{subfigure}{0.4\linewidth}
        \centering
        \begin{tikzpicture}
        \begin{axis}[
            width=0.9\linewidth,
            height=1.5\linewidth,
            xmode=log, ymode=log,
            xlabel={s}, ylabel={Relative $L_{\mu}^2$ error},
            ytick={10^(-0.9), 10^(-0.7), 10^(-0.5), 10^(-0.3), 10^(-0.1)},
            grid=both,
            legend style={
                at={(0.98,0.98)},
                anchor=north east,
                draw=black,
                fill=white,
                font=\scriptsize,
            },
        ]  
        \addplot[blue, mark=square*, line width=1.5pt,  mark size=1pt]
            table[x=s, y=mean_l2_at_s, col sep=comma]{csv/all-no_noise.csv};
        \addlegendentry{Mean error};    
        
        \addplot[name path=upper, draw=none, forget plot]
            table[x=s, y expr=\thisrow{mean_l2_at_s} + \thisrow{std_l2_at_s},
                  col sep=comma]{csv/all-no_noise.csv};
        \addplot[name path=lower, draw=none, forget plot]
            table[x=s, y expr=\thisrow{mean_l2_at_s} - \thisrow{std_l2_at_s},
                  col sep=comma]{csv/all-no_noise.csv};
        \addplot[blue!30, fill opacity=0.5, forget plot]
            fill between[of=upper and lower];
        \addplot[orange, mark=none, dashed, line width=1.5pt,
                 domain=1e1:3e3, samples=200, on layer=axis foreground]
            {0.8*ln(x)^(-1)};
        \addlegendentry{$\asymp \log(s)^{-1}$}
        \end{axis}
        \end{tikzpicture}
        \caption{\centering PCA proj. and approx. error \\ $d_{\cX} = d_{\cY} = 5$, $\sigma = 0$}
        \label{fig:all}
    \end{subfigure}
    \hspace{0.05\linewidth}
    \begin{subfigure}{0.4\linewidth}
        \centering
        \begin{tikzpicture}
        \begin{axis}[
            width=0.9\linewidth,
            height=1.5\linewidth,
            xmode=log, ymode=log,
            xlabel={s}, ylabel={Relative $L_{\mu}^2$ error},
            grid=both,
            legend style={
                at={(0.98,0.98)},
                anchor=north east,
                draw=black,
                fill=white,
                font=\scriptsize,
            },
        ]  
        \addplot[blue, mark=square*, line width=1.5pt,  mark size=1pt]
            table[x=s, y=mean_l2_at_s, col sep=comma]{csv/all-noise.csv};
        \addlegendentry{Mean error};    
        
        \addplot[name path=upper, draw=none, forget plot]
            table[x=s, y expr=\thisrow{mean_l2_at_s} + \thisrow{std_l2_at_s},
                  col sep=comma]{csv/all-noise.csv};
        \addplot[name path=lower, draw=none, forget plot]
            table[x=s, y expr=\thisrow{mean_l2_at_s} - \thisrow{std_l2_at_s},
                  col sep=comma]{csv/all-noise.csv};
        \addplot[blue!30, fill opacity=0.5, forget plot]
            fill between[of=upper and lower];
        \addplot[orange, mark=none, dashed, line width=1.5pt,
                 domain=1e1:3e3, samples=200, on layer=axis foreground]
            {2.05*ln(x)^(-1/2)};
        \addlegendentry{$\asymp \log(s)^{-1/2}$}
        \end{axis}
        \end{tikzpicture}
        \caption{\centering All errors \\ $d_{\cX} = d_{\cY} = 5$, $\sigma = 5$} 
        \label{fig:all-noise}
    \end{subfigure}
    \caption{\emph{\textbf{PCA projection errors.}} If \( d_{\cX} < \dim(\cX) \) and \( d_{\cY} < \dim(\cY) \), true and empirical PCA projection errors impact the total error. In the noiseless case (left), the Hermite-PCA approximation error decays at the rate \( \cO(\log(s)^{-1}) \) predicted by \( \textup{Err}_{\textup{approx}} + \textup{Err}_{\textup{emp-PCA}} \). For sufficiently large $s$, it is dominated by the true PCA projection error $\mathrm{Err}_{\textup{true-PCA}}$, which is independent of \( s \). Consequently, the error curve levels off if \( s \) is large enough. In the presence of noise, the total error decay is dominated, as predicted, by the noise error \( \textup{Err}_{\textup{noise}} \) of order \( \cO(\log(s)^{-1/2}) \) until, for sufficiently large \( s \), it reaches a plateau determined by \( \mathrm{Err}_{\textup{true-PCA}} + \cO(\sigma) \).}
\end{figure}

In summary, the findings in this subsection illustrate the match between the empirically observed convergence of the Hermite-PCA approximation algorithm and our theoretical error bounds. The algorithm appears to perform even better in practice --- showing near-optimal scaling in all hyperparameters --- than predicted by our theory.

\subsection{Approximation of functionals}

To illustrate the influence of the Sobolev order \( k \) on the approximation error, as predicted in Section~\ref{subsec: concrete rates of approximation}, we approximate parametric functionals whose parameter allows us to control the Sobolev regularity.
Let \( \varrho = \bigotimes_{i = 1}^{d_{\cX}} \cN(0, \lambda_i) \). For \( \bm{a} = (a_i)_{i = 1}^{\infty} \in \ell^2(\bbN) \) and \( b > 0 \), define the functional
\[
F_b : \ell^2(\bbN) \to \bbR, \quad F_b(\bm{x}) = \absd{\sum_{i = 1}^{\infty} a_i x_i}^b,
\]
as well as the input-truncated function
\[
f_b(\bm{x}) := (F_b \circ \DX)(\bm{x}) = \absd{\sum_{i = 1}^{d_{\cX}} a_i x_i}^b = \abs{\ip{\bm{\widetilde{a}}}{\bm{x}}_2}^b, \quad \bm{x} \in \bbR^{d_{\cX}},
\]
with \( \bm{\widetilde{a}} := (a_1, \dots, a_{d_{\cX}}) \). Suppose that \( \bm{\widetilde{a}} \neq \bm{0} \).
The next result shows how the parameter \( b \) influences the Sobolev regularity of \( f_b \).
\begin{proposition}
\label{prop: f-b regularity}
    We have \( f_b \in H_{\varrho}^{k}(\bbR^{d_{\cX}}) \) if and only if \( b > k - \frac{1}{2} \) or \( b \in 2 \bbN \).
\end{proposition}

\begin{proof}
    We start with the change of variables 
    \( 
    \bm{x} = \bm{A}^{1/2} \bm{z}
    \)
    with \( \bm{A} := \textup{diag}(\lambda_1, \dots, \lambda_{d_{\cX}}) \),
    which transforms \( \varrho \) into the standard Gaussian measure \( \cN_{d_{\cX}} := \otimes_{i = 1}^{d_{\cX}} \cN(0, 1) \). Then 
    \[ 
    f_b(\bm{A}^{1/2} \bm{z}) = \abs{\ip{\bm{A}^{1/2} \bm{\widetilde{a}}}{ \bm{z} }_2}^b.
    \]
    Next, choose an orthogonal matrix \( \bm{Q} \) such that
    \(
    \bm{Q} \bm{A}^{1/2} \bm{\widetilde{a}} = \beta \bm{e_1}
    \)
    with \( \beta := \sqrt{\sum_{i = 1}^{d_{\cX}} \lambda_i a_i^2} \) and \( \bm{e_1} = (\delta_{1,i})_{i = 1}^{d_{\cX}} \).
    Setting \( \bm{y} = \bm{Q} \bm{z}\), we find that \( f_b(\bm{A}^{1/2} \bm{Q}^{\top} \bm{y}) = \beta^b \abs{y_1}^b \). By rotational invariance of~\( \cN_{d_{\cX}} \), we conclude that
    \[
    f_b \in H_{\varrho}^{k}(\bbR^{d_{\cX}}) \iff \abs{y_1}^b \in H^k_{\cN_{d_{\cX}}}(\bbR^{d_{\cX}})
    \iff h_b \in H^k_{\cN_1}(\bbR) \textup{ with } h_b(t) := \abs{t}^b.
    \]
    If \( b \not\in \bbN \), then the derivatives of \( h_b \) are given by
    \[
    \frac{\D^j}{\D t^j} h_b(t) = b (b-1) \cdots (b-j+1) \abs{t}^{b - j} \sgn(t)^j, \quad \forall t \neq 0, \forall j \in [k],
    \]
    and \( b (b-1) \cdots (b-j+1) \neq 0 \). In this case, we have \( h_b \in H^k_{\textup{loc}}(\bbR) \) and \( \int_{\bbR} \absd{\frac{\D^j}{\D t^j} h_b}^2 \D \cN_1 < \infty \) for every \( j \in [k] \) if and only if \( b > k - \frac{1}{2} \).
    If \( b \in 2 \bbN \), then \( h_b \) is a monomial and hence lies in \( H^k_{\cN_1}(\bbR) \).
    Finally, assume that \( b \) is an odd integer. Then \( \frac{\D^b}{\D t^b} h_b(t) = b! \, \textup{sign}(t) \) for \( t \neq 0 \). We conclude that \( h_b \in H^k_{\textup{loc}}(\bbR) \) and \( \int_{\bbR} \absd{\frac{\D^j}{\D t^j} h_b}^2 \D \cN_1 < \infty \) for every \( j \in [k] \) if and only if \( k \leq b \). The claim now follows from Proposition~\ref{prop: classical vs. Gaussian weak derivatives}.
\end{proof}

\begin{figure}[h]
    \begin{subfigure}{0.3\linewidth}
        \centering
        \begin{tikzpicture}
        \begin{axis}[
            width=\linewidth,
            height=1.6\linewidth,
            xmode=log, ymode=log,
            xlabel={s}, ylabel={Relative $L_{\mu}^2$ error},
            ytick={10^(-0.7), 10^(-0.9), 10^(-1.1), 10^(-1.3)},
            grid=both,
            legend style={
                at={(0.98,0.98)},
                anchor=north east,
                draw=black,
                fill=white,
                font=\scriptsize,
            },
        ]  
        \addplot[blue, mark=square*, line width=1.5pt, mark size=1pt]
            table[x=s, y=mean_l2_at_s, col sep=comma]{csv/functional_1.csv};
        \addlegendentry{Mean error};    
    
        \addplot[name path=upper, draw=none, forget plot]
            table[x=s, y expr=\thisrow{mean_l2_at_s} + \thisrow{std_l2_at_s},
                  col sep=comma]{csv/functional_1.csv};
        \addplot[name path=lower, draw=none, forget plot]
            table[x=s, y expr=\thisrow{mean_l2_at_s} - \thisrow{std_l2_at_s},
                  col sep=comma]{csv/functional_1.csv};
        \addplot[blue!30, fill opacity=0.5, forget plot]
            fill between[of=upper and lower];

        \addplot[orange, mark=none, dashed, line width=1.5pt,
                 domain=1e1:3e3, samples=200, on layer=axis foreground]
            {0.45*ln(x)^(-1)};
        \addlegendentry{$\asymp \log(s)^{-1}$}
        \end{axis}
        \end{tikzpicture}
        \caption{\centering Approximation of $F_{1 - \frac{1}{2} + 0.001}$ \\ $k=1$, $d_{\cX}=10$, $\sigma=0$}
        \label{fig:functional_1}
    \end{subfigure}
        \hspace{0.02\linewidth}
    \begin{subfigure}{0.3\linewidth}
        \centering
        \begin{tikzpicture}
        \begin{axis}[
            width=\linewidth,
            height=1.6\linewidth,
            xmode=log, ymode=log,
            xlabel={s}, ylabel={Relative $L_{\mu}^2$ error},
            ytick={10^(-1.2), 10^(-1.6), 10^(-2), 10^(-2.4)},
            grid=both,
            legend style={
                at={(0.98,0.98)},
                anchor=north east,
                draw=black,
                fill=white,
                font=\scriptsize,
            },
        ]  
        \addplot[blue, mark=square*, line width=1.5pt, mark size=1pt]
            table[x=s, y=mean_l2_at_s, col sep=comma]{csv/functional_2.csv};
        \addlegendentry{Mean error};    
    
        \addplot[name path=upper, draw=none, forget plot]
            table[x=s, y expr=\thisrow{mean_l2_at_s} + \thisrow{std_l2_at_s},
                  col sep=comma]{csv/functional_2.csv};
        \addplot[name path=lower, draw=none, forget plot]
            table[x=s, y expr=\thisrow{mean_l2_at_s} - \thisrow{std_l2_at_s},
                  col sep=comma]{csv/functional_2.csv};
        \addplot[blue!30, fill opacity=0.5, forget plot]
            fill between[of=upper and lower];

        \addplot[orange, mark=none, dashed, line width=1.5pt,
                 domain=1e1:3e3, samples=200, on layer=axis foreground]
            {0.35*ln(x)^(-2)};
        \addlegendentry{$\asymp \log(s)^{-2}$}
        \end{axis}
        \end{tikzpicture}
        \caption{\centering Approximation of $F_{2 - \frac{1}{2} + 0.001}$ \\ $k=2$, $d_{\cX}=10$, $\sigma=0$}
        \label{fig:functional_2}
    \end{subfigure}
    \hspace{0.02\linewidth}
    \begin{subfigure}{0.3\linewidth}
        \centering
        \begin{tikzpicture}
        \begin{axis}[
            width=\linewidth,
            height=1.6\linewidth,
            xmode=log, ymode=log,
            xlabel={s}, ylabel={Relative $L_{\mu}^2$ error},
            ytick={10^(-1.6), 10^(-2), 10^(-2.4), 10^(-2.8), 10^(-3.2)},
            grid=both,
            legend style={
                at={(0.98,0.98)},
                anchor=north east,
                draw=black,
                fill=white,
                font=\scriptsize,
            },
        ]  
        \addplot[blue, mark=square*, line width=1.5pt, mark size=1pt]
            table[x=s, y=mean_l2_at_s, col sep=comma]{csv/functional_3.csv};
        \addlegendentry{Mean error};    
    
        \addplot[name path=upper, draw=none, forget plot]
            table[x=s, y expr=\thisrow{mean_l2_at_s} + \thisrow{std_l2_at_s},
                  col sep=comma]{csv/functional_3.csv};
        \addplot[name path=lower, draw=none, forget plot]
            table[x=s, y expr=\thisrow{mean_l2_at_s} - \thisrow{std_l2_at_s},
                  col sep=comma]{csv/functional_3.csv};
        \addplot[blue!30, fill opacity=0.5, forget plot]
            fill between[of=upper and lower];

        \addplot[orange, mark=none, dashed, line width=1.5pt,
                 domain=1e1:3e3, samples=200, on layer=axis foreground]
            {0.4*ln(x)^(-3)};
        \addlegendentry{$\asymp \log(s)^{-3}$}
        \end{axis}
        \end{tikzpicture}
        \caption{\centering Approximation of $F_{3 - \frac{1}{2} + 0.001}$ \\ $k=3$, $d_{\cX}=10$, $\sigma=0$}
        \label{fig:functional_3}
    \end{subfigure}
    \caption{\emph{\textbf{Spectral approximation.}} In the absence of PCA projection and noise errors, the Hermite-PCA approximation error decays in a spectral fashion at the rate \( \cO(\log(s)^{-k} ) \) predicted by the term \(\mathrm{Err}_{\textup{approx}}\): the smoother the target functional, the faster the convergence.}
    \label{fig:functional}
\end{figure}

\begin{remark}
    Note that \( F_b \) is not Lipschitz continuous for \( b \neq 1 \), which Theorem~\ref{thm: Hermite-PCA Sobolev error bound} requires. However, if we apply true PCA and assume no noise, inspection of the proof of Theorem~\ref{thm: Hermite-PCA Sobolev error bound} shows that we can estimate the approximation error by 
    \begin{align*}
        \textup{Err}_{\textup{approx}} &\leq \left( 1 + \sqrt{\frac{3 s}{M \epsilon}} \frac{1}{1 - \delta} \right) \inf_{p \in \cP_{\bbR^{d_{\cY}}}} \nm{ p - \EY \circ F_b \circ \DX }_{L^2_{\varrho}(\bbR^{d_{\cX}}; \bbR^{d_{\cY}})} \\
        &\lesssim_{k, \Lambda} \frac{1}{\sqrt{\epsilon}} v_{\bm{\tau(s + 1)}}^k \nm{\EY \circ F_b \circ \DX}_{H^k_{\varrho}(\bbR^{d_{\cX}}; \bbR^{d_{\cY}})},
    \end{align*}
    where we used \( M \geq s \) in the last step.
    See \eqref{eq: least squares error bound} (note that \( \varpi = \cD_{\cX} \sharp \varrho \) for true PCA) and Proposition~\ref{prop: higher rates on finite-dim domains}. The implicit constant depends on \( k \) and \( \Lambda = \sum_{i = 1}^{d_{\cX}} \lambda_i^{-1} \).
    The functional \( F_b \) is thus suitable to study the decay of the approximation error for different values of~\( b \).
\end{remark}

In our experiments, we set \(a_i = i^{-4}\), \(b = k - \frac{1}{2} + 0.001\), and consider the cases \( k = 1, 2, 3 \). We truncate the input sequences after \( 10 \) elements and use true PCA for the encoder in \(\cX\) with \( d_{\cX} = 10 \). As \( \dim(\cY) = d_{\cY} = 1 \), no decoding in \(\cY\) is needed. To avoid noise errors, we set \( \sigma = 0 \). Consequently, the only non-vanishing error term is the approximation error \( \mathrm{Err}_{\textup{approx}} \). We again take \( \lambda_i \asymp i^{-2} \) and normalize such that \( \sum_{i=1}^{10} \lambda_i = 1 \). By~\eqref{eq: approx error; alg EVs} and Proposition~\ref{prop: f-b regularity}, we expect the approximation error to be of order \( \cO(\log(s)^{-k}) \). This is reflected by the results in Figure~\ref{fig:functional}. 

Overall, the findings in this subsection confirm a major facet of the Hermite-PCA approximation algorithm: its spectral approximation property. The algorithm itself is independent of the Sobolev order \( k \), but it automatically yields faster convergence the smoother the target operator is. The empirically observed decay rates are consistent with the rates predicted by our theory.


\section{A general error bound for Hermite-PCA approximation}
\label{sec: general error bound}

This and the next two sections of this paper are devoted to the proof of Theorem \ref{thm: Hermite-PCA Sobolev error bound}. In this section, we begin by establishing a general error bound in Theorem \ref{thm: Hermite-PCA error bound} for the Hermite-PCA approximation defined in Algorithm \ref{alg: Hermite-PCA approximation}, without making any regularity assumptions on \( F \) and also allowing for an arbitrary multi-index set \( S \) in the definition of the subspace \( \cP_{\bbR^{d_{\cY}}} \) in \eqref{eq: hat-f subspace def}. This yields the desired error bound, up to a best approximation error of the encoded function \( \hEY \circ F \circ \hDX \) in \( \cP_{\bbR^{d_{\cY}}} \). To analyze this term, we require an \( \ell^2 \)-characterization of Gaussian Sobolev spaces in terms of Hermite polynomial coefficients, together with several estimates for the decay of the weight sequences that arise in this characterization. This is the topic of Section~\ref{sec: l2-characterization}. Finally, in Section~\ref{sec: main res proof} we combine these two pieces to establish Theorem \ref{thm: Hermite-PCA Sobolev error bound}.

\begin{theorem}[General error bound for Hermite-PCA approximation]
\label{thm: Hermite-PCA error bound}
    Let \( 0 < \Upsilon < 1 \) and \(0 < \delta, \epsilon < 1\) be fixed and let \(S \subset \bbN^{d_{\cX}}_0 \) with \( \bm{0} \in S \), \( s : = |S| \), and \( m(S) : = \max_{\bm{\gamma} \in S} \nm{\bm{\gamma}}_1 \geq 3 \). 
    Suppose that
    \begin{equation}
    \label{eq: condition for N_X}
        N_{\cX} \geq \max \left\{ C_1 d_{\cX} \log(12 / \epsilon) \Upsilon^{-4}, C_2^2 \log(6 / \epsilon) \kappa^{-2} (\lambda_{d_{\cX}})^{-2} \right\}
    \end{equation}
    with parameter \( \kappa > 0 \) satisfying
    \begin{align}
    \begin{split}
    \label{eq: condition for kappa}
        \kappa \leq \min \bigg \{ \frac{1}{2} \left(1 - \frac{1}{16^{1/d_{\cX}} }\right) ,
        \frac{\lambda_{d_{\cX}} }{118} \bigg( m(S) \log(m(S)) + \log(s) +
        \abs{\log(\Upsilon)} + m(S) \abs{\log(\lambda_{d_{\cX}})} \bigg)^{-1} \bigg \},
    \end{split}
    \end{align}
    where  \(C_1, C_2 \geq 1\) are the constants from Propositions~\ref{prop: PCA error} and~\ref{prop: bound for operator norm of Delta_mu}, respectively. 
    Moreover, suppose that 
    \begin{equation}
    \label{eq: condition for N_Y}
        N_{\cX} \geq N_{\cY} \geq C_1 d_{\cY} \log(12 / \epsilon) \Upsilon^{-4}
    \end{equation}
    and 
    \begin{equation}
    \label{eq: condition for M}
        M \geq C_{\delta} \, s \log(12 s/ \epsilon), \qquad C_{\delta} = ((1 + \delta) \log(1 + \delta) - \delta)^{-1}.
    \end{equation}
    Let \( F \) be \( L \)-Lipschitz and set
    \[
    \widetilde{\Upsilon} := \left ( 1 + \sqrt{\frac{6 s}{M \epsilon}} \frac{1}{1-\delta} \right ) \inf_{p \in \cP_{\bbR^{d_{\cY}}}} \nm{\hEY \circ F \circ \hDX- p }_{L^2_{\hat{\varrho}}(\bbR^{d_{\cX}} ; \bbR^{d_{\cY}})}  + \sigma \sqrt{\frac{12 s}{M \epsilon} } \frac{1}{1-\delta}
    \]
    where \( \cP_{\bbR^{d_{\cY}}} \) is as in \eqref{eq: hat-f subspace def} with index set \( S \).
    Then, with probability at least \(1 - \epsilon\) in the draw of \( (\whX_i)_{i \in [N_{\cX}]} \), \( (\whE_i)_{ i \in [N_{\cY}] } \), and, conditionally on the PCA data, \( (X_i, E_i)_{i \in [M]} \), the approximation \(\whF\) defined by Algorithm \ref{alg: Hermite-PCA approximation} uniquely exists and satisfies
    \begin{align*}
    \begin{split}
        \nm{F - \whF}_{L_{\mu}^2(\cX; \cY)}
        & \leq L \sqrt{\sum^{\dim(\cX)}_{i = d_{\cX}+1} \lambda_i} + L K_{\mu} \Upsilon
        +  \sqrt{\sum^{\dim(\cY)}_{i = d_{\cY}+1} \lambda^{F \sharp \mu}_i} + (\nm{F(0)}_{\cY} + L K_{\mu} + \sigma ) \Upsilon + 2 \sqrt{2} \sigma \\
        & \quad + 4 \widetilde{\Upsilon} + 5 \Upsilon(\widetilde{\Upsilon} + \wtC_F) 
        \end{split}
    \end{align*}
where  \(\wtC_F = \nm{F(0)}_{\cY} + (\sqrt{8} + 1) L\).
\end{theorem}

We remark that the constant \( 118 \) in~\eqref{eq: condition for kappa} can be decreased by a more refined analysis aimed at optimizing constants. This is not the purpose of this work.

\subsection{Overview of the proof of Theorem \ref{thm: Hermite-PCA error bound}}

The remainder of this section is devoted to the proof of Theorem \ref{thm: Hermite-PCA error bound}. By the triangle inequality, we can split the total error into three terms,
\begin{align}
\begin{split}
\label{eq: total error splitting}
 \nm{\whF - F}_{L_{\mu}^2(\cX; \cY)}
    &\leq    \nm{\whF - \hDY \circ \hEY \circ F \circ \hDX \circ \hEX}_{L_{\mu}^2(\cX; \cY)}  \\
    &\quad + \nm{\hDY \circ \hEY \circ F \circ \hDX \circ \hEX - \hDY \circ \hEY \circ F}_{L_{\mu}^2(\cX; \cY)} \\
    &\quad + \nm{\hDY \circ \hEY \circ F - F}_{L_{\mu}^2(\cX; \cY)} \\
    &=: \textup{Err}_{\cA} + \textup{Err}_{\cX} + \textup{Err}_{\cY}.
\end{split}
\end{align}
The right-hand side consists of the \emph{approximation error} \(\textup{Err}_{\cA}\) and the \emph{empirical PCA projection errors} \(\textup{Err}_{\cX}\) and  \(\textup{Err}_{\cY}\). Recall that $F$ is $L$-Lipschitz and observe that the encoders and decoders are all $1$-Lipschitz by construction. Therefore, we can simplify these error terms as follows:
\begin{align}
    &\textup{Err}_{\cX} \leq  L  \nmd{\hDX \circ \hEX - \cI_{\cX}}_{L_{\mu}^2(\cX; \cX)}, 
    \label{eq: encoding error: general} \\
    &\textup{Err}_{\cY} = \nm{\hDY \circ \hEY - \cI_{\cY}}_{L_{F \sharp \mu}^2(\cY; \cY)}, 
    \label{eq: decoding error: general} \\
    &\textup{Err}_{\cA} 
    =   \nmd{\whF - \hDY \circ \hEY \circ F \circ \hDX \circ \hEX}_{L_{\mu}^2(\cX; \cY)}. 
    \label{eq: approximation error: general}
\end{align}
Here, \(\cI_{\cX}\), \(\cI_{\cY}\) denote the identity operators on \(\cX\) and \(\cY\), respectively.
In the next three subsections, we estimate these terms separately.

\subsection{True PCA projection errors}\label{subsubsec: true PCA: encoding and decoding error}

Since \(\hEX , \hDX \) are approximations to the true PCA encoders and decoders \( \cE_{\cX} , \cD_{\cX} \), and likewise for \(\hEY , \hDY \), our first task is to study the true PCA projection errors. By~\cite[Thm. 3.8]{lanthaler_ErrorEstimatesDeepONetsDeep_2022}, we have 
\begin{equation}
\label{eq: encoding projection error}
    \nm{\DX \circ \EX - \cI_{\cX}}_{L_{\mu}^2(\cX; \cX)}
    = \sqrt{\sum_{i = d_{\cX} + 1}^{\dim(\cX)} \lambda_i}.
\end{equation}
Analogously,
\begin{equation}
\label{eq: decoding projection error}
    \nmd{\DY \circ \EY - \cI_{\cY}}_{L_{\nu}^2(\cY; \cY)}
    = \sqrt{\sum_{i = d_{\cY} + 1}^{\dim(\cY)} \lambda_i^{\nu}},
\end{equation}
where we recall from Section \ref{sec: setup} that \(\nu\) is the distribution of \( Y = F(X) + \sigma E \) with \( X \sim \mu \) and \( E \sim \rho \) being independent. However, the decoding error~\eqref{eq: decoding error: general}  is given in terms of the measure \( F  \sharp \mu \), which only coincides with \( \nu \) in the  noiseless case \(\sigma = 0\). To handle the noisy case, we require the following lemma.

\begin{lemma}[True PCA projection error, noisy case]
\label{lem: true PCA decoding, noisy}
    The true PCA projection error satisfies
    \begin{equation*}
        \nmd{\DY \circ \EY - \cI_{\cY}}_{L_{\nu}^2(\cY; \cY)} \leq \sqrt{\sum_{i = d_{\cY} + 1}^{\dim(\cY)} \lambda_i^{F \sharp \mu}} + \sqrt{2} \sigma.
    \end{equation*}
\end{lemma}
Note that this bound is noise-consistent in the sense that if \(\sigma \to 0^+\), then we recover the noiseless PCA projection error in \( \cY \).

\begin{proof}
    By~\eqref{eq: decoding projection error}, the triangle inequality, and the Hoffman--Wielandt inequality (see e.g.~\cite[Lemma 5.1]{markus_EigenAndSingularValues_1964}), we have
    \[
    \nmd{\DY \circ \EY - \cI_{\cY}}_{L_{\nu}^2(\cY; \cY)}^2
    \leq \sum_{i = d_{\cY} + 1}^{\dim(\cY)} \lambda_i^{F \sharp \mu}
    + \sum_{i = d_{\cY} + 1}^{\dim(\cY)} \abs{\lambda_i^{\nu} - \lambda_i^{F \sharp \mu}}
    \leq \sum_{i = d_{\cY} + 1}^{\dim(\cY)} \lambda_i^{F \sharp \mu} + \nm{\Sigma_{\nu} - \Sigma_{F \sharp \mu}}_1.
    \]
    Here, \(\nm{\cdot}_1\) denotes the \(1\)-Schatten norm (or trace norm), i.e., the sum of the singular values of a bounded linear operator. Since \( X \) and \( E \) are independent and \( \bbE[E] = 0 \), a short computation shows that 
    \(
    \Sigma_{\nu} - \Sigma_{F \sharp \mu} = \bbE[\sigma E \otimes \sigma E].
    \)
    Together with the readily checked property \(\nm{Z_1 \otimes Z_2}_1 = \nm{Z_1}_{\cY} \nm{Z_2}_{\cY}\) for \(Z_1, Z_2 \in \cY\) we conclude that 
    \[
    \nm{\Sigma_{\nu} - \Sigma_{F \sharp \mu}}_1
    \leq \bbE[\nm{\sigma E}_{\cY}^2]
    \leq 2 \sigma^2,
    \]
    where the last step holds by~\eqref{eq: expectation of E}.
\end{proof}

\subsection{Empirical PCA projection errors}

We now estimate \( \mathrm{Err}_{\cX} \) and \( \mathrm{Err}_{\cY} \). For this, we use the following result which relates the empirical and true PCA projection errors. To state it in suitable generality, let us consider a generic separable Hilbert space \(\cH\) equipped with a sub-Gaussian measure \(\tilde{\mu}\) with parameter \(K\). We define the true PCA encoder and decoder \(\cE_{\cH}, \cD_{\cH}\) with dimension \(d_{\cH}\) analogously to~\eqref{eq: true PCA encoder decoder for X}. Furthermore, we draw \(n\) samples \(Z_1, \dots, Z_n \sim_{\textup{i.i.d.}} \tilde{\mu}\) and define the empirical PCA encoder and decoder \(\widehat{\cE}_{\cH}, \widehat{\cD}_{\cH}\) analogously as in~\eqref{eq: empirical PCA encoder decoder for X}. 

\begin{proposition}[Empirical vs. true PCA projection error]
\label{prop: PCA error}
    Let \(\epsilon \in (0, 1)\). The projection error for empirical PCA, based on \(N \geq \log(2 / \epsilon)\) samples \(Z_1, \dots, Z_N \sim_{\textup{i.i.d.}} \tilde{\mu}\), satisfies the bound
    \[
    \nm{\widehat{\cD}_{\cH} \circ \widehat{\cE}_{\cH} - \cI_{\cH}}_{L_{\tilde{\mu}}^2(\cH; \cH)}
    \leq \nm{\cD_{\cH} \circ \cE_{\cH} - \cI_{\cH}}_{L_{\tilde{\mu}}^2(\cH; \cH)}
    + \left( \frac{C_1 K^4 d_{\cH} \log(2 / \epsilon)}{N} \right)^{1/4}
    \]
    with probability at least \(1 - \epsilon\) in the draw of \(Z_1, \dots, Z_N\). Here, \(C_1 \geq 1\) is an absolute constant.
\end{proposition}

\begin{proof}
    This is essentially Proposition 2 in~\cite{lanthaler_OperatorLearningPCANetUpper_2023}. Therein, the dependence of the constant \(Q = C_1 K^4\) on \(K\) is not made explicit, but follows by inspection of the proof. We also note that the failure probability \(\epsilon\) is restricted therein to \((0, 1/2)\) and \(C_1\) is only stated to be positive. The proof, however, also works for \(\epsilon \in (0,1)\) and by enlarging \(C_1\) if necessary, we can choose \(C_1 \geq 1\).
\end{proof}

\begin{remark}
\label{rmk: PCA error}
    The scaling of \( N \) in Proposition~\ref{prop: PCA error} can be improved considerably in some cases. In fact, Milbradt and Wahl~\cite{milbradt_HighprobabilityBoundsReconstructionError_2020} considered certain sub-Gaussian measures under a stronger condition than~\eqref{eq: sub-Gaussian moment condition}, satisfied, e.g., by the Gaussian measure \( \mu \) on \( \cX \). They showed that with high probability the empirical PCA projection error is at most of the same order as the true PCA projection error if \( N \gtrsim d_{\cH} \log( 1/\epsilon ) \)~\cite[Thm. 2]{milbradt_HighprobabilityBoundsReconstructionError_2020}. Their result, however, depends on specific eigenvalue-decay assumptions (e.g. polynomial or exponential) for the covariance operator of \( \tilde{\mu} \). We impose no such assumptions in this paper.
\end{remark}

We can thus estimate the empirical PCA projection errors with high probability by the bounds for the true PCA projection error, which we derived in Section~\ref{subsubsec: true PCA: encoding and decoding error}, and by the sampling error terms, which depend on the numbers of samples \(N_{\cX}\), \(N_{\cY}\) and the encoding and decoding dimensions \(d_{\cX}\),~\(d_{\cY}\). We summarize this in the following result:

\begin{theorem}[Empirical PCA projection errors]
\label{thm: empirical PCA projection errors}
    Let \(\epsilon \in (0,1)\) and consider the empirical PCA projection errors \( \mathrm{Err}_{\cX} \) and \( \mathrm{Err}_{\cY} \) defined in \eqref{eq: encoding error: general} and \eqref{eq: decoding error: general}, respectively, with \( N_{\cX}, N_{\cY} \geq \log(2/\epsilon) \). Then
    \begin{equation*}
        \textup{Err}_{\cX} \leq L  \sqrt{\sum_{i = d_{\cX} + 1}^{\dim(\cX)} \lambda_i}
        + L K_{\mu} \left( \frac{C_1 d_{\cX} \log(2 / \epsilon)}{N_{\cX}} \right)^{1/4}
    \end{equation*}
    with probability at least \(1 - \epsilon\) in the draw of \(\whX_1, \dots, \whX_{N_{\cX}} \sim_{\textup{i.i.d.}} \mu\), and 
    \begin{equation*}
        \textup{Err}_{\cY} \leq \sqrt{\sum_{i = d_{\cY} + 1}^{\dim(\cY)} \lambda_i^{F \sharp \mu}}
        + (\nm{F(0)}_{\cY} + L  K_{\mu} + \sigma) \left( \frac{C_1 d_{\cY} \log(2 / \epsilon)}{N_{\cY}} \right)^{1/4}
        + 2  \sqrt{2} \sigma
    \end{equation*}
    with probability at least \(1 - \epsilon\) in the draw of \(\whY_1, \dots, \whY_{N_{\cY}} \sim_{\textup{i.i.d.}} \nu\).
    \end{theorem}
    
\begin{proof}
    Recall from Section~\ref{sec: setup} that \(\mu\) and \( \nu \) are sub-Gaussian with parameters \(K_{\mu}\) and \(K_{\nu} = \nm{F(0)}_{\cY} + L K_{\mu} + \sigma\), respectively. The bound for \( \textup{Err}_{\cX} \) then follows from~\eqref{eq: encoding error: general}, Proposition~\ref{prop: PCA error} with \( \tilde{\mu} = \mu \), and~\eqref{eq: encoding projection error}. 
    To derive the bound for \( \textup{Err}_{\cY} \), we first observe that
    \(
    Z \mapsto \nmd{\hDY \circ \hEY(Z) - Z}_{\cY}
    \)
    is a \( 1 \)-Lipschitz mapping \( \cY \to [0, \infty) \). We can thus apply~\cite[Prop. 7.29]{villani_OptimalTransport_2009} to deduce that
    \begin{equation}
    \label{eq: decoding projection error: Wasserstein estimate}
        \nmd{\hDY \circ \hEY - \cI_{\cY}}_{L_{F \sharp \mu}^2(\cY; \cY)} 
        \leq \nmd{\hDY \circ \hEY - \cI_{\cY}}_{L_{\nu}^2(\cY; \cY)} 
        + W_2(F \sharp \mu, \nu),
    \end{equation}
    where \(W_2(F \sharp \mu, \nu)\) denotes the \(2\)-Wasserstein distance between \(F \sharp \mu\) and \(\nu\). 
    It follows from the definition of the latter that
    \begin{equation}
    \label{eq: Wasserstein bound}
        W_2(F \sharp \mu, \nu)^2 \leq \bbE_{(F(X), F(X) + \sigma E)} [\nm{F(X) - (F(X) + \sigma E)}_{\cY}^2] 
        = \bbE [\nm{\sigma E}_{\cY}^2] 
        \leq 2 \sigma^2,
    \end{equation}
    where the last step holds by~\eqref{eq: expectation of E}. 
    The bound for \( \textup{Err}_{\cY} \) now follows from~\eqref{eq: decoding error: general} in combination with~\eqref{eq: decoding projection error: Wasserstein estimate} and~\eqref{eq: Wasserstein bound} after applying Proposition~\ref{prop: PCA error} with \( \tilde{\mu} = \nu \) and Lemma~\ref{lem: true PCA decoding, noisy}.
\end{proof}

\subsection{Error bounds for weighted least-squares approximation}

We now consider the term  \( \mathrm{Err}_{\cA} \) defined in \eqref{eq: approximation error: general}. Since $\whf$ is a solution of a weighted least-squares problem, we develop in this subsection some general results on weighted least-squares approximations to operators. In particular, we show how to choose the sampling measure \( \mu_{\mathrm{samp}} \) in such a way that we obtain near-optimal sample complexity. This is based on the Christoffel function of a suitably chosen approximation space. The resulting~\emph{Christoffel sampling} method was originally developed in~\cite{cohen_OptimalWeightedLeastsquaresMethods_2017} for scalar-valued problems. See also~\cite{adcock_OptimalSamplingLeastsquaresApproximation_2025} for extensions, variations and further topics.

Let \( \varpi \) be a Borel probability measure on \( \cX \). Fix \(s \in \bbN\) and let \(\cP \subset L_{\varpi}^2(\cX) \cap C(\cX)\) be a vector space of dimension \(s\) with orthonormal basis \(\{\xi_i\}_{i = 1}^{s}\). Given a vector subspace \(\cY' \subset \cY\), the \(\cY'\)-lift of \(\cP\) is defined as the vector space
\[
\cP_{\cY'} := \left \{ \sum_{i = 1}^s Y_i \xi_i : Y_i \in \mathcal{Y}' \right \} \subset L_{\varpi}^2(\cX; \cY) \cap C(\cX; \cY).
\]
Next, let \(w : \cX \to (0, \infty)\) be a weight function, which we will specify  later on, such that \(\int_{\cX} w^{-1} \D \varpi = 1\), and set
\[
\D \varpi_{\textup{samp}} := w^{-1} \D \varpi.
\]
Now draw \( X_1,\ldots,X_M \sim_{\textup{i.i.d.}} \varpi_{\textup{samp}} \) and \( E_1,\ldots,E_M \sim_{\textup{i.i.d.}} \rho \), where \( \rho \) satisfies Assumption \ref{assumption : noise model}, and set \( Y_i = F(X_i) + \sigma E_i \) for \( i \in [M] \).
The weighted least-squares approximation of \(F\) in \(\cP_{\cY'}\) is now given by 
\begin{equation}
\label{eq: L-S approximation}
    \whF = \whF((X_i, Y_i)_{i = 1}^M) \in \argmin{P \in \cP_{\cY'}}{\frac{1}{M} \sum_{i = 1}^M w(X_i) \nm{P(X_i) - Y_i}_{\cY}^2}.
\end{equation}
The (reciprocal) Christoffel function of \(\cP\) is defined as
\[
K(\cP): \cX \to \bbR, \qquad K(\cP)(X) := \sup\left\{ \frac{\abs{p(X)}^2}{\nm{p}_{L_{\varpi}^2(\cX)}^2} : p \in \cP, p \neq 0 \right\},
\]
and we set
\[
\kappa_w (\cP) := \nm{w K(\cP)}_{L_{\varpi}^{\infty}(\cX)}.
\]
It can be readily checked that
\begin{equation}
\label{eq: Christoffel function formula}
    K(\cP) = \sum_{i = 1}^s \xi_i^2.
\end{equation}

\begin{theorem}[Least-squares error in probability]
\label{thm: L-S error in probability}
    Let \(0 < \delta, \epsilon < 1\) and
    \begin{equation}
    \label{eq: condition on M}
        M \geq C_{\delta} \kappa_w(\cP) \log(6 s/ \epsilon), \qquad C_{\delta} = ((1 + \delta) \log(1 + \delta) - \delta)^{-1}.
    \end{equation}
    Let \( F \in L^2_{\varpi}(\cX; \cY) \cap C(\cX; \cY) \) and suppose that \( \cY' \) has finite dimension. Then, with probability at least \(1 - \epsilon\) in the draw of \(X_1, \dots, X_M \sim \varpi_{\textup{samp}}\) and \( E_1,\ldots,E_M \sim \rho \), the least-squares approximation \(\whF\) in~\eqref{eq: L-S approximation} uniquely exists and satisfies
    \begin{equation*}
        \nm{F - \whF}_{L_{\varpi}^2(\cX; \cY)} \leq 
        \left( 1 + \sqrt{\frac{3 \kappa_w(\cP)}{M \epsilon}} \frac{1}{1 - \delta} \right) \inf_{P \in \cP_{\cY'}} \nm{F - P}_{L_{\varpi}^2(\cX; \cY)} + \sigma \sqrt{\frac{6 \kappa_w(\cP)}{M \epsilon}} \frac{1}{1 - \delta}.
    \end{equation*}
\end{theorem}

\begin{proof}
    This is a generalization of Corollary 5.9 in~\cite{adcock_OptimalSamplingLeastsquaresApproximation_2025}. Therein, however, the noise term \(\nmd{\widehat{0}_{\bm{E}}}_{L_{\varpi}^2(\cX;\cY)}\), which we introduce below, is bounded differently. The lifting to the Hilbert-valued case is standard and consists of choosing an orthonormal basis of \(\cY'\) and applying Parseval's identity. For brevity, we present details of the lifting argument only in the proof of the bound
    \begin{equation}
    \label{eq: noise term bound in probability}
        \nm{\widehat{0}_{\bm{E}}}_{L_{\varpi}^2(\cX;\cY)} \leq \sigma \sqrt{\frac{6 \kappa_w(\cP)}{M \epsilon}} \frac{1}{1 - \delta}.
    \end{equation}
    
    Before we do this, we derive a suitable upper bound for \( \nmd{F - \whF}_{L_{\varpi}^2(\cX; \cY)} \). To this end, take \(P^* \in \cP_{\cY'}\) such that \(\nm{F - P^*}_{L_{\varpi}^2(\cX;\cY)} = \inf_{P \in \cP_{\cY'}} \nm{F - P}_{L_{\varpi}^2(\cX;\cY)}\) and write \(G := F - P^*\). Note that \(P^*\) uniquely exists because \( \cY'\) and hence \( \cP_{\cY'} \) is finite-dimensional.
    Following the proof in~\cite{adcock_OptimalSamplingLeastsquaresApproximation_2025}, we split the error
    \[
    \nm{F - \whF}_{L_{\varpi}^2(\cX;\cY)} \leq \nm{G}_{L_{\varpi}^2(\cX;\cY)} + \nm{\whG_{\bm{0}}}_{L_{\varpi}^2(\cX;\cY)} + \nm{\widehat{0}_{\bm{E}}}_{L_{\varpi}^2(\cX;\cY)}.
    \]
    Here, \(\whF = \whF_{\bm{E}} \in \cP_{\cY'}\) is the least-squares approximation to \(F\) from the noisy samples \(Y_i = F(X_i) + \sigma E_i\), see~\eqref{eq: L-S approximation}, \(\whG_{\bm{0}} \in \cP_{\cY'}\) is the least-squares approximation to \(G\) from the noiseless samples \( G(X_i) = F(X_i) - P^*(X_i) \), and \(\widehat{0}_{\bm{E}} \in \cP_{\cY'}\) is the least-squares approximation to the zero operator from the noisy samples \(\sigma E_i\). If~\eqref{eq: condition on M} holds, then with probability at least \( 1 - \epsilon / 3 \), the least-squares approximations \(\whF\), \(\whG_{\bm{0}}\), and \(\widehat{0}_{\bm{E}}\) uniquely exist and 
    \[
    \nm{G}_{L_{\varpi}^2(\cX;\cY)} + \nm{\whG_{\bm{0}}}_{L_{\varpi}^2(\cX;\cY)}
    \leq \left( 1 + \sqrt{\frac{3 \kappa_w(\cP)}{M \epsilon}} \frac{1}{1 - \delta} \right) \inf_{P \in \cP_{\cY'}} \nm{F - P}_{L_{\varpi}^2(\cX; \cY)}.
    \]

    We are left with proving that~\eqref{eq: noise term bound in probability} holds with probability at least \(1 - \epsilon / 3\). Then the claim follows after an application of the union bound. 
    We define the weighted (normalized) measurement matrix and noise vector 
    \[
    \bm{A} := \left( \frac{\sqrt{w(X_i)}}{\sqrt{M}} \xi_j(X_i) \right)_{i,j = 1}^{M, s} \in \bbR^{M \times s}, \qquad
    \bm{b} = ( b_i )_{i = 1}^M := \left( \frac{\sqrt{w(X_i)}}{\sqrt{M}} (\sigma E_i) \right)_{i = 1}^M \in \cY^M.
    \]
    Let \(\{\psi_k\}_{k = 1}^{\dim(\cY')}\) be an orthonormal basis of \(\cY'\) and introduce the components \[b_i^{(k)} := \ip{b_i}{\psi_k}_{\cY}, \quad
    e_i^{(k)} := \ip{E_i}{\psi_k}_{\cY}, \quad
    \widehat{0}_{\bm{E}}^{(k)} := \ip{\widehat{0}_{\bm{E}}}{\psi_k}_{\cY}.
    \]
    Note that \(\widehat{0}_{\bm{E}}^{(k)} \in \cP\) and therefore \(\widehat{0}_{\bm{E}}^{(k)} = \sum_{j = 1}^s c_j^{(k)} \xi_j\) for some coefficients \(\bm{c^{(k)}} = (c_j^{(k)})_{j = 1}^s \in \bbR^s\).
    Again following the argument in~\cite{adcock_OptimalSamplingLeastsquaresApproximation_2025} (more specifically, equations (A.3) therein) and applying Parseval's identity gives
    \begin{align*}
        \nm{\widehat{0}_{\bm{E}}}_{L_{\varpi}^2(\cX; \cY)}^2 
        &\leq \frac{1}{(1 - \delta) M} \sum_{i = 1}^M w(X_i) \nm{\widehat{0}_{\bm{E}}(X_i)}_{\cY}^2 
        = \frac{1}{(1 - \delta) M} \sum_{k = 1}^{\dim(\cY)} \sum_{i = 1}^M w(X_i) \abs{\widehat{0}_{\bm{E}}^{(k)}(X_i)}^2 \\
        &= \frac{1}{1 - \delta} \sum_{k = 1}^{\dim(\cY)} \nm{\bm{A} \bm{c^{(k)}}}_2^2
    \end{align*}
    with probability at least \( 1 - \epsilon / 3 \).
    It can be readily checked that \(\bm{c^{(k)}}\) is a solution to the normal equations \(\bm{A}^{\top} \bm{A} \bm{c^{(k)}} = \bm{A}^{\top} \bm{b^{(k)}}\) with \(\bm{b^{(k)}} = (b_i^{(k)})_{i = 1}^M \in \bbR^M\). Hence, 
    \begin{align*}
        \sum_{k = 1}^{\dim(\cY)} \nm{\bm{A} \bm{c^{(k)}}}_2^2 &\leq \sum_{k = 1}^{\dim(\cY)} \nm{\bm{c^{(k)}}}_2 \nm{\bm{A}^{\top} \bm{b^{(k)}}}_2
        \leq \left( \sum_{k = 1}^{\dim(\cY)} \nm{\bm{c^{(k)}}}_2^2 \right)^{1 / 2} \left( \sum_{k = 1}^{\dim(\cY)} \nm{\bm{A}^{\top} \bm{b^{(k)}}}_2^2 \right)^{1/2} \\
        &= \nm{\widehat{0}_{\bm{E}}}_{L_{\varpi}^2(\cX; \cY)} \left( \sum_{k = 1}^{\dim(\cY)} \nm{\bm{A}^{\top} \bm{b^{(k)}}}_2^2 \right)^{1/2},
    \end{align*}
    and consequently, with probability at least \( 1 - \epsilon / 3 \),
    \[
    \nm{\widehat{0}_{\bm{E}}}_{L_{\varpi}^2(\cX; \cY)}^2 \leq \frac{1}{(1 - \delta)^2} \sum_{k = 1}^{\dim(\cY)}  \nm{\bm{A}^{\top} \bm{b^{(k)}}}_2^2.
    \]
    We expand
    \begin{align*}
        \nm{\bm{A}^{\top} \bm{b^{(k)}}}_2^2
        &= \sum_{j = 1}^s \sum_{i, \ell = 1}^M \sigma^2 e_i^{(k)} e_{\ell}^{(k)} \frac{w(X_i)}{M} \frac{w(X_{\ell})}{M} \xi_j(X_i) \xi_j(X_{\ell}) 
    \end{align*}
    and take the expectation with respect to the \(E_i\), which are centered and i.i.d., to deduce that
    \[
    \bbE_{\bm{E}} \left[ \nm{\bm{A}^{\top} \bm{b^{(k)}}}_2^2 \right] 
    = \sum_{j = 1}^s \sum_{i = 1}^M \sigma^2 \bbE[(e_1^{(k)})^2] \frac{w(X_i)^2}{M^2} \xi_j(X_i)^2.
    \]
    Summing over \(k\) yields
    \[
    \bbE_{\bm{E}} \left[ \sum_{k = 1}^{\dim(\cY)}  \nm{\bm{A}^{\top} \bm{b^{(k)}}}_2^2 \right] 
    \leq \sum_{j = 1}^s \sum_{i = 1}^M \sigma^2 \bbE[\nm{E_1}_{\cY}^2] \frac{w(X_i)^2}{M^2} \xi_j(X_i)^2
    \leq 2 \sigma^2 \kappa_{w}(\cP) \sum_{i = 1}^M \frac{w(X_i)}{M^2}.
    \]
    In the last step, we used~\eqref{eq: Christoffel function formula} as well as \( \bbE[\nm{E_1}_{\cY}^2] \leq 2 \), which holds by~\eqref{eq: expectation of E}. Taking the expectation over the \( X_i \) implies 
    \[
    \bbE_{\bm{E},\bm{X}} \left[ \sum_{k = 1}^{\dim(\cY)}  \nm{\bm{A}^{\top} \bm{b^{(k)}}}_2^2 \right]
    \leq 2 \sigma^2 \frac{\kappa_{w}(\cP)}{M}.
    \]
    By Markov's inequality, we have
    \[
    \frac{1}{(1 - \delta)^2} \sum_{k = 1}^{\dim(\cY)}  \nm{\bm{A}^{\top} \bm{b^{(k)}}}_2^2
    \leq \frac{2 \sigma^2}{(1 - \delta)^2} \frac{3 \kappa_{w}(\cP)}{M \epsilon}
    \]
    with probability at least \( 1 - \epsilon / 3 \).
    The claim now follows from the union bound.
\end{proof}

\subsection{Application to the Hermite-PCA approximation}

We now use Theorem \ref{thm: L-S error in probability} to bound \( \mathrm{Err}_{\cA} \). Let \(S \subseteq \bbN^{d_{\cX}}_0\) be an arbitrary index set with \( \bm{0} \in S \) of size \( |S| = s \). Then let
\begin{align*}
    \varpi & = \hDX \sharp\hat{\varrho}, & \cP &= \spn\left\{ H_{\bm{\gamma}, \bm{\hat{\lambda}}} \circ \hEX: \bm{\gamma} \in S \right\} \\
    \cY' &= \hDY(\bbR^{d_{\cY}}), &
    w &= \left( \frac{1}{s} \sum_{\bm{\gamma} \in S} (H_{\bm{\gamma}, \bm{\hat{\lambda}}} \circ \hEX)^2 \right)^{-1}.
\end{align*}
Note that these choices depend on the PCA data \( \whX_i, \whY_i \). Also note that \( \cY' \) is finite-dimensional, as required in Theorem~\ref{thm: L-S error in probability}, and \( w > 0 \). We now consider the weighted least-squares approximation
\begin{equation}
\label{eq: least-squares for hat-F}
    \whF \in \argmin{P \in \cP_{\cY'}}{\frac1M \sum^{M}_{i=1} w(X_i) \nm{P(X_i) - \hDY \circ \hEY(Y_i)}_{\cY}^2}.
\end{equation}
Notice that this formulation uses the projected samples \(  \hDY \circ \hEY(Y_i) \). This is crucial, in that it allows us to relate \( \whF \) to the latent space approximation \( \whf \).

\begin{lemma}
    A function \( \whf : \bbR^{d_{\cX}} \rightarrow \bbR^{d_{\cY}} \) is a solution of \eqref{eq: f-hat def} if and only if the function \( \whF = \hDY \circ \whf \circ \hEX : \cX \rightarrow \cY \) is a solution to \eqref{eq: least-squares for hat-F}.
\end{lemma}
\begin{proof}
    Observe that we can write any \( P \in \cP_{\cY'} \) as \( P = \hDY \circ p \circ \hEX \) for
    \(
    p \in \left\{ \sum_{\bm{\gamma} \in S} \bm{c_{\gamma}} H_{\bm{\gamma}, \bm{\hat{\lambda}}} : \bm{c_{\gamma}} \in \bbR^{d_{\cY}} \right\}
    \)
    and vice versa. Therefore
    \begin{align*}
        \nm{P(X_i) - \hDY \circ \hEY(Y_i)}_{\cY} & = \nm{\hDY \circ p \circ \hEX(X_i)  - \hDY \circ \hEY(Y_i)}_{\cY} = \nm{p(\hEX(X_i))  -  \hEY(Y_i)}_{2}.
    \end{align*}
    Furthermore, by the change of variables $\bm{x} = \hEX(X)$ and the fact that \( \hEX \circ \hDX = \cI_{\bbR^{d_{\cX}}} \) is the identity on \( \bbR^{d_{\cX}} \), we have
    \begin{align*}
        \D \varpi_{\mathrm{samp}}(X) &= w(X)^{-1} \D \varpi(X)
        =  \left( \frac{1}{s} \sum_{\bm{\gamma} \in S} H_{\bm{\gamma}, \bm{\hat{\lambda}}}(\hEX(X))^2  \right) \D \hDX \sharp\hat{\varrho}(X)
        \\
        & =  \left( \frac{1}{s} \sum_{\bm{\gamma} \in S} H_{\bm{\gamma}, \bm{\hat{\lambda}}}(\bm{x} )^2 \right ) \D\hat{\varrho}(\bm{x})
         = \D \varrho_{\mathrm{samp}}(\bm{x}), 
    \end{align*}
    where \( \varrho_{\mathrm{samp}} \) is as in \eqref{eq: upsilon-samp def}. The result now follows.
\end{proof}

We now want to apply Theorem \ref{thm: L-S error in probability}. For this, we require two observations. First, notice that
\[
\hDY \circ \hEY(Y_i) = \hDY \circ \hEY \circ F(X_i) + \sigma \hDY \circ \hEY(E_i)
\]
and the noise terms \( \hDY \circ \hEY(E_i) \) are sub-Gaussian with parameter \( K = 1 \) since the random variables \(E_i\) are sub-Gaussian with \( K_{\rho} = 1 \) and the map \( \hDY \circ \hEY \) is an orthogonal projection. Second, observe that the functions \( \{ H_{\bm{\gamma} , \bm{\hat{\lambda}} } \circ \hEX \}_{\bm{\gamma} \in S } \) are an orthonormal basis for \( \cP \) in \( L^2_{\varpi}(\cX) \), where we recall that \( \varpi = \hDX \sharp\hat{\varrho} \). Indeed, using once more the fact that \( \hEX \circ \hDX = \cI_{\bbR^{d_{\cX}}} \), we have
\begin{align*}
    \int_{\cX} H_{\bm{\gamma} , \bm{\hat{\lambda}} } \circ \hEX (X) \, H_{\bm{\gamma'} , \bm{\hat{\lambda}} } \circ \hEX(X) \D \hDX \sharp\hat{\varrho}(X) = \int_{\bbR^{d_{\cX}}} H_{\bm{\gamma} , \bm{\hat{\lambda}} } (\bm{x}) H_{\bm{\gamma'} , \bm{\hat{\lambda}} } (\bm{x}) \D\hat{\varrho}(\bm{x}) = \delta_{\bm{\gamma},\bm{\gamma'}}.
\end{align*}
Therefore \( w^{-1} \) is, up to the scalar multiple \( s \), precisely the reciprocal Christoffel function of \( \cP \). This implies that \( \kappa_{w}(\cP) = s \). Therefore, Theorem \ref{thm: L-S error in probability} with \(F\) replaced by \(\hDY \circ \hEY \circ F\) yields
\begin{align*}
    \nm{\hDY \circ \hEY \circ F - \whF }_{L^2_{\varpi}(\cX ; \cY)} \leq & \left ( 1 + \sqrt{\frac{3 s}{M \epsilon}} \frac{1}{1-\delta} \right ) \inf_{P \in \cP_{\cY'}} \nm{\hDY \circ \hEY \circ F - P }_{L^2_{\varpi}(\cX ; \cY)} 
    + \sigma \sqrt{\frac{6 s}{M \epsilon} } \frac{1}{1-\delta}
\end{align*}
with probability at least $1-\epsilon$, provided
\(
M \geq C_{\delta} \, s \log(6 s / \epsilon).
\)
To relate this to the error \( \mathrm{Err}_{\cA} \), we use the fact that \( \hEX \circ \hDX = \cI_{\bbR^{d_{\cX}}} \) for a third time and write
\begin{align*}
    \nm{\hDY \circ \hEY \circ F - \whF }_{L^2_{\varpi}(\cX ; \cY)} & = \nm{\hDY \circ \hEY \circ F \circ \hDX - \whF   \circ \hDX }_{L^2_{\hat{\varrho}}(\bbR^{d_{\cX}} ; \cY)}
    \\
    & = \nm{\hDY \circ \hEY \circ F \circ \hDX \circ \hEX \circ \hDX - \whF   \circ \hDX }_{L^2_{\hat{\varrho}}(\bbR^{d_{\cX}} ; \cY)}
    \\
    & = \nm{\hDY \circ \hEY \circ F \circ \hDX \circ \hEX - \whF }_{L^2_{\varpi}(\cX ; \cY)}. 
\end{align*}
Thus we get 
\begin{equation}
\label{eq: least squares error bound}
\begin{split}
    &\nm{\hDY \circ \hEY \circ F  \circ \hDX \circ \hEX  - \whF }_{L^2_{\varpi}(\cX ; \cY)} \\
    &\hspace{8ex} \leq \left ( 1 + \sqrt{\frac{3 s}{M \epsilon}} \frac{1}{1-\delta} \right ) \inf_{P \in \cP_{\cY'}} \nm{\hDY \circ \hEY \circ F - P }_{L^2_{\varpi}(\cX ; \cY)} 
    + \sigma \sqrt{\frac{6 s}{M \epsilon} } \frac{1}{1-\delta}
\end{split}
\end{equation}
with probability at least $1-\epsilon$ (conditionally on the PCA samples and labels \( \whX_i, \whY_i \)), provided \( M \geq C_{\delta} \, s \log(6 s / \epsilon) \). However, there remains an issue. The left-hand side is close to the approximation error \( \mathrm{Err}_{\cA} \) defined in \eqref{eq: approximation error: general}, but the measure is \( \varpi = \hDX \sharp\hat{\varrho} \) as opposed to \( \mu \). In the next subsection, we perform the necessary switch of measure.

\subsection{\texorpdfstring{Bounding the \( L^2_{\mu}(\cX ; \cY) \)-norm error in terms of the  \( L^2_{\varpi}(\cX ; \cY) \)-norm error}{Bounding the L2\_{mu}(X;Y)-norm error in terms of the L2\_{varpi}(X;Y)-norm error}}

We now show how a change of measure from \( \mu \) to \( \varpi \) affects the overall error bound. The following lemma assumes closeness of the true PCA eigenvalues \( \lambda_i \) and their approximate counterparts \( \hat{\lambda}_i \). Later, we quantify the error between the two, which leads to Theorem \ref{thm: effect of measure mu to varpi full}, the main result of this subsection. In the following, \( \nm{ \cdot }_{\infty} \) denotes the operator norm and \(\Gamma(s,x) = \int_{x}^{\infty} t^{s-1} e^{-t} \D t \) is the upper incomplete Gamma function.

\begin{lemma}
\label{lem: effect of measure mu to varpi}
    Let  \(\kappa := \nmd{\Smu - \hSmu}_{\infty} / \lambda_{d_{\cX}}\) and suppose that \(\kappa < 1 / 2\). Let \( F \) be \( L \)-Lipschitz and let \( P = \hDY \circ p \circ \hEX \) for some \( p \in \cP_{\bbR^{d_{\cY}}} \), where \( \cP_{\bbR^{d_{\cY}}} \) is given by \eqref{eq: hat-f subspace def}. Then, for every $R \geq 1$ and every realization of \( \hEX, \hDX, \hEY, \hDY \), we have
    \begin{align*}
        &  \nmd{P - \hDY \circ \hEY \circ F \circ \hDX \circ \hEX}_{L_{\mu}^2(\cX; \cY)}
        \\  
        & \hspace{8ex} \leq  \left( \frac{1 - \kappa}{1 - 2\kappa} \right)^{d_{\cX} / 4}
        \exp \left( \frac{R^2}{4 \lambda_{d_{\cX}}} \frac{\kappa}{(1 - \kappa)(1 - 2\kappa)} \right) \bar{\Upsilon}+ 2 ( C_F + L R ) e^{-R^2/16} \\
        &\hspace{8ex} \quad + 2 e^{1/16} \sqrt{s} \, \left(\frac{216}{\lambda_{d_{\cX}}}\right)^{m(S)/2} \sqrt{\Gamma\left(m(S) + 1, \frac{R^2}{8} \right)}
        \left( \bar{\Upsilon} + C_F \right),
    \end{align*}
    where \( m(S) := \max_{\bm{\gamma} \in S} \nm{\bm{\gamma}}_1 \), \( C_F  := \nm{F(0)}_{\cY} + \sqrt{8} L \), and
    \[
    \bar{\Upsilon}: =  \nmd{P - \hDY \circ \hEY \circ F \circ \hDX \circ \hEX}_{L_{\varpi}^2(\cX; \cY)} = \nmd{p - \hEY \circ F \circ \hDX }_{L^2_{\hat{\varrho}}(\bbR^{d_{\cX}} ; \bbR^{d_{\cY}}) }.
    \]
\end{lemma}

\begin{proof}
    We first note that
    \begin{equation*}
      \nmd{P - \hDY \circ \hEY \circ F \circ \hDX \circ \hEX}_{L_{\mu}^2(\cX; \cY)}  
      = \nmd{p- \hEY \circ F \circ \hDX }_{L_{\bar{\varrho}}^2(\bbR^{d_{\cX}}; \bbR^{d_{\cY}})},
    \end{equation*}
    where 
    \[
    \bar{\varrho} := \hEX \sharp \mu.
    \]
    We will argue below that \( \bar{\varrho} \) is a nondegenerate Gaussian measure on \( \bbR^{d_{\cX}} \), as is \( \hat{\varrho}\). Hence, the two measures are equivalent and by the Radon--Nikodym theorem, the Radon--Nikodym derivative \( \D \bar{\varrho} / \D \hat{\varrho} \) exists and is a \(\hat{\varrho}\)-integrable function. The main idea of the proof is to split the total error into a localized part (on some ball of finite radius) and tail contributions (in the complement). On the ball, we can uniformly bound \( \D \bar{\varrho} / \D \hat{\varrho} \). The remainder terms can be controlled by the Gaussian tail bound~\eqref{eq: tail bound for mu} and pointwise bounds for Hermite polynomials.
    
    To start, consider the centered ball \( B_R := \{\bm{x} \in \bbR^{d_{\cX}} : \nm{\bm{x}}_2 \leq R\} \) in \( \bbR^{d_{\cX}} \) of radius \(R > 0\). Let \( \ind{B_R} \) denote the corresponding indicator function and define \( p_R := p \ind{B_R} \).
    We split the error as just described:
    \begin{align*}
        &  \nmd{P - \hDY \circ \hEY \circ F \circ \hDX \circ \hEX}_{L_{\mu}^2(\cX; \cY)}  \nonumber \\
        & \hspace{8ex} \leq \nm{\hEY \circ F \circ \hDX - p_R}_{L_{\bar{\varrho}}^2(\bbR^{d_{\cX}}; \bbR^{d_{\cY}})}
        + \nm{p_R - p}_{L_{\bar{\varrho}}^2(\bbR^{d_{\cX}}; \bbR^{d_{\cY}})}  \\
        & \hspace{8ex} \leq \left( \int_{B_R} \nm{\hEY \circ F \circ \hDX - p}_2^2 \frac{\D \bar{\varrho}}{\D\hat{\varrho} } \D\hat{\varrho} \right)^{1/2}
        + \nm{\hEY \circ F \circ \hDX}_{L_{\bar{\varrho}}^2(\bbR^{d_{\cX}} \setminus B_R; \bbR^{d_{\cY}})} \nonumber \\
        & \hspace{8ex} \quad + \nm{p_R - p}_{L_{\bar{\varrho}}^2(\bbR^{d_{\cX}}; \bbR^{d_{\cY}})}.
    \end{align*}
    Consequently, 
    \begin{equation*}
    \begin{split}
        & \nmd{P - \hDY \circ \hEY \circ F \circ \hDX \circ \hEX}_{L_{\mu}^2(\cX; \cY)} 
        \\
        &\hspace{8ex} \leq \nm{\frac{\D \bar{\varrho}}{\D\hat{\varrho}}}_{L^{\infty}(B_R)}^{1/2}  \bar{\Upsilon}
        + \nm{\hEY \circ F \circ \hDX}_{L_{\bar{\varrho}}^2(\bbR^{d_{\cX}} \setminus B_R; \bbR^{d_{\cY}})}  + \nm{p_R - p}_{L_{\bar{\varrho}}^2(\bbR^{d_{\cX}}; \bbR^{d_{\cY}})}.
    \end{split}
    \end{equation*}
    In the remainder of the proof, we estimate the three terms on the right-hand side of this inequality in three separate steps.
    
    \textbf{Step 1: Bound \(\nm{\frac{\D \bar{\varrho}}{\D \hat{\varrho}}}_{L^{\infty}(B_R)}\).} 
    We commence by defining
    \begin{equation}
    \label{eq: Dmu and kappa}
        \Dmu := \Smu - \hSmu, \quad \textup{so that } \quad \kappa = \frac{\nm{\Dmu}_{\infty}}{\lambda_{d_{\cX}}}.
    \end{equation}
    Note that \(\Dmu\) and hence \(\kappa\) depend on the (unlabeled) PCA samples \( \whX_1, \dots, \whX_{N_{\cX}} \). We suppress this dependence in the notation. Suppose that \(\kappa < 1 / 2\). By Weyl's inequality, we have
    \begin{equation}
    \label{eq: Weyl's inequality}
        \sup_{i= 1, \dots, \dim(\cX)} \absd{\lambda_i - \hat{\lambda}_i} \leq \nmd{\Dmu}_{\infty},
    \end{equation}
    See e.g.~\cite[Cor. 4.10]{stewart_MatrixPerturbationTheory_1990}, which is formulated for Hermitian matrices but can be readily generalized to compact self-adjoint operators on a separable Hilbert space. We deduce that 
    \begin{equation}
    \label{eq: hat(lambda) estimate}
        \hat{\lambda}_{d_{\cX}} \geq \lambda_{d_{\cX}} - \nm{\Dmu}_{\infty} = (1 - \kappa) \lambda_{d_{\cX}} > 0.
    \end{equation}
    This implies that \( \hat{\varrho} \) is a nondegenerate Gaussian measure on \( \bbR^{d_{\cX}} \).
    
    Next, we study the relation between the two measures \(\bar{\varrho}\) and \(\hat{\varrho}\). To this end, let us represent the covariance operators \(\Sigma_{\bar{\varrho}}\) and \(\Sigma_{\hat{\varrho}}\) as matrices with respect to the standard basis of \(\bbR^{d_{\cX}}\). By a slight abuse of notation, we denote the matrix representations again by \(\Sigma_{\bar{\varrho}}\) and \(\Sigma_{\hat{\varrho}}\), and find
    \begin{equation}
    \label{eq: structure of Sigma_(mu_1) and Sigma_(mu_2)}
        (\Sigma_{\bar{\varrho}})_{ij} = \ip{\widehat{\phi}_j}{\Smu \widehat{\phi}_i}_{\cX} \quad \textup{ and } \quad 
        (\Sigma_{\hat{\varrho}})_{ij} = \ip{\widehat{\phi}_j}{\hSmu \widehat{\phi}_i}_{\cX} = \hat{\lambda}_i \delta_{ij}, \quad i, j = 1, \dots, d_{\cX}.
    \end{equation} 
    Observe that \( \bm{x}^{\top} \Sigma_{\bar{\varrho}} \, \bm{x} =  \ip{\hDX(\bm{x})}{\Sigma \, \hDX(\bm{x})}_{\cX} > 0 \) for \( \bm{x} \in \bbR^{d_{\cX}} \), \( \bm{x} \neq \bm{0} \), because \( \mu \) is nondegenerate. This shows that \( \bar{\varrho} \) is nondegenerate.
    
    The Radon--Nikodym derivative \( \D \bar{\varrho} / \D \hat{\varrho}\) is given by the quotient of the corresponding Gaussian densities, 
    \begin{equation}
    \label{eq: Radon--Nikodym derivative}
        \frac{\D \bar{\varrho}}{\D \hat{\varrho}}(\bm{x}) 
        = \left( \frac{\det \Sigma_{\hat{\varrho}}}{\det \Sigma_{\bar{\varrho}}} \right)^{1/2}
        \exp\left( - \frac{1}{2}\ip{\bm{x}}{(\Sigma_{\bar{\varrho}}^{-1} - \Sigma_{\hat{\varrho}}^{-1}) \bm{x}}_2 \right), \quad \forall \bm{x} \in \bbR^{d_{\cX}}.
    \end{equation}
    We bound the exponential and the quotient of determinants in~\eqref{eq: Radon--Nikodym derivative} separately, starting with the former.
    By the Cauchy--Schwarz inequality, we have
    \[
    \abs{\ipd{\bm{x}}{(\Sigma_{\bar{\varrho}}^{-1} - \Sigma_{\hat{\varrho}}^{-1}) \bm{x}}_2} \leq \nm{\Sigma_{\bar{\varrho}}^{-1} - \Sigma_{\hat{\varrho}}^{-1}}_{\infty} \nm{\bm{x}}_2^2, \qquad \forall \bm{x} \in \bbR^{d_{\cX}}.
    \]
    Using the identity
    \(
    \Sigma_{\bar{\varrho}}^{-1} - \Sigma_{\hat{\varrho}}^{-1} 
    = \Sigma_{\bar{\varrho}}^{-1} (\Sigma_{\hat{\varrho}} - \Sigma_{\bar{\varrho}}) \Sigma_{\hat{\varrho}}^{-1},
    \)
    we can bound
    \begin{equation}
    \label{eq: bound for Sigma_(mu_1)^(-1) - Sigma_(mu_2)^(-1)}
        \nmd{\Sigma_{\bar{\varrho}}^{-1} - \Sigma_{\hat{\varrho}}^{-1}}_{\infty}
        \leq \nmd{\Sigma_{\bar{\varrho}}^{-1}}_{\infty} \nmd{\Sigma_{\bar{\varrho}} - \Sigma_{\hat{\varrho}}}_{\infty} \nmd{\Sigma_{\hat{\varrho}}^{-1}}_{\infty}.
    \end{equation}
    We know that \(\nmd{\Sigma_{\hat{\varrho}}^{-1}}_{\infty} = \hat{\lambda}_{d_{\cX}}^{-1}\), and it readily follows from~\eqref{eq: structure of Sigma_(mu_1) and Sigma_(mu_2)} that \(\nm{\Sigma_{\hat{\varrho}} - \Sigma_{\bar{\varrho}}}_{\infty} \leq \nm{\Delta}_{\infty}\).
    The term \(\nmd{\Sigma_{\bar{\varrho}}^{-1}}_{\infty}\) can be bounded using a Neumann series representation. For this, let \(I_{d_{\cX}}\) denote the identity matrix in \(\bbR^{d_{\cX} \times d_{\cX}}\). Writing \(I_{d_{\cX}} - \Sigma_{\bar{\varrho}} \Sigma_{\hat{\varrho}}^{-1} = (\Sigma_{\hat{\varrho}} - \Sigma_{\bar{\varrho}}) \Sigma_{\hat{\varrho}}^{-1} \) yields
    \[
    \nm{I_{d_{\cX}} - \Sigma_{\bar{\varrho}} \Sigma_{\hat{\varrho}}^{-1}}_{\infty}
    \leq \nm{\Sigma_{\hat{\varrho}} - \Sigma_{\bar{\varrho}}}_{\infty} \nm{\Sigma_{\hat{\varrho}}^{-1}}_{\infty}
    \leq \frac{\nmd{\Delta}_{\infty}}{\hat{\lambda}_{d_{\cX}}}.
    \]
    By~\eqref{eq: hat(lambda) estimate} and \(\kappa < 1/2\), the right-hand side is less than one, which yields the convergent Neumann series representation
    \[
    \Sigma_{\hat{\varrho}} \Sigma_{\bar{\varrho}}^{-1} 
    = (\Sigma_{\bar{\varrho}} \Sigma_{\hat{\varrho}}^{-1})^{-1} 
    = \sum_{k = 0}^{\infty} (I_{d_{\cX}} - \Sigma_{\bar{\varrho}} \Sigma_{\hat{\varrho}}^{-1})^k.
    \]
    Consequently,
    \begin{equation}
    \label{eq: bound for Sigma_(mu_1)^(-1)}
        \nmd{\Sigma_{\bar{\varrho}}^{-1}}_{\infty} 
        \leq \nmd{\Sigma_{\hat{\varrho}}^{-1}}_{\infty} \nmd{\Sigma_{\hat{\varrho}} \Sigma_{\bar{\varrho}}^{-1}}_{\infty} 
        \leq \frac{\hat{\lambda}_{d_{\cX}}^{-1}}{1 - \nmd{I_{d_{\cX}} - \Sigma_{\bar{\varrho}} \Sigma_{\hat{\varrho}}^{-1}}_{\infty}}
        \leq \frac{1}{\hat{\lambda}_{d_{\cX}} - \nm{\Delta}_{\infty}}.
    \end{equation}
    Combining~\eqref{eq: bound for Sigma_(mu_1)^(-1) - Sigma_(mu_2)^(-1)} with~\eqref{eq: bound for Sigma_(mu_1)^(-1)} and~\eqref{eq: hat(lambda) estimate}, we conclude
    \begin{equation}
    \label{eq: bound for difference of inverse covariance matrices}
        \nmd{\Sigma_{\bar{\varrho}}^{-1} - \Sigma_{\hat{\varrho}}^{-1}}_{\infty} 
        \leq \frac{\nmd{\Dmu}_{\infty}}{\hat{\lambda}_{d_{\cX}} \left(\hat{\lambda}_{d_{\cX}} - \nmd{\Dmu}_{\infty} \right)}
        \leq \frac{\nm{\Dmu}_{\infty}}{(\lambda_{d_{\cX}} - \nm{\Dmu}_{\infty})\left(\lambda_{d_{\cX}} - 2\nm{\Dmu}_{\infty} \right)}.
    \end{equation}
    
    Next, we bound the quotient of determinants in~\eqref{eq: Radon--Nikodym derivative}. From~\cite[Cor. 2.14]{ipsen_PerturbationBoundsDeterminantsCharacteristic_2008} it follows that
    \[
    \abs{\frac{\det(\Sigma_{\hat{\varrho}})}{\det(\Sigma_{\bar{\varrho}})} - 1}
    \leq \left( \frac{\nmd{\Sigma_{\hat{\varrho}} - \Sigma_{\bar{\varrho}}}_{\infty}}{\bar{\lambda}_{d_{\cX}}} + 1 \right)^{d_{\cX}} - 1,
    \]
    where \( \bar{\lambda}_{d_{\cX}} \) is the smallest eigenvalue of \( \Sigma_{\bar{\varrho}} \).
    Using the inequalities \(\nmd{\Sigma_{\hat{\varrho}} - \Sigma_{\bar{\varrho}}}_{\infty} \leq \nmd{\Delta}_{\infty}\) and \(\hat{\lambda}_{d_{\cX}} \geq \lambda_{d_{\cX}} - \nmd{\Delta}_{\infty}\) once again, we find 
    \[
    \bar{\lambda}_{d_{\cX}} \geq \hat{\lambda}_{d_{\cX}} - \nmd{\Sigma_{\hat{\varrho}} - \Sigma_{\bar{\varrho}}}_{\infty} \geq \lambda_{d_{\cX}} - 2 \nm{\Delta}_{\infty} 
    = \lambda_{d_{\cX}} (1 - 2 \kappa), 
    \]
    and consequently,
    \[
    \abs{\frac{\det(\Sigma_{\hat{\varrho}})}{\det(\Sigma_{\bar{\varrho}})}}
    \leq \left( \frac{\nmd{\Dmu}_{\infty}}{\lambda_{d_{\cX}} - 2 \nm{\Dmu}_{\infty}} + 1 \right)^{d_{\cX}}.
    \]
    Combining this with~\eqref{eq: bound for difference of inverse covariance matrices} finally yields
    \begin{align*}
    \begin{split}
        \nm{\frac{\D \bar{\varrho}}{\D \hat{\varrho}}}_{L^{\infty}(B_R)}
        &\leq \left( \frac{\nmd{\Dmu}_{\infty}}{\lambda_{d_{\cX}} - 2 \nm{\Dmu}_{\infty}} + 1 \right)^{d_{\cX} / 2}
        \exp \left( \frac{R^2}{2} \frac{\nm{\Dmu}_{\infty}}{(\lambda_{d_{\cX}} - \nm{\Dmu}_{\infty})(\lambda_{d_{\cX}} - 2\nm{\Dmu}_{\infty})} \right) \\
        &= \left( \frac{1 - \kappa}{1 - 2 \kappa} \right)^{d_{\cX} / 2}
        \exp \left( \frac{R^2}{2 \lambda_{d_{\cX}}} \frac{\kappa}{(1 - \kappa)(1 - 2\kappa)} \right).
    \end{split}
    \end{align*}

    \textbf{Step 2: Bound \(\nmd{\hEY \circ F \circ \hDX}_{L_{\bar{\varrho}}^2(\bbR^{d_{\cX}} \setminus B_R; \bbR^{d_{\cY}})}\).}
    Since \(\hEY \circ F : \cX \to \bbR^{d_{\cY}} \) is \(L\)-Lipschitz, we have
    \(
    \nmd{\hEY \circ F(X)}_2 \leq \nm{F(0)}_{\cY} + L  \nm{X}_{\cX}
    \)
    for every \( X \in \cX \). Recalling that \(\bar{\varrho} = \hEX \sharp \mu\) and writing 
    \(A_R := (\hEX)^{-1}(\bbR^{d_{\cX}} \setminus B_R)\), we find that
    \begin{align*}
        \nm{\hEY \circ F \circ \hDX}_{L_{\bar{\varrho}}^2(\bbR^{d_{\cX}} \setminus B_R; \bbR^{d_{\cY}})}
        &= \left( \int_{A_R} \nm{\hEY \circ F \circ \hDX \circ \hEX(X)}_2^2 \D \mu(X) \right)^{1/2} \\
        &\leq \nm{F(0)}_{\cY} \sqrt{\mu(A_R)}  +  L  \left( \int_{A_R} \nm{X}_{\cX}^2 \D \mu(X) \right)^{1/2},
    \end{align*}
    where we used in the last step that \( \hDX \circ \hEX \) is a projection in \( \cX \). Denote by~\(B_R^{\cX} := \{X \in \cX : \nm{X}_{\cX} \leq R \}\) the centered \(R\)-ball in \(\cX\) and observe that
    \[
    A_R
    = \left\{ X \in \cX : \sum_{i = 1}^{d_{\cX}} \ipd{X}{\widehat{\phi}_i}_{\cX}^2 > R^2 \right\} \\
    \subset  \{X \in \cX : \nm{X}_{\cX} > R \} 
    = \cX \setminus B_R^{\cX}.
    \]
    The Gaussian tail bound~\eqref{eq: tail bound for mu} implies
    \(
    \mu(A_R)
    \leq 4 e^{-R^2 / 8}
    \)
    as well as 
    \begin{align*}
        \int_{A_R} \nm{X}_{\cX}^2 \D \mu(X)
        &\leq \int_{\cX \setminus B_R^{\cX}} \nm{X}_{\cX}^2 \D \mu(X)
        \leq R^2 \mu(\nm{X}_{\cX} > R) + 2 \int_{R}^{\infty} t \mu(\nm{X}_{\cX} > t) \D t \\
        &\leq 4 R^2  e^{-R^2 / 8} + 2 \int_{R}^{\infty} 4 t e^{-t^2 / 8} \D t 
        = (4 R^2 + 32) e^{-R^2 / 8}.
    \end{align*}
    In conclusion, 
    \begin{equation*}
        \nm{\hEY \circ F \circ \hDX}_{L_{\bar{\varrho}}^2(\bbR^{d_{\cX}} \setminus B_R; \bbR^{d_{\cY}})}
        \leq 2 ( C_F + L R ) e^{-R^2 / 16}.
    \end{equation*}
    
    \newpage
    \textbf{Step 3: Bound \(\nmd{p_R - p}_{L_{\bar{\varrho}}^2(\bbR^{d_{\cX}}; \bbR^{d_{\cY}})}\).}
    Write \( p = \sum_{\bm{\gamma} \in S} \bm{c}_{\bm{\gamma}} H_{\bm{\gamma}, \bm{\hat{\lambda}} } \) with coefficients \( \bm{c_{\gamma}} \in \bbR^{d_{\cY}}\). By~the Cauchy--Schwarz inequality and Parseval's identity, we infer
    \begin{align}
    \begin{split}
    \label{eq: p-R - p}
        \nm{p_R - p}_{L_{\bar{\varrho}}^2(\bbR^{d_{\cX}}; \bbR^{d_{\cY}})}
        &= \nm{ \sum_{\bm{\gamma} \in S } \bm{c}_{\bm{\gamma}} H_{\bm{\gamma}, \bm{\hat{\lambda}} }}_{L_{\bar{\varrho}}^2(\bbR^{d_{\cX}} \setminus B_R; \bbR^{d_{\cY}})} \\
        &\leq s^{1/2} \left( \sum_{\bm{\gamma} \in S} \nm{\bm{c}_{\bm{\gamma}} }^2_2 \right)^{1/2} \max_{\bm{\gamma} \in S} \nm{ H_{\bm{\gamma}, \bm{\hat{\lambda}}} }_{L_{\bar{\varrho}}^2(\bbR^{d_{\cX}} \setminus B_R)} \\
        &= s^{1/2} \nmd{p}_{L_{\hat{\varrho}}^2(\bbR^{d_{\cX}}; \bbR^{d_{\cY}})} \max_{\bm{\gamma} \in S} \nm{ H_{\bm{\gamma}, \bm{\hat{\lambda}} } }_{L_{\bar{\varrho}}^2(\bbR^{d_{\cX}} \setminus B_R)}. 
    \end{split}
    \end{align}
    We first bound the norm of \( p \). By the triangle inequality, we have
    \begin{equation}
    \label{eq: split p according to projected F}
        \nmd{p}_{L_{\hat{\varrho}}^2(\bbR^{d_{\cX}}; \bbR^{d_{\cY}})} 
        \leq \bar{\Upsilon}
        + \nm{\hEY \circ F \circ \hDX}_{L_{\hat{\varrho}}^2(\bbR^{d_{\cX}}; \bbR^{d_{\cY}})}.
    \end{equation}
    Lipschitz continuity implies
    \[
    \nm{\hEY \circ F \circ \hDX(\bm{x})}_{2} \leq \nm{\hEY \circ F \circ \hDX(\bm{0})}_{2} + L \nm{\bm{x}}_2
    \leq \nm{F(0)}_{\cY} + L \nm{\bm{x}}_2, \quad \forall \bm{x} \in \bbR^{d_{\cX}}.
    \]
    Integrating over \( \bbR^{d_{\cX}} \) with respect to \( \hat{\varrho} \) yields
    \begin{equation}
    \label{eq: projected F L2 hat-upsilon bd}
        \nm{ \hEY \circ F \circ \hDX }_{L_{\hat{\varrho}}^{2}(\bbR^{d_{\cX}}; \bbR^{d_{\cY}})}
        \leq \nm{F(0)}_{\cY} + L  \sqrt{\sum_{i = 1}^{d_{\cX}} \hat{\lambda}_i}.
    \end{equation}
    Next, we derive pointwise bounds for the Hermite polynomials \(H_{\bm{\gamma}, \bm{\hat{\lambda}}}\) with \(\bm{\gamma} \in S \). We use the estimate
    \(\abs{H_n(x)} \leq \left( 3 \max\{1, \abs{x}\} \right)^n\)
    for every \(x \in \bbR\), \(n \in \bbN_0\), see~\cite[eq. (2.8)]{schwab_DeepLearningHighDimension_2023}. 
    With this, we compute for \( \bm{x} \in \bbR^{d_{\cX}} \) and \( \bm{\gamma} \neq \bm{0} \)
    \begin{align*}
        \abs{H_{\bm{\gamma}, \bm{\hat{\lambda}} }(\bm{x})}^2 
        &= \prod_{i = 1}^{d_{\cX}} \abs{ H_{\gamma_i} \bigg( \frac{x_i}{\sqrt{\hat{\lambda}_i}} \bigg) }^2
        \leq \prod_{i = 1}^{d_{\cX}} \left( 3^2 \max\{1, (\hat{\lambda}_i)^{-1 / 2} \absd{x_i}\}^2 \right)^{\gamma_i} \\
        &\leq 9^{\nm{\bm{\gamma}}_1} \left( \frac{1}{\nm{\bm{\gamma}}_1} \sum_{i = 1}^{d_{\cX}} \gamma_i \max\{ 1, (\hat{\lambda}_i)^{-1} \absd{x_i}^2 \} \right)^{\nm{\bm{\gamma}}_1},
    \end{align*}
    where the last step follows from the weighted arithmetic--geometric mean inequality. Now use the bounds \(\gamma_i \leq \nm{\bm{\gamma}}_1\), \(\max\{ 1, \hat{\lambda}_i^{-1} \absd{x_i}^2 \} \leq 1 + \hat{\lambda}_i^{-1} \absd{x_i}^2\), as well as \(\hat{\lambda}_i^{-1} \leq \hat{\lambda}_{d_{\cX}}^{-1} \leq (1 - \kappa)^{-1} \lambda_{d_{\cX}}^{-1}\) by~\eqref{eq: hat(lambda) estimate} to further estimate
    \[
    \abs{H_{\bm{\gamma}, \bm{\hat{\lambda}} }(\bm{x})}^2 
    \leq 9^{\nm{\bm{\gamma}}_1} \left( d_{\cX} + (1 - \kappa)^{-1} \lambda_{d_{\cX}}^{-1} \nm{\bm{x}}_2^2 \right)^{\nm{\bm{\gamma}}_1}.
    \]
    Note that this bound also holds for \( \bm{\gamma} = \bm{0} \).
    By definition, \( \nm{\bm{\gamma}}_1 \leq m(S) \) for every \( \bm{\gamma} \in S \). Set \(K := (1 - \kappa) \lambda_{d_{\cX}}\) and \( T_R := d_{\cX} + K^{-1} R^2 \). For brevity, we write \( m := m(S) \). Recall from Step 2 that \( (\hEX)^{-1}(\bbR^{d_{\cX}} \setminus B_R) \subset \cX \setminus B_R^{\cX} \). 
    Therefore, 
    \begin{align*}
        &\int_{\bbR^{d_{\cX}} \setminus B_R} \absd{H_{\bm{\gamma}, \bm{\hat{\lambda}}}(\bm{x})}^2 \D \bar{\varrho}(\bm{x})\\
        &\leq \int_{\cX \setminus B_R^{\cX}} 9^{m} \left( d_{\cX} + K^{-1} \nm{X}_{\cX}^2 \right)^{m} \D \mu(X) \\
        &\leq 9^{m} \int_{0}^{T_R} m \, t^{m - 1} \mu(\left\{ X \in \cX : \nm{X}_{\cX} > R \right\}) \D t \\
        &\hspace{8ex} + 9^{m} \, m \int_{T_R}^{\infty} t^{m - 1} \mu \left( \left\{X \in \cX : d_{\cX} + K^{-1} \nm{X}_{\cX}^2 > t \right\} \right) \D t \\
        &\leq 9^{m} \, T_R^m \, 4 \, e^{-R^2 / 8} +  9^{m} m \cdot 4 \int_{T_R}^{\infty} t^{m - 1} e^{-(t - d_{\cX}) K / 8} \D t \\
        &= 9^m (d_{\cX} + K^{-1} R^2)^m \cdot 4 \, e^{-R^2 / 8}
        + 4 \, e^{K d_{\cX} / 8} \, m \left( \frac{72}{K} \right)^m \Gamma\left( m, \frac{K d_{\cX} + R^2}{8} \right), 
    \end{align*}
    where we used~\eqref{eq: tail bound for mu} in the third step. 
    From the assumption \(\sum_{i = 1}^{\dim(\cX)} \lambda_i = 1\) and the fact that the \( \lambda_i \) are ordered, it follows that \(\lambda_{d_{\cX}} \leq 1 / d_{\cX}\). Together with \(\kappa < 1/2\), we deduce that \( \lambda_{d_{\cX}} / 2 \leq K \leq 1 / d_{\cX} \). Recall further that \( R \geq 1 \). Consequently,
    \begin{align*}
        \max_{\bm{\gamma} \in S } \nm{ H_{\bm{\gamma}, \bm{\hat{\lambda}} } }_{L_{\bar{\varrho}}^2(\bbR^{d_{\cX}} \setminus B_R)}^2
        &\leq \left( \frac{9}{\lambda_{d_{\cX}}} \right)^m (1 + 2 R^2)^m \cdot 4 \, e^{-R^2 / 8}
        + 4 \, e^{1/8} \, m \left( \frac{144}{\lambda_{d_{\cX}}} \right)^m \Gamma\left( m, \frac{R^2}{8} \right) \\
        &\leq \left( \frac{27}{\lambda_{d_{\cX}}} \right)^m \cdot R^{2m} \cdot 4 \, e^{-R^2 / 8}
        + 4 \, e^{1/8} \, m \left( \frac{144}{\lambda_{d_{\cX}}} \right)^m \Gamma\left( m, \frac{R^2}{8} \right) \\
        &\leq 4 \left( \frac{216}{\lambda_{d_{\cX}}} \right)^m \left( \frac{R^2}{8} \right)^{m} e^{-R^2 / 8}
        + 4 \, e^{1/8} \, m \left( \frac{144}{\lambda_{d_{\cX}}} \right)^m \Gamma\left( m, \frac{R^2}{8} \right) \\
        &\leq 4 \, e^{1/8} \left( \frac{216}{\lambda_{d_{\cX}}} \right)^m \Gamma\left( m + 1, \frac{R^2}{8} \right).
    \end{align*}
    In the last step we used that \( (R^2 / 8)^m e^{-R^2 / 8} + m \Gamma(m, R^2/8) = \Gamma(m + 1, R^2 / 8) \).
    In combination with~\eqref{eq: p-R - p}--\eqref{eq: projected F L2 hat-upsilon bd}, we deduce that 
    \begin{align*}
        &\nm{p_R - p}_{L_{\bar{\varrho}}^2(\bbR^{d_{\cX}}; \bbR^{d_{\cY}})} \\
        & \hspace{8ex} \leq 2 \, e^{1/16} \sqrt{s} \, \left(\frac{216}{\lambda_{d_{\cX}}}\right)^{m/2} \sqrt{\Gamma\left(m + 1, \frac{R^2}{8} \right)}
        \left( \bar{\Upsilon} + \nm{F(0)}_{\cY} + L  \sqrt{\sum_{i = 1}^{d_{\cX}} \hat{\lambda}_i} \right) .
    \end{align*}
    For the third term in brackets, notice that Weyl's inequality implies that \( \hat{\lambda}_i \leq \nm{\Delta}_{\infty} + \lambda_i \leq (1+\kappa) \lambda_i \) for every \( i \in [d_{\cX} ] \). Therefore, \( \sum_{i = 1}^{d_{\cX}} \hat{\lambda}_i \leq (1+\kappa ) \sum_{i = 1}^{d_{\cX}} \lambda_i \leq 1 +\kappa < 2 \). We conclude that 
    \[
        \nm{p_R - p}_{L_{\bar{\varrho}}^2(\bbR^{d_{\cX}}; \bbR^{d_{\cY}})} 
        \leq 2 e^{1/16} \sqrt{s} \, \left(\frac{216}{\lambda_{d_{\cX}}}\right)^{m/2} \sqrt{\Gamma\left(m + 1, \frac{R^2}{8} \right)}
        \left( \bar{\Upsilon} + C_F \right) .
    \]
    Combining the estimates from all three steps finally completes the proof.
\end{proof}

Lemma \ref{lem: effect of measure mu to varpi} yields a bound for passing from the measure \( \mu \) to the measure \( \varpi \), assuming sufficient closeness of the true and approximate PCA eigenvalues, as quantified by the uniform norm of the term \( \Delta = \Sigma - \widehat{\Sigma}\). We now estimate this quantity, which leads to a more explicit bound. 

Recall that \( \tr(\Sigma) = 1 \) and \( \nm{\Sigma}_{\infty} = \lambda_1 \leq 1 \). The following result is an immediate consequence of~\cite[Thm. 2]{koltchinskii_AsymptoticsConcentrationBoundsBilinear_2016}.

\begin{proposition}[Bound for \(\nm{\Dmu}_{\infty}\) in probability]
\label{prop: bound for operator norm of Delta_mu}
    There exists a constant \(C_2 \geq 1\) such that for every \( \epsilon \in (0, 1/3]\) with probability at least \(1 - \epsilon\) in the draw of \(\whX_1, \dots, \whX_{N_{\cX}} \sim \mu\), we have
    \begin{equation*}
        \nmd{\Dmu}_{\infty}
        \leq C_2 \sqrt{\log(1 / \epsilon)} (N_{\cX})^{-1/2},
    \end{equation*}
    provided that \(N_{\cX} \geq \max\{ \lambda_1^{-1}, \log(1 / \epsilon) \}\).
\end{proposition}

Note that Lemma \ref{lem: effect of measure mu to varpi} introduces two parameters \( \kappa \) and \( R \). The final step in this subsection is to choose these parameters appropriately. For this we require the following two lemmas.

\begin{lemma}[Choice of \(\kappa\)]
\label{lem: choice of kappa}
    If
    \[
    \kappa \leq  \min\left\{ \frac{1}{2} \left(1 - \frac{1}{16^{1/d_{\cX}}} \right), \frac{\lambda_{d_{\cX}}}{R^2} \right\} \quad \textup{and} \quad R \geq 2,
    \]
    then
    \[
    \left( \frac{1 - \kappa}{1 - 2\kappa} \right)^{d_{\cX} / 4}
    \exp \left( \frac{R^2}{4 \lambda_{d_{\cX}}} \frac{\kappa}{(1 - \kappa)(1 - 2 \kappa)} \right) \leq 4, \quad \forall d_{\cX} \in \bbN.
    \]
\end{lemma}

\begin{proof}
    Observe that if \(\kappa \leq \frac{1}{2} (1 - \frac{1}{16^{1/d_{\cX}}}) \), then \( (\frac{1 - \kappa}{1 - 2 \kappa})^{d_{\cX} / 4} \leq 2\). Since \(\lambda_{d_{\cX}} \leq 1\), the condition \(R \geq 2\) implies \(\kappa \leq 1/4\). Therefore, 
    \[
    \exp\left( \frac{R^2}{4 \lambda_{d_{\cX}}} \frac{\kappa}{(1 - \kappa)(1 - 2 \kappa)} \right)
    \leq \exp\left( \frac{1}{4(1 - \frac{1}{4})(1 - \frac{1}{2})} \right) \leq 1.95,
    \]
    and the claim follows.
\end{proof}

\begin{lemma}[Choice of \(R\)]
\label{lem: choice of R}
    Let \( \Upsilon > 0 \) and \( S \subset \bbN^{d_{\cX}}_0 \) be such that \( s := \absd{S} \geq 1 \) and \( m(S) := \max_{\bm{\gamma} \in S} \nm{\bm{\gamma}}_1 \geq 3 \). If
    \begin{equation}
    \label{eq: lower bound on R^2}
        R^2 \geq 118 \left( m(S) \log(m(S)) + \log(s) + \abs{\log(\Upsilon)} + m(S) \abs{\log(\lambda_{d_{\cX}} ) } \right),
    \end{equation}
    then 
    \[
    R e^{-R^2 / 16} \leq \Upsilon \quad \text{and} \quad
    \sqrt{s} \left( \frac{216}{\lambda_{d_{\cX}}} \right)^{ m(S)/2 } \sqrt{\Gamma \left( m(S) + 1, \frac{R^2}{8} \right)}
    \leq \Upsilon.
    \]
\end{lemma}

\begin{proof}
    Assume~\eqref{eq: lower bound on R^2} and write \( m = m(S) \) for brevity.
    Setting \( y := -R^2/8 \), we first note that \( R e^{-R^2/16} \leq \Upsilon \) if and only if
    \begin{equation}
    \label{eq: optimality condition for R 0}
        y e^y \geq -\frac{\Upsilon^2}{8}.
    \end{equation}
    If \( - \Upsilon^2 / 8 \leq -1 / e \), then~\eqref{eq: optimality condition for R 0} is true for every \( y \). Otherwise, \( - \Upsilon^2 / 8 \in ( -1 / e, 0 ) \). Since \( y \leq -1 \),~\eqref{eq: optimality condition for R 0} is equivalent to 
    \(
    y \leq W_{-1}( - \Upsilon^2 / 8 ),
    \)
    where  \( W_{-1} \) denotes the secondary branch of the Lambert \(W\) function.
    It is an easy exercise to verify the lower bound
    \[
    W_{-1}(x) \geq \log(-x) - \log(-\log(-x)) - \frac{1}{2}, \quad \forall x \in (-1 / e, 0).
    \]
    Rearranging terms and using the inequality \(\log(\absd{\log(x)}) \leq \abs{\log(x)}\) for \(x > 0\), \( x \neq 1 \), we find that~\eqref{eq: optimality condition for R 0} holds true if 
    \begin{equation}
    \label{eq: optimality condition for R 01}
        R^2 \geq 32 \abs{\log(\Upsilon)} + 16 \log(8) + 4.
    \end{equation}
    
    We are left with showing that 
    \begin{equation}
    \label{eq: optimality condition for R}
        \sqrt{s} \left( \frac{216}{\lambda_{d_{\cX}}} \right)^{ m/2 } \sqrt{\Gamma \left( m + 1, \frac{R^2}{8} \right)}
        \leq \Upsilon.
    \end{equation}
    To this end, let us write \(r := R^2 / 8\). We use the following bound from~\cite[Prop. 2.7]{pinelis_ExactLowerUpperBounds_2020} for the upper incomplete Gamma function:
    \[
    \Gamma(m + 1, x)
    \leq \frac{x^{m} e^{-x}}{1 - m / x}, \quad \textup{ provided that } x > m.
    \]
    We henceforth assume that \(r \geq 2 m\), or equivalently,
    \begin{equation}
    \label{eq: lower bound for R, II}
        R^2 \geq 16 m.
    \end{equation}
    Consequently,  \(\Gamma(m + 1, r) \leq 2 r^{m} e^{-r}\) and~\eqref{eq: optimality condition for R} holds true if 
    \begin{equation*}
        r^{m} e^{-r} \leq \frac{1}{2} \Upsilon^{2} s^{-1} \left( \frac{216}{\lambda_{d_{\cX}}} \right)^{-m} =: a.
    \end{equation*}
    Equivalently, 
    \begin{equation}
    \label{eq: optimality condition for r, I}
        y e^y \geq - \frac{a^{1 / m}}{m} =: b \quad \textup{with} \quad y := - \frac{r}{m} \leq -2.
    \end{equation}
    If \( b \leq -1/e \), then~\eqref{eq: optimality condition for r, I} holds true for every \( y \). If \( b \in (- 1 / e, 0) \), then~\eqref{eq: optimality condition for r, I} is equivalent to 
    \(
    y \leq W_{-1}( b ).
    \)
    Arguing similarly as above, we find that~\eqref{eq: optimality condition for r, I} holds true if
    \begin{align*}
        r &\geq 2\Big [ m \log(m)  + \log(216) m + \log(s)  + 2 \absd{\log(\Upsilon)} + m \absd{\log(\lambda_{d_{\cX}})} + \log(2) \Big ] + \frac{m}{2}.
    \end{align*}
    Since \( \log(m) > 1 \) for \(m \geq 3 \), this in turn follows from~\eqref{eq: lower bound on R^2}. To conclude, we note that~\eqref{eq: lower bound on R^2} also implies~\eqref{eq: optimality condition for R 01} and~\eqref{eq: lower bound for R, II}.
\end{proof}

We can now present the main result of this subsection. 
\begin{theorem}
\label{thm: effect of measure mu to varpi full}
    Let \( \epsilon \in (0,1/3] \), \( 0 < \Upsilon < 1 \), and \( S \subset \bbN^{d_{\cX}}_0 \) be such that \( s := |S| \geq 1 \) and \( m(S) := \max_{\bm{\gamma} \in S} \nm{\bm{\gamma}}_1 \geq 3 \). Suppose that
    \begin{equation}
    \label{eq: NX lower bound}
        N_{\cX} \geq  C^2_2 \log(1/\epsilon) \kappa^{-2} (\lambda_{d_{\cX}})^{-2} 
    \end{equation}
    with parameter \( \kappa > 0 \) satisfying
    \begin{equation}
    \label{eq: kappa final choice}
        \kappa \leq \min \bigg \{ \frac{1}{2} \left( 1 - \frac{1}{16^{1/d_{\cX}} } \right) ,
        \frac{\lambda_{d_{\cX}} }{118} \bigg( m(S) \log(m(S)) + \log(s) +
        \abs{\log(\Upsilon)} + m(S) \abs{\log(\lambda_{d_{\cX}})} \bigg)^{-1} \bigg \},
    \end{equation}
    where \(C_2 \geq 1\) is the constant from Proposition~\ref{prop: bound for operator norm of Delta_mu}. Then the following holds with probability at least \(1-\epsilon \) in the draw of \( \whX_1,\ldots,\whX_{N_{\cX}} \sim \mu \):
    
    Let \( F \) be \(L\)-Lipschitz and let \( P = \hDY \circ p \circ \hEX \) for some \( p \in \cP_{\bbR^{d_{\cY}}} \), with \( \cP_{\bbR^{d_{\cY}}} \) given by \eqref{eq: hat-f subspace def}, and set 
    \[
    \bar{\Upsilon} := \nmd{P - \hDY \circ \hEY \circ F \circ \hDX \circ \hEX}_{L_{\varpi}^2(\cX; \cY)} = \nmd{p - \hEY \circ F \circ \hDX }_{L^2_{\hat{\varrho}}(\bbR^{d_{\cX}} ; \bbR^{d_{\cY}}) }.
    \]
    Then
    \[
    \nmd{P - \hDY \circ \hEY \circ F \circ \hDX \circ \hEX}_{L_{\mu}^2(\cX; \cY)}
    \leq 4 \bar{\Upsilon} + 5 \Upsilon ( \bar{\Upsilon} + \wtC_F )
    \]
    where \( \wtC_F  := \nm{F(0)}_{\cY} + (\sqrt{8} + 1) L \).
\end{theorem}

\begin{proof}
    Fix \( \kappa > 0 \) so that~\eqref{eq: kappa final choice} is satisfied and suppose that \( N_{\cX} \) satisfies~\eqref{eq: NX lower bound}. This implies, in particular, that \( N_{\cX} \geq \max\{ \lambda_1^{-1}, \log(1/\epsilon) \} \). By Proposition \ref{prop: bound for operator norm of Delta_mu}, we then have \( \| \Delta \|_{\infty}/\lambda_{d_{\cX}}  \leq \kappa \) with probability at least \( 1 - \epsilon \). We now apply Lemma~\ref{lem: effect of measure mu to varpi} and invoke Lemma~\ref{lem: choice of kappa} and Lemma~\ref{lem: choice of R} with \( R^2 = m(S) \log(m(S)) + \log(s) + \abs{\log(\Upsilon)} + m(S) \) to conclude the claim.
\end{proof}

\subsection{Final arguments}

We are now, finally, ready to finish the proof of Theorem \ref{thm: Hermite-PCA error bound}. This involves combining Theorem \ref{thm: empirical PCA projection errors}, which bounds the empirical PCA projection errors \eqref{eq: encoding error: general} and \eqref{eq: decoding error: general}, with Theorem \ref{thm: effect of measure mu to varpi full}, which together with \eqref{eq: least squares error bound} yields a bound for the approximation error \eqref{eq: approximation error: general}.

\begin{proof}[Proof of Theorem \ref{thm: Hermite-PCA error bound}]
    We split the error according to \eqref{eq: total error splitting}. Theorem \ref{thm: empirical PCA projection errors} with \( \epsilon \) replaced by \( \epsilon / 6 \) and the choices
    \[
    N_{\cX} \geq C_1 d_{\cX} \log(12/\epsilon) \Upsilon^{-4} \quad \text{and}\quad N_{\cY} \geq C_1 d_{\cY} \log(12/\epsilon) \Upsilon^{-4}
    \]
    yield that
    \begin{equation*}
        \mathrm{Err}_{\cX} \leq L \sqrt{\sum^{\dim(\cX)}_{i = d_{\cX}+1} \lambda_i} + L K_{\mu} \Upsilon
    \end{equation*}
    with probability at least $1-\epsilon/6$ in the draw of \( \whX_{1},\ldots,\whX_{N_{\cX}} \)
    and
    \begin{equation*}
        \mathrm{Err}_{\cY} \leq  \sqrt{\sum^{\dim(\cY)}_{i = d_{\cY}+1} \lambda^{F \sharp \mu}_i} + (\nm{F(0)}_{\cY} + L K_{\mu} + \sigma ) \Upsilon + 2 \sqrt{2} \sigma
    \end{equation*}
    with probability at least $1-\epsilon/6$ in the draw of \( \whY_{1},\ldots,\whY_{N_{\cY}} \). 
    
    Next, consider \( \mathrm{Err}_{\cA} \). Since \( M \geq C_{\delta} s \log(12 s / \epsilon ) \) by assumption, using \eqref{eq: least squares error bound} and applying a simple change of variables, we infer that
    \begin{align*}
        \nm{\hDY \circ \hEY \circ F  \circ \hDX \circ \hEX  - \whF }_{L^2_{\varpi}(\cX ; \cY)} \leq \widetilde{\Upsilon}
    \end{align*}
    with probability at least $1-\epsilon / 2$ in the draw of \( X_1,\ldots,X_M \). Then, Theorem \ref{thm: effect of measure mu to varpi full} with \( P = \whF \), \( \epsilon \)~replaced by \( \epsilon / 6 \), together with \eqref{eq: condition for N_X} and \eqref{eq: condition for kappa} imply that
    \[
    \mathrm{Err}_{\cA} = \nm{\hDY \circ \hEY \circ F  \circ \hDX \circ \hEX  - \whF }_{L^2_{\mu}(\cX ; \cY)} \leq 4 \widetilde{\Upsilon} + 5 \Upsilon ( \widetilde{\Upsilon} + \wtC_F).
    \]
    The result now follows from the union bound.
\end{proof}


\section{\texorpdfstring{\( \ell^2 \)-characterization of Gaussian Sobolev spaces}{l2-characterization of Gaussian Sobolev spaces}}
\label{sec: l2-characterization}

In this and the next section, we aim to use Theorem \ref{thm: Hermite-PCA error bound} to prove the main result, Theorem \ref{thm: Hermite-PCA Sobolev error bound}. The key step is the analysis of the best approximation error 
\begin{equation*}
\inf_{p \in \cP_{\bbR^{d_{\cY}}}} \nm{\hEY \circ F \circ \hDX- p }_{L^2_{\hat{\varrho}}(\bbR^{d_{\cX}} ; \bbR^{d_{\cY}})} 
\end{equation*}
and the derivation of the choice \eqref{eq: S hat-tau def} for the index set \(S\). This analysis relies crucially on the \(\ell^2\)-characterization of the Sobolev spaces \( H^k_{\varrho}(\bbR^d ; \bbR^{d'} ) \) in terms of the Hermite polynomial coefficients and on the analysis of the decay of the associated weight sequences that arise in that characterization. We address these aspects of the analysis in this section, before completing the proof of Theorem \ref{thm: Hermite-PCA Sobolev error bound} in the next section.

We fix dimensions \(d, d' \in \bbN\). Recall the notation \([d] = \{1, \dots, d\}\) and \([\infty] = \bbN\). As before, write \( \bm{\lambda} = (\lambda_i)_{i=1}^d \) and let \( \varrho = \cN(0, \bm{\lambda}) \) be a centered, nondegenerate Gaussian measure on \( \bbR^{d} \) with diagonal covariance matrix with diagonal entries $\lambda_i$.

\subsection{\texorpdfstring{\(\ell^2\)-characterization of \( H^k_{\varrho}(\bbR^d ; \bbR^{d'} ) \)}{l2-characterization of Hk\_{varrho}(Rd;Rd')}}

We first require some notation for lowering the entries of multi-indices.
For \(i \in \bbN\) and \(\bm{\gamma} \in \bbN_0^{d}\), we set \(\bm{\gamma^{(i)}} := \bm{0}\) if \(\gamma_i = 0\). If \(\gamma_i > 0\), we define 
\[
\gamma_k^{(i)} := 
\begin{cases}
    \gamma_k - 1 & \textup{ if } k = i, \\
    \gamma_k & \textup{ otherwise}.
\end{cases}
\]
We recursively define \(\bm{\gamma^{(i,j)}} := \bm{(\gamma^{(i)})^{(j)}}\). 
For \(i_1, \dots, i_n \in \bbN\), we set \(\bm{\gamma^{(i_1, \dots, i_{j-1})}} := \bm{\gamma}\) if \(j = 1\). We also extend this definition to the case \( d = \infty \) by replacing \( \bbN_0^d \) with \( \Gamma \).

\begin{proposition}[Partial derivatives]
    Let \(f \in C_b^n(\bbR^d; \bbR^{d'})\) with \( n \in \bbN \). For \(i_1, \dots, i_n \in [d]\), we have
    \begin{equation}
    \label{eq: derivative identity}
        \partial_{i_n} \cdots \partial_{i_1} f 
        = \sum_{\bm{\gamma} \in \bbN_0^{d}} \left( \prod_{r = 1}^n \frac{\gamma_{i_r}^{(i_1, \dots, i_{r-1})}}{\lambda_{i_r}} \right)^{1/2}
        \cdot \int_{\bbR^d} f H_{\bm{\gamma}, \bm{\lambda}} \D \varrho 
        \cdot H_{\bm{\gamma^{(i_1, \dots, i_n)}}, \bm{\lambda}}.
    \end{equation}
    In particular,
    \begin{equation}
    \label{eq: norm of derivatives}
        \nm{\partial_{i_n} \cdots \partial_{i_1} f}_{L_{\varrho}^2(\bbR^d; \bbR^{d'})}^2
        = \sum_{\bm{\gamma} \in \bbN_0^{d}} \left(\prod_{r = 1}^n \frac{\gamma_{i_r}^{(i_1, \dots, i_{r-1})}}{\lambda_{i_r}}\right) \nm{\int_{\bbR^d} f H_{\bm{\gamma}, \bm{\lambda}} \D \varrho }_{2}^2.
    \end{equation}
\end{proposition}

\begin{proof}
    It suffices to prove~\eqref{eq: derivative identity}, as~\eqref{eq: norm of derivatives} then follows by Parseval's identity. Let \( \alpha_j \) count the occurrences of \( j \) in  \( (i_1, \dots, i_n) \) and set \( \bm{\alpha} := (\alpha_j)_{j = 1}^d \). To prove~\eqref{eq: derivative identity}, it is enough to show that 
    \[
    \int_{\bbR^d} (\partial_{i_n} \cdots \partial_{i_1} f) H_{\bm{\gamma^{(i_1, \dots, i_n)}, \bm{\lambda}}} \D \varrho 
    = \left( \prod_{r = 1}^n \frac{\gamma_{i_r}^{(i_1, \dots, i_{r-1})}}{\lambda_{i_r}} \right)^{1/2}
    \cdot \int_{\bbR^d} f H_{\bm{\gamma}, \bm{\lambda}} \D \varrho, \quad \forall \bm{\gamma} \geq \bm{\alpha}.
    \]
    For \( d' = n = 1 \), this holds by~\cite[Lem. 10.14]{daprato_IntroductionInfiniteDimensionalAnalysis_2006}. The proof of the latter lemma also provides the argument for the induction step to derive the claim for general \( n \in \bbN \). It is formulated for \( d = \infty \) but holds verbatim in the case \( d \in \bbN \) as well. The case \( d' > 1 \) follows by considering the components of \( f \in C_b^n(\bbR^d; \bbR^{d'}) \) and applying the arguments in the scalar-valued case.
\end{proof}

Next, let \( i_1, \dots, i_n \in [d] \). By~\eqref{eq: norm of derivatives}, the product \( \prod_{r = 1}^n \gamma_{i_r}^{(i_1, \dots, i_{r-1})} / \lambda_{i_r} \) is the weight associated with the derivative \( \partial_{i_n} \cdots \partial_{i_1} f \). In view of~\eqref{eq: Hk-norm} in the definition of the \( H_{\varrho}^{k} \)-norm in Appendix~\ref{app: Sobolev}, we define the weights
\begin{equation}
\label{eq: weight formula}
    v_{\bm{\gamma}, (k), d} := \left( 1 + \sum_{j = 1}^k \sum_{i_1, \dots, i_j = 1}^d \prod_{r = 1}^j \frac{\gamma_{i_r}^{(i_1, \dots, i_{r-1})}}{\lambda_{i_r}} \right)^{-1/2}
\end{equation}
for \( k \in \bbN \) and \( \bm{\gamma} \in \bbN_0^{d} \).
We extend this definition to \(d = \infty\) by replacing \(\bbN_0^d\) by \(\Gamma\). Note that \( \gamma_{i_r}^{(i_1, \dots, i_{r-1})} = 0 \) if \( i_r  \not\in \supp(\bm{\gamma}) \), so that the inner sum is finite even if \( d = \infty \).

\begin{proposition}[{\(\ell^2\)-characterization of \(H_{\varrho}^{k}\)}]
\label{prop: l2-characterization of W_mu^(k,2)}
    Let \(f \in H_{\varrho}^{k}(\bbR^d; \bbR^{d'})\). Then, for \(i_1, \dots, i_n \in [d] \) with \( 1 \leq n \leq k \) we have
    \[
    \partial_{i_n} \cdots \partial_{i_1} f 
    = \sum_{\bm{\gamma} \in \bbN_0^{d}} \left( \prod_{r = 1}^n \frac{\gamma_{i_r}^{(i_1, \dots, i_{r-1})}}{\lambda_{i_r}} \right)^{1/2}
    \cdot \int_{\bbR^d} f H_{\bm{\gamma}, \bm{\lambda}} \D \varrho
    \cdot H_{\bm{\gamma^{(i_1, \dots, i_{n})}}, \bm{\lambda}} \quad \textup{in} \quad L^2_{\varrho}((\bbR^d; \bbR^{d'}))
    \]
    and
    \[
    \nm{f}_{H_{\varrho}^{k}(\bbR^d; \bbR^{d'})}^2 
    = \sum_{\bm{\gamma} \in \bbN_0^{d}} v_{\bm{\gamma}, (k), d}^{-2} \, \nm{\int_{\bbR^d} f H_{\bm{\gamma}, \bm{\lambda}} \D \varrho}_2^2.
    \]
    Conversely, if for a family of vectors \((\bm{y}_{\bm{\gamma}})_{\bm{\gamma} \in \bbN_0^{d}} \subset \bbR^{d'}\) there holds
    \[
    \sum_{\bm{\gamma} \in \bbN_0^{d}} v_{\bm{\gamma}, (k), d}^{-2} \, \nm{\bm{y}_{\bm{\gamma}}}_{2}^2 < \infty,
    \]
    then 
    \[
    f := \sum_{\bm{\gamma} \in \bbN_0^{d}} \bm{y}_{\bm{\gamma}} H_{\bm{\gamma}, \bm{\lambda}} \in H_{\varrho}^{k}(\bbR^{d}; \bbR^{d'}).
    \]
\end{proposition}

\begin{proof}
    This is a straightforward inductive generalization of Proposition C.8 in~\cite{adcock_SampleComplexityLearningLipschitz_2026}. It is formulated for \( d = \infty \), but holds verbatim in the case \( d \in \bbN \) as well.
\end{proof}

\subsection{Upper bounds for the weights}
\label{subsec: upper bounds for the weights}

We now derive an upper bound for the decay of the weights \(v_{\bm{\gamma}, (k), d}\) in~\eqref{eq: weight formula}. 

We start by noting \( v_{\bm{\gamma}, (k), d} \leq ( 1 + \sum_{i = 1}^d \gamma_i / \lambda_i )^{-1/2} \), which implies that all the superlevel sets \( \{\bm{\gamma} \in \bbN_0^d : v_{\bm{\gamma}, (k), d} > \varepsilon \} \) with \( \varepsilon > 0 \) are finite.
Consequently, there exists a bijection \(\tau_{(k),d} : \bbN \to \bbN_0^{d}\) that gives a nonincreasing rearrangement of the weights \(v_{\bm{\gamma}, (k), d}\), i.e., \(v_{\bm{\tau_{(k), d}(1)}, (k), d} \geq v_{\bm{\tau_{(k), d}(2)}, (k), d} \geq \cdots > 0\), with \( \lim_{s \to \infty} v_{\bm{\tau_{(k), d}(s)}, (k), d} = 0 \). This mapping is unique up to permutations of multi-indices carrying the same weight.
To lighten notation, we write \( v_{\bm{\tau_{(k), d}(i)}, (k), d} = v_{\bm{\tau_{(k), d}(i)}} \).

For brevity, we further introduce the following notational conventions.
If \(d = \dim(\cX)\), we write \(v_{\bm{\gamma}, (k), \dim(\cX)} = v_{\bm{\gamma}, (k)}\) and \(\tau_{(k), \dim(\cX)} = \tau_{(k)}\).
If in addition \(k = 1\), we write \(v_{\bm{\gamma}, (1), \dim(\cX)} = v_{\bm{\gamma}}\) and \(\tau_{(1), \dim(\cX)} = \tau\). If \( d = \infty \), the same notation applies with \( \bbN_0^d \) replaced by \( \Gamma \). Note that this notation is consistent with~\eqref{eq: v_gamma},~\eqref{eq: tau} and the nonincreasing rearrangement \( \tau \) is well-defined because \( \lambda_i \to 0 \).

If the \(\lambda_i\) are replaced by their empirical counterparts \( \hat{\lambda}_i \), we define the weights \(\hat{v}_{\bm{\gamma}, (k), d}\) together with a nonincreasing rearrangement \(\hat{\tau}_{(k), d}\) analogously with the same notational conventions as before. The only difference is that we write \(\hat{v}_{\bm{\gamma}, (k), d_{\cX}} = \hat{v}_{\bm{\gamma}, (k)}\) and \(\hat{\tau}_{(k), d_{\cX}} = \hat{\tau}_{(k)}\). That is, we drop the dimension in the index if \( d = d_{\cX} \) instead of \( d = \dim(\cX) \). If in addition \(k = 1\), we write \(\hat{v}_{\bm{\gamma}, (1), d_{\cX}} = \hat{v}_{\bm{\gamma}}\) and \(\hat{\tau}_{(1), d_{\cX}} = \hat{\tau}\). Note that this is consistent with the notation in~\eqref{eq: hat(v)_gamma},~\eqref{eq: hat(tau)}.

We proceed with two lemmas that we will later need in Proposition~\ref{prop: higher rates on finite-dim domains} to bound the best approximation error in terms of the weight \( v_{\bm{\tau(s + 1)}}^k \). 

\begin{lemma}[Weight compatibility across dimensions]
\label{lem: weight compatibility across dimensions}
    Let \( d', d \in \bbN \cup \{\infty\} \) with \(d' \leq d\). For any nonincreasing sequence \(\lambda_1 \geq \lambda_2 \geq \cdots > 0\) (with \( \lambda_i \to 0 \) if \( d = \infty \) or \(d' = \infty \)), we have
    \(
    v_{\bm{\tau_{(1), d'}(i)}} \leq v_{\bm{\tau_{(1), d}(i)}}
    \)
    for every \( i \in \bbN \).
\end{lemma}

\begin{proof}
    We only consider the case \( d < \infty \). The argument in the infinite-dimensional case is analogous. 
    Let \(\underline{\bm{\gamma'}}\) denote the extension of \(\bm{\gamma'} \in \bbN_0^{d'}\) by zeros to an element of \(\bbN_0^{d}\). Then \(v_{\bm{\gamma'}, (1), d'} = v_{\underline{\bm{\gamma'}}, (1), d}\). Hence, every weight \(v_{\bm{\gamma'}, (1), d'}\) with \(\bm{\gamma'} \in \bbN_0^{d'} \) appears among the weights \(v_{\bm{\gamma}, (1), d}\) with \(\bm{\gamma} \in \bbN_0^{d}\). Using the notation \(\{ \cdot \}_b\) to denote a multiset, we deduce that
    \[
    A := \{ v_{\bm{\gamma'}, (1), d'} : \bm{\gamma'} \in \bbN_0^{d'} \}_b \subset 
    \{ v_{\bm{\gamma}, (1), d} : \bm{\gamma} \in \bbN_0^d \}_b =: B.
    \]
    Next, fix some \(i \in \bbN\). By definition of \(\tau_{(1), d'}\), we have 
    \(
    v_{\bm{\tau_{(1), d'}(1)}}, \dots, v_{\bm{\tau_{(1), d'}(i)}} \geq v_{\bm{\tau_{(1), d'}(i)}}.
    \)
    As all weights on the left-hand side belong to \(A\) and therefore to \(B\), the latter has at least \(i\) elements whose values are at least \(v_{\bm{\tau_{(1), d'}(i)}}\). In particular, this holds for its \(i\)-th largest element, i.e.,
    \(v_{\bm{\tau_{(1), d}(i)}} \geq v_{\bm{\tau_{(1), d'}(i)}}\). As \(i\) was arbitrary, the claim follows.
\end{proof}

\begin{lemma}[Weight order preservation under rearrangements]
\label{lem: weight order invariance under rearrangements}
    Let \(\bm{a} = (a_i)_{i \in \cI}\), \(\bm{b} = (b_i)_{i \in \cI}\) be two sequences with countable index set \(\cI\). Suppose that there exist bijections \(\tau, \pi : \bbN \to \cI\) that are nonincreasing rearrangements of \(\bm{a}\) and \(\bm{b}\), respectively. If \( a_i \leq b_i \) for every \( i \in \cI \), then
    \(
    a_{\tau(i)} \leq b_{\pi(i)}
    \)
    for every \( i \in \bbN \).
\end{lemma}

\begin{proof}
    Fix some \(i \in \bbN\) and set \(t := a_{\tau(i)}\). By definition of \(\tau\), we have \( a_{\tau(1)}, \dots, a_{\tau(i)} \geq t \). By assumption, this implies \( b_{\tau(1)}, \dots, b_{\tau(i)} \geq t \). Consequently, there exist at least \(i\) indices \(j_1, \dots, j_i \in \cI\) such that \(b_{j_n} \geq t\) for every \( n \in [i] \). In particular, this must hold for the \(i\)-th largest element, i.e., \(b_{\pi(i)} \geq t\). As \(i\) was arbitrary, the claim follows.
\end{proof}

Finally, we conclude this section by showing that in finite dimensions the weights \( v_{\bm{\gamma}, (k), d} \) associated with Sobolev order \(k\) are bounded by the \(k\)-th power of corresponding weights \( v_{\bm{\gamma}, (1), d} \) associated with Sobolev order one. Crucially, this holds for the rearranged multi-indices subject to the nonincreasing rearrangement \( \tau_{(1),d} \) corresponding to the order-one Sobolev weights. This is a key step in establishing the spectral convergence properties of the Hermite-PCA approximation, as it justifies the \(k\)-independent choice \eqref{eq: S hat-tau def} of the index set \( S \).

\begin{proposition}[Weight compatibility across Sobolev orders]
\label{prop: weight compatibility across Sobolev orders}
    Let \(d \in \bbN\), \(\lambda_1 \geq \lambda_2 \geq \cdots \geq \lambda_d > 0\) and set \( \Lambda := \sum_{i = 1}^d \lambda_i^{-1} \). For \( k \in \bbN \) and every \( \bm{\gamma} \in \bbN_0^d \), we have
    \[
    v_{\bm{\gamma}, (k), d} \leq C_{k, \Lambda} v_{\bm{\gamma},(1),d}^k, \qquad  C_{k, \Lambda} := \max\left\{ (4 k)^{k / 2}, (1 + 2(k-1) \Lambda)^{k/2} \right\}.
    \]
    In particular, taking \( \bm{\gamma} = \bm{\tau_{(1), d}(s)} \) yields
    \[
    v_{\bm{\tau_{(1), d}(s)}, (k), d} \leq C_{k, \Lambda} v_{\bm{\tau_{(1), d}(s)}}^k, \quad \forall s \in \bbN.
    \]
\end{proposition}
\begin{proof}
    If \( k = 1 \), the claim is immediate. So assume that \( k \geq 2 \). Fix \( \bm{\gamma} \in \bbN_0^d \) and write
    \[
    S_j(\bm{\gamma}) := \sum_{i_1, \dots, i_j = 1}^d \prod_{r = 1}^j \frac{\gamma_{i_r}^{(i_1, \dots, i_{r-1})}}{\lambda_{i_r}}, \quad
    \textup{so that }  v_{\bm{\gamma}, (k), d} = \left( 1 + \sum_{j = 1}^k S_j(\bm{\gamma}) \right)^{-1/2}. 
    \]
    In particular, we have \( S_1(\bm{\gamma}) = \sum_{i = 1}^d \gamma_i / \lambda_i \).

    Next, we split the coordinates by size. Let
    \[
    A := \{ i \in [d] : \gamma_i \geq k \},
    \qquad S_A(\bm{\gamma}) := \sum_{i \in A} \frac{\gamma_i}{\lambda_i},
    \qquad S_{A^c}(\bm{\gamma}) := S_1(\bm{\gamma}) - S_A(\bm{\gamma}).
    \]
    Since \( \gamma_i \leq k - 1 \) for \( i \not\in A \), we have
    \begin{equation}
    \label{eq: bound for S-A-c}
        S_{A^c}(\bm{\gamma}) \leq (k - 1) \sum_{i \not\in A} \lambda_i^{-1} \leq (k-1) \Lambda.
    \end{equation}
    We use the large coordinates in \( A \) to derive a lower bound for \( S_k \). For \( i_1, \dots, i_k \in A \), we have \( \gamma_{i_r} \geq k \) for every \( r \in [k] \). Consequently,
    \[
    \gamma_{i_r}^{(i_1, \dots, i_{r-1})} \geq \gamma_{i_r} - (k-1) \geq \gamma_{i_r} \left( 1 - \frac{k-1}{k} \right) = \frac{\gamma_{i_r}}{k}.
    \]
    Multiplying over \( r \in [k] \) and summing over \( (i_1, \dots, i_k) \in A^k \) gives
    \begin{equation}
    \label{eq: bound for S-k}
        S_k(\bm{\gamma}) \geq \sum_{i_1, \dots, i_k \in A} \prod_{r = 1}^k \frac{\gamma_{i_r} / k}{\lambda_{i_r}} 
        = k^{-k} \left( \sum_{i \in A} \frac{\gamma_i}{\lambda_i} \right)^k
        = k^{-k} S_A^k.
    \end{equation}
    
    We now consider the following three exhaustive cases: 

    \noindent \emph{Case 1:} \( S_1(\bm{\gamma}) \leq 1 \). Then \( v_{\bm{\gamma}, (k), d}^{-2} \geq 1 \geq 2^{-k} ( 1 + S_1)^k \).

    \noindent \emph{Case 2:} \( S_1(\bm{\gamma}) > 1 \) and \(S_A(\bm{\gamma}) \geq (k-1) \Lambda \). By~\eqref{eq: bound for S-A-c}, we have \( S_{A^c}(\bm{\gamma}) \leq S_A(\bm{\gamma}) \), so that \( S_A(\bm{\gamma}) \geq S_1(\bm{\gamma}) / 2 \). Moreover, \( 1 + S_1(\bm{\gamma}) \leq 2 S_1(\bm{\gamma}) \). In combination with~\eqref{eq: bound for S-k}, we conclude that
    \[
    v_{\bm{\gamma}, (k), d}^{-2} \geq S_k(\bm{\gamma}) \geq k^{-k} \left( \frac{S_1(\bm{\gamma})}{2} \right)^k
    \geq (2k)^{-k} \left( \frac{1 + S_1(\bm{\gamma})}{2} \right)^k = (4k)^{-k} v_{\bm{\gamma},(1),d}^{-2k}.
    \]
    
    \noindent \emph{Case 3:} \( S_1(\bm{\gamma}) > 1 \) and \( S_A(\bm{\gamma}) < (k-1) \Lambda \). Using~\eqref{eq: bound for S-A-c}, we find
    \(
    S_1(\bm{\gamma}) < 2 (k-1) \Lambda =: T.
    \)
    Consequently, 
    \[
    v_{\bm{\gamma}, (k), d}^{-2} \geq 1 \geq (1 + T)^{-k} (1 + S_1(\bm{\gamma}))^k 
    = (1 + T)^{-k} v_{\bm{\gamma}, (1), d}^{-2k}.
    \]

    All three cases give
    \[
    v_{\bm{\gamma}, (k), d}^{-2} \geq C_{k, \Lambda} v_{\bm{\gamma}, (1), d}^{-2k}, \qquad
    C_{k, \Lambda} := \min\left\{ (4k)^{-k}, (1 + T)^{-k} \right\},
    \]
    for every \( \bm{\gamma} \in \bbN_0^d \). Taking square roots and inverting yields the claim.
\end{proof}


\newpage

\section{Proof of Theorem \ref{thm: Hermite-PCA Sobolev error bound}}
\label{sec: main res proof}

In this section, we prove the main result of the paper, i.e., Theorem \ref{thm: Hermite-PCA Sobolev error bound}.

\subsection{Analysis of the best approximation error}

With the results of Section~\ref{sec: l2-characterization} at hand, we first bound the best approximation error
\begin{equation*}
    \inf_{p \in \cP_{\bbR^{d_{\cY}}}} \nm{\hEY \circ F \circ \hDX - p }_{L^2_{\hat{\varrho}}(\bbR^{d_{\cX}} ; \bbR^{d_{\cY}})} 
\end{equation*}
where \( \cP_{\bbR^{d_{\cY}}} := \left \{ \sum_{\bm{\gamma} \in S} \bm{c}_{\bm{\gamma}} H_{\bm{\gamma},\bm{\hat{\lambda}}} : \bm{c}_{\bm{\gamma}} \in \bbR^{d_{\cY}} \right \} \subseteq L^2_{\hat{\varrho} } (\bbR^{d_{\cX} } ; \bbR^{d_{\cY}} ) \) and \( S := \hat{\tau}([s]) \) are given by~\eqref{eq: hat-f subspace def} and~\eqref{eq: S hat-tau def}, respectively, with \( \hat{\tau} \) as in~\eqref{eq: hat(tau)}. Crucially, the following result bounds the error in terms of the true weights~\( v_{\bm{\gamma}} \).

\begin{proposition}
\label{prop: higher rates on finite-dim domains}
    Let \( k \in \bbN \) and suppose that \( 0 \leq \kappa < 1/2 \), where \( \kappa \) is the parameter in~\eqref{eq: Dmu and kappa}. For every \( s \in \bbN \) and every realization of \( \hDX, \hEY \), we then have 
    \[
    \inf_{p \in \cP_{\bbR^{d_{\cY}}}} \nm{ \hEY \circ F \circ \hDX - p }_{L^2_{\hat{\varrho}}(\bbR^{d_{\cX}} ; \bbR^{d_{\cY}})}
    \leq \bigg( 6 + 6 k \sum_{i = 1}^{d_{\cX}} \lambda_i^{-1} \bigg)^{k/2} v_{\bm{\tau(s + 1)}}^k \nm{ \hEY \circ F \circ \hDX }_{H_{\hat{\varrho}}^{k}(\bbR^{d_{\cX}}; \bbR^{d_{\cY}})}.
    \]
\end{proposition}

\begin{proof}
    We define \(p^* := \sum_{\bm{\gamma} \in \hat{\tau}([s])} c_{\bm{\gamma}} H_{\bm{\gamma}, \bm{\hat{\lambda}}}\) with \(c_{\bm{\gamma}} = \int_{\bbR^{d_{\cX}}} (\hEY \circ F \circ \hDX) H_{\bm{\gamma}, \bm{\hat{\lambda}}} \D \hat{\varrho}\). 
    Then, by Parseval's identity,
    \begin{align*}
        \inf_{p \in \cP_{\bbR^{d_{\cY}}}} \nm{\hEY \circ F \circ \hDX - p }_{L^2_{\hat{\varrho}}(\bbR^{d_{\cX}} ; \bbR^{d_{\cY}})}^2
        &\leq \nm{ \hEY \circ F \circ \hDX - p^*}_{L^2_{\hat{\varrho}}(\bbR^{d_{\cX}} ; \bbR^{d_{\cY}})}^2
        = \sum_{i = s + 1}^{\infty} \nmd{c_{\bm{\hat{\tau}(i)}}}_2^2 \\ 
        &\leq \hat{v}_{\bm{\hat{\tau}(s+1)}}^{2 k} \sum_{i = s + 1}^{\infty} \hat{v}_{\bm{\hat{\tau}(i)}}^{-2 k} \nmd{c_{\bm{\hat{\tau}(i)}}}_2^2.
    \end{align*}
    We now apply Proposition~\ref{prop: weight compatibility across Sobolev orders} with \( d \) replaced by \( d_{\cX} \) and \( \lambda_i \) replaced by \( \hat{\lambda}_i \). To this end, define \( \widehat{\Lambda} := \sum_{i = 1}^{d_{\cX}} \hat{\lambda}_i^{-1} \) so that \( C_{k,\widehat{\Lambda}} = \max\{ (4k)^{k/2}, (1 + 2(k-1) \widehat{\Lambda})^{k/2} \}\). Then
    \[
    \sum_{i = s + 1}^{\infty} \hat{v}_{\bm{\hat{\tau}(i)}}^{-2 k} \nmd{c_{\bm{\hat{\tau}(i)}}}_2^2
    \leq C_{k,\widehat{\Lambda}}^2 \sum_{i = 1}^{\infty} \hat{v}_{\bm{\hat{\tau}(i)},(k),d_{\cX}}^{-2} \nmd{c_{\bm{\hat{\tau}(i)}}}_2^2, \quad \forall s \in \bbN.
    \]
    By changing the order of summation and applying Proposition~\ref{prop: l2-characterization of W_mu^(k,2)}, with \( d \) replaced by \( d_{\cX} \) and \( \varrho \) replaced by \( \hat{\varrho} \), we see that the right-hand side is equal to
    \[
    C_{k,\widehat{\Lambda}}^2 \sum_{i = 1}^{\infty} \hat{v}_{\bm{\hat{\tau}_{(k), d_{\cX}}(i)},(k),d_{\cX}}^{-2} \nmd{c_{\bm{\hat{\tau}_{(k),d_{\cX}}(i)}}}_2^2
    = C_{k,\widehat{\Lambda}}^2 \, \nmd{ \hEY \circ F \circ \hDX }_{H_{\hat{\varrho}}^{k}(\bbR^{d_{\cX}}; \bbR^{d_{\cY}})}^2.
    \]
    
    We are left with bounding the weight \( \hat{v}_{\bm{\hat{\tau}(s+1)}}^{2 k} \). 
    It follows from Weyl's inequality~\eqref{eq: Weyl's inequality} that
    \begin{equation}
    \label{eq: Weyl bound on hat-lambda}
        (1-\kappa) \lambda_i \leq \hat{\lambda}_i \leq (1 + \kappa) \lambda_i, \quad \forall i \in [d_{\cX}],
    \end{equation}
    with \( 0 \leq \kappa < 1/2 \) by assumption. This implies, in particular, that \( \widehat{\Lambda} \leq (1 - \kappa)^{-1} \Lambda \leq 2 \Lambda \) with \( \Lambda := \sum_{i = 1}^{d_{\cX}} \lambda_i^{-1} \). Since \( \sum_{i = 1}^{d_{\cX}} \leq 1 \) by assumption, we further have \( \Lambda \geq 1 \). We consequently have the bound \( \whC_{k,d} \leq (4 + 4 k \Lambda)^{k/2} \). From~\eqref{eq: Weyl bound on hat-lambda} we can further derive that
    \[
    \hat{v}_{\bm{\gamma}, (1), d_{\cX}} 
    = \left( 1 + \sum_{i = 1}^{d_{\cX}} \frac{\gamma_i}{\hat{\lambda}_i} \right)^{-1/2}
    \leq (1 + \kappa)^{1/2} \left( 1 + \sum_{i = 1}^{d_{\cX}} \frac{\gamma_i}{\lambda_i} \right)^{-1/2}
    \leq \left( \frac{3}{2} \right)^{1/2} v_{\bm{\gamma}, (1), d_{\cX}}, \quad \forall \bm{\gamma} \in \bbN_0^{d_{\cX}}.
    \]
    Note that \( \tau_{(1), d_{\cX}} \) is a nonincreasing rearrangement of the scaled weights \( (3/2)^{1/2} v_{\bm{\gamma}, (1), d_{\cX}} \). Hence, applying Lemma~\ref{lem: weight order invariance under rearrangements} yields 
    \[
    \hat{v}_{\bm{\hat{\tau}(i)}} 
    = \hat{v}_{\bm{\hat{\tau}_{(1),d_{\cX}}(i)},(1),d_{\cX}} 
    \leq \left( \frac{3}{2} \right)^{1/2}v_{\bm{\tau_{(1), d_{\cX}}(i)}, (1), d_{\cX}} 
    = \left( \frac{3}{2} \right)^{1/2} v_{\bm{\tau_{(1), d_{\cX}}(i)}}, \quad \forall i \in \bbN.
    \]
    By Lemma~\ref{lem: weight compatibility across dimensions} we further have \( v_{\bm{\tau_{(1), d_{\cX}}(i)}} \leq v_{\bm{\tau_{(1), \dim(\cX)}(i)}} = v_{\bm{\tau(i)}} \). Combining all estimates finally yields the claim.
\end{proof}

\subsection{Final arguments}

\begin{proof}[Proof of Theorem \ref{thm: Hermite-PCA Sobolev error bound}]  
    We shall apply Theorem \ref{thm: Hermite-PCA error bound} with \( \delta = 1/2\) (this value is arbitrary),
    \begin{equation*}
        \Upsilon := \left ( \frac{C_1 d_{\cX} \log(12/\epsilon)}{N_{\cX}} \right )^{1/4} + \left ( \frac{C_1  d_{\cY} \log(12/\epsilon)}{N_{\cY}} \right )^{1/4},
    \end{equation*}
    and \( \kappa \) chosen as large as possible so that \eqref{eq: condition for kappa} holds. Recall that we choose the set \( S \subset \bbN_0^{d_{\cX}} \) as in \eqref{eq: S hat-tau def}, that is \( S = \hat{\tau}([s]) \) with \( s = \abs{S} \). A short argument shows that in this case \( m(S) \leq s - 1 \leq s \), where \( m(S) := \max_{\bm{\gamma} \in S} \nm{\bm{\gamma}}_1 \). Indeed, suppose that \( \nm{\bm{\gamma^*}}_1 = t \geq s \). Notice that
    \[
    \hat{v}_{\bm{\gamma^*}} = \left ( 1 + \sum^{d_{\cX}}_{i=1} \frac{\gamma^*_i}{\hat{\lambda}_i} \right )^{-1/2} \leq  \left ( 1 + \frac{\nm{\bm{\gamma^*}}_1}{\hat{\lambda}_1} \right )^{-1/2} < \left ( 1 + \frac{i}{\hat{\lambda}_1} \right )^{-1/2} = \hat{v}_{i \bm{e_1}},\quad i = 0,\ldots,s-1,
    \]
    where \( \bm{e_1} = (\delta_{1,j})_{j = 1}^{d_{\cX}} \).
    Therefore, there are at least $s$ multi-indices for which the corresponding values \( \hat{v}_{\bm{\gamma}} \) exceed \( \hat{v}_{\bm{\gamma^*}} \). It follows that \( \bm{\gamma^*} \notin S \), as required.
    
    Theorem \ref{thm: Hermite-PCA error bound} requires \( m(S) \geq 3 \). The latter condition is satisfied if \( S \) contains more elements than there are multi-indices \( \bm{\gamma} \in \bbN_0^{d_{\cX}} \) with \( \nm{\bm{\gamma}}_1 \leq 2 \). This is the case if \( \abs{S} > 1 + 2 \binom{d_{\cX}}{1} + \binom{d_{\cX}}{2} = \binom{d_{\cX} + 2}{2} \). Indeed, the term \( 1 \) accounts for the zero multi-index. The term \( 2 \binom{d_{\cX}}{1} \) counts the multi-indices of the form \( \bm{\gamma} = i \bm{e_j} \), \( i \in \{1,2\}\), \(j \in [d_{\cX}]\), and the third term counts multi-indices of the form \( \bm{\gamma} = \bm{e_i} + \bm{e_j} \) with \( i \neq j \). We thus need to choose \( s \geq \binom{d_{\cX} + 2}{2} + 1 \).
    
    To apply Theorem \ref{thm: Hermite-PCA error bound}, we further need to verify that \eqref{eq: condition for N_X}, \eqref{eq: condition for N_Y} and \eqref{eq: condition for M} hold. The latter two follow immediately from \eqref{eq: main condition for N_Y} and \eqref{eq: main condition for M} and the choices for \( \delta \) and \( \Upsilon \). We now consider \eqref{eq: condition for N_X}. First notice that by the choice of \( \kappa \),~\eqref{eq: condition for N_X} is implied by the condition
    \[
    N_{\cX} \geq  \max \left \{ I_1 , I_2 , I_3 \right \}
    \]
    where, after recalling that \( m(S) \leq s \), we have
    \begin{align*}
    I_1 &= C_1 d_{\cX} \log(12/\epsilon) \Upsilon^{-4} ,
    \\
    I_2 & = 4 C^2_2 \log(6/\epsilon) (\lambda_{d_{\cX}})^{-2} \left ( 1 - \frac{1}{16^{1/d_{\cX}} } \right )^{-2} ,
    \\
    I_3 & = 118^2 C^2_2 \log(6/\epsilon)  (\lambda_{d_{\cX}})^{-4} \left ( 2 s \log(s) + | \log(\Upsilon) | + s | \log(\lambda_{d_{\cX}}) | \right )^2.
    \end{align*}
    We now verify that \( N_{\cX} \geq I_i \), \( i = 1,2,3 \), separately. The case \( i = 1 \) follows immediately from the definition of \( \Upsilon \). For \( i = 2 \), we use the inequality \( 1-16^{-1/z} \geq 1/(2 z) \) for \( z \geq 1 \) to observe that
    \[
    I_2 \leq 16 C^2_2 \log(6/\epsilon) (\lambda_{d_{\cX}})^{-2} d^2_{\cX} .
    \]
    Hence, using \eqref{eq: main condition for N_X} we see that \( N_{\cX} \geq I_2 \) provided \( c_1 \) is sufficiently large. For \( i = 3 \), we first observe that \( \Upsilon < 1 \) due to \eqref{eq: main condition for N_X} and \eqref{eq: main condition for N_Y}, for sufficiently large \( c_1 , c_2 \). Therefore
    \[
    | \log(\Upsilon) | = \log(1/\Upsilon) \leq \frac{1}{4} \log(N_{\cX}) - \frac{1}{4} \log(\log(12)) \leq \log(N_{\cX}).
    \]
    Hence, \( I_3 \leq N_{\cX} \) by \eqref{eq: main condition for N_X} and provided \( c_1 \) is sufficiently large.
    
    Having verified the required conditions, we now apply Theorem \ref{thm: Hermite-PCA error bound} to deduce that
    \[
    \nm{F - \whF}_{L_{\mu}^2(\cX; \cY)}
    \lesssim L \sqrt{\sum^{\dim(\cX)}_{i = d_{\cX}+1} \lambda_i} +  \sqrt{\sum^{\dim(\cY)}_{i = d_{\cY}+1} \lambda^{F \sharp \mu}_i} + \sigma 
    + (\nm{F(0)}_{\cY} + L (1+K_{\mu}) + \sigma ) \Upsilon + (1 + \Upsilon) \widetilde{\Upsilon}. 
    \]
    Applying the definition of \( \Upsilon \) and the fact that \( \Upsilon < 1 \), we get
    \begin{align*}
        \nm{F - \whF}_{L_{\mu}^2(\cX; \cY)}
        & \lesssim L \sqrt{\sum^{\dim(\cX)}_{i = d_{\cX}+1} \lambda_i} +  \sqrt{\sum^{\dim(\cY)}_{i = d_{\cY}+1} \lambda^{F \sharp \mu}_i}  + \sigma 
        \\
        &\quad +(\nm{F(0)}_{\cY} + L (1+K_{\mu}) + \sigma ) \left[ \left ( \frac{d_{\cX} \log(12/\epsilon)}{N_{\cX}} \right )^{1/4} + \left ( \frac{d_{\cY} \log(12/\epsilon)}{N_{\cY}} \right )^{1/4} \right] + \widetilde{\Upsilon}.
    \end{align*}
    We now estimate the term \( \widetilde{\Upsilon} \) by applying Proposition \ref{prop: higher rates on finite-dim domains}. Recalling that  \( \kappa < 1/2 \) and \( M \geq s \) by construction, we see that
    \[
    \widetilde{\Upsilon} \lesssim \frac{(6 + 6 k \Lambda)^{k/2}}{\sqrt{\epsilon}}  v_{\bm{\tau(s + 1)}}^k \nm{ \hEY \circ F \circ \hDX }_{H_{\hat{\varrho}}^{k}(\bbR^{d_{\cX}}; \bbR^{d_{\cY}})} + \sigma\sqrt{\frac{s}{M \epsilon} }
    \]
    for every \( s \in \bbN \).
    Substituting this into the previous bound finally completes the proof.
\end{proof}


\section{Conclusions}
\label{sec: conclusions}

In this article, we presented a fully data-driven algorithm, termed~\emph{Hermite-PCA approximation}, for the recovery of Sobolev operators from pointwise, noisy samples. It is based on principal component analysis for dimension reduction and employs weighted linear least-squares fitting, making it computationally efficient. It has several important properties. First, it achieves near-optimal sample complexity under weak assumptions --- we only assume the input measure \( \mu \) to be Gaussian. Second, it does so in a spectral fashion, achieving faster convergence rates the smoother the objective operator is. Third, it can be adapted to any other regularity, e.g., mixed regularity types, while maintaining its properties. We gave a full error analysis of our algorithm, enabling optimal choices of all hyperparameters. In addition, we provided numerical results which closely match our theoretical findings. In particular, our experiments numerically illustrate, to the best of our knowledge, for the first time the curse of sample complexity which is inherent to the approximation of finitely regular operators.

There are several future research directions.
First, our algorithm requires a custom sampling distribution \( \mu_{\textup{samp}} \) of the input training data. It is based on the Christoffel function of the learning problem and is thus intrinsic to using Hermite polynomials~\cite{adcock_SparsePolynomialApproximationHighDimensional_2022}. In practice, Monte Carlo sampling from the underlying Gaussian input distribution \( \mu \) would be more favorable. It is an open question whether our error bounds also hold in the case of i.i.d. samples from \( \mu \).

Second, the bound~\eqref{eq: main condition for N_X} indicates an at least quadratic scaling of \( N_{\cX} \), the number of unlabeled data samples used in the empirical PCA encoding step, in terms of the approximation dimension~\( s \). We believe that this is an artifact of our analysis. In fact, our empirical results suggest that it can be significantly improved to a log-linear scaling. This is a topic of future work.

Third, in Appendix~\ref{app: regularity}, we identify conditions for the operator \( F \) which imply its latent space representation to be Sobolev regular. It can be shown that \( H_{\mu}^k(\cX; \cY) \)-regularity of \( F \) itself is not sufficient for the analysis of the algorithm presented in this work. It is an open question whether the Hermite-PCA algorithm can attain optimal rates for all \( H_{\mu}^k(\cX; \cY) \)-operators based on only finitely many samples from \( \mu \), or whether different algorithmic approaches are required. We mention~\cite{adcock_UniversalSampleoptimalAlgorithmsRecovery_2026} for recent results in finite dimensions.

Finally, as noted above, our algorithm can be readily adapted to other types of regularity, such as anisotropic or mixed smoothness or otherwise. An in-depth study of operator learning under other smoothness types is an interesting topic for future work.

\section*{Acknowledgments}
BA acknowledges support from the Natural Sciences and Engineering Research
Council of Canada (NSERC) through grant RGPIN/2026-04531.
MG and GM acknowledge support from the Hausdorff Center for Mathematics (HCM) in Bonn, funded by the Deutsche Forschungsgemeinschaft (DFG, German Research Foundation) under Germany’s Excellence Strategy – EXC-2047/2 – 390685813. 

\section*{Declaration on the use of generative artificial intelligence}
Generative AI was used to assist with language polishing, proofreading of the manuscript, and code development. All generated suggestions were carefully reviewed and edited by the authors.

{\small \printbibliography}

\newpage

\appendix


\section{\texorpdfstring{Definition of \( H^k_{\mu}(\cX; \cY ) \)}{Definition of Hk\_{mu}(X;Y)}}
\label{app: Sobolev}

In this appendix, we present a formal definition of the Sobolev spaces \( H^k_{\mu}(\cX ; \cY ) \) along with several results needed in the main text. We commence with some preliminary notions. 

Recall that \( \nm{\cdot}_{\infty} \) denotes the operator norm.
For \( k \in \bbN \), we denote by \( C^k_b(\cX; \cY) \) the space of bounded \(C^k\)-operators with bounded derivatives with norm 
\[
\nm{F}_{C^k_b(\cX; \cY)} := \sup_{X \in \cX} \nm{F(X)}_{\cY} + \sum_{j = 1}^k \sup_{X \in \cX} \nm{D^j F(X)}_{\infty}.
\]

\begin{definition}[Cylindrical functionals and operators]
    A functional \(\varphi: \cX \to \bbR\) is called a \emph{cylindrical functional} if there exist \(n \in \bbN\), \(\ell_1, \dots, \ell_n \in \cX^*\), and a function \(\omega: \bbR^n \to \bbR\) such that
    \[
    \varphi(X) = \omega(\ell_1(X), \dots, \ell_n(X)), \quad \forall X \in \cX.
    \]
    We call \(\varphi\) a \emph{\(k\) times boundedly Fréchet differentiable cylindrical functional} with \( k \in \bbN \cup \{\infty\} \) if, with the above notation, \( \omega \in C_b^k(\bbR^n) \). The space of all such functionals is denoted by \(\cF C_b^k(\cX)\).
    Moreover, we define the space of all \emph{\(k\) times boundedly Fréchet differentiable cylindrical \(\cY\)-valued operators} by 
    \[
    \cF C_b^k(\cX; \cY) := \spn\left\{ \cX \ni X \mapsto \varphi(X) Y \in \cY: \varphi \in \cF C_b^k(\cX), Y \in \cY \right\}.
    \]
\end{definition}

Let \( \HS_k(\cX; \cY) \) denote the set of all \( k \)-linear operators \( F : \cX^k \to \cY \) with finite Hilbert-Schmidt~norm
\[
\nm{F}_{\HS_k(\cX; \cY)} = \left( \sum_{i_1, \dots, i_k = 1}^{\dim(\cX)} \nm{F(\phi_{i_1}, \dots, \phi_{i_k})}_{\cY}^2 \right)^{1/2},
\]
where we recall that \( \{\phi_i\}_{i = 1}^{\dim(\cX)} \) is an orthonormal PCA basis of \( \cX \). In fact, the Hilbert-Schmidt norm is independent of the chosen orthonormal basis of \( \cX \). We set \( \HS_0(\cX; \cY) = \cY \).

\begin{definition}[The space \( H_{\mu}^{k}(\cX; \cY) \)]
    The space \( H_{\mu}^{k}(\cX; \cY) \) is defined as the completion of the space \( \cF C_b^{k}(\cX; \cY) \) under the Sobolev norm
    \[
    \nm{ F }_{H_{\mu}^{k}(\cX; \cY)} := \left( \nm{F}_{L_{\mu}^2(\cX; \cY)}^2 + \sum_{j = 1}^k \nm{D^j F}_{L_{\mu}^2(\cX; \HS_j(\cX; \cY))}^2 \right)^{1/2},
    \]
    which can be equivalently written as
    \begin{equation}
    \label{eq: Hk-norm}
        \nm{ F }_{H_{\mu}^{k}(\cX; \cY)} 
        = \left( \int_{\cX} \nm{F}_{\cY}^2 \D \mu + \sum_{j = 1}^k \int_{\cX} \sum_{i_1, \dots, i_j = 1}^{\dim(\cX)} \nm{\partial_{i_j} \cdots \partial_{i_1} F }_{\cY}^2 \D \mu \right)^{1/2}.
    \end{equation}
    Let \( (\varphi_n)_{n = 1}^{\infty} \subset \cF C_b^{k}(\cX; \cY) \) be a sequence of cylindrical operators such that \( \lim_{n \to \infty} \varphi_n = F \) in \( H_{\mu}^{k}(\cX; \cY) \). For \( j \in [k] \), we define the weak Gaussian derivatives
    \[
    D^j F := \lim_{n \to \infty} D^j \varphi_n \quad \text{ in }  L_{\mu}^2(\cX; \HS_j(\cX; \cY)).
    \]
\end{definition}

The weak Gaussian derivatives \( D^j F \) are well-defined in the sense that if two sequences \( \{\varphi_i\}_i \), \( \{ \widetilde{\varphi}_i\}_i \) in \( \cF C_b^k(\cX; \cY) \) are Cauchy in \( H_{\mu}^{k}(\cX; \cY) \) and converge to \( F \) in \( L_{\mu}^2(\cX; \cY) \), then the sequences \( \{ D^j \varphi_i \}_i \), \( \{ D^j \widetilde{\varphi}_i \}_i \) have equal limits in \( L^2_{\mu}(\cX; \HS_j(\cX; \cY)) \). We refer to~\cite[Sec. 5.2]{bogachev_GaussianMeasures_1998} for details. We remark that our definition of the space \( H_{\mu}^{k}(\cX; \cY) \) deviates from the one therein. We consider the Fréchet derivative along the full space \(\cX\) whereas in~\cite{bogachev_GaussianMeasures_1998} the derivative is taken along the Cameron-Martin space \(H(\mu)\) of \( \mu \). The proof for well-definedness of the weak derivatives is, however, analogous in both cases. If \( \dim(\cX) < \infty \), then the two resulting Sobolev spaces coincide with equivalent norms, since in this case we have \( H(\mu) = \cX \).

As usual, we write \( \partial^j_i F = D^j F(\phi_i, \dots, \phi_i) \) for the \(j\)-th Gaussian (weak) partial derivative in direction \( \phi_i \). We set \( D^0 F = F \). Moreover, for a multi-index \( \bm{\alpha} \in \Gamma \), we write \( D^{\bm{\alpha}} = \prod_{i = 1}^{\infty} \partial^{\alpha_i}_i \).

\begin{remark}
    An equivalent definition of \( H_{\mu}^{k}(\cX; \cY) \) can be given via closability of Fréchet differential operators up to order \(k\). We refer to~\cite{adcock_SampleComplexityLearningLipschitz_2026,daprato_IntroductionInfiniteDimensionalAnalysis_2006,lunardi_InfiniteDimensionalAnalysis_2015} for details.
\end{remark}

In finite dimensions, Gaussian Sobolev functions can be characterized by classical Sobolev functions. To this end, as in the main text, write \( \bm{\lambda} = (\lambda_i)_{i = 1}^d \), \( d \in \bbN \), with \( \lambda_i > 0\), and let \( \varrho = \cN(0, \bm{\lambda}) \) denote the centered, nondegenerate Gaussian measure on \( \bbR^{d} \) with diagonal covariance matrix with diagonal entries \( \lambda_i \).

\begin{proposition}[Classical vs. Gaussian weak derivatives]
\label{prop: classical vs. Gaussian weak derivatives}
    The Sobolev space \( H_{\varrho}^{k}(\bbR^{d}; \bbR^{d'}) \) consists of all functions \( f \in H_{\textup{loc}}^{k}(\bbR^{d}; \bbR^{d'}) \) such that \( f \in L_{\varrho}^2(\bbR^{d}; \bbR^{d'}) \) and \( D^j f \in L_{\varrho}^2(\bbR^{d}; \HS_j(\bbR^{d}; \bbR^{d'})) \) for~\( j \in [k] \). The corresponding classical and Gaussian weak derivatives coincide.
\end{proposition}

\begin{proof}
    This is~\cite[Prop. 1.5.2]{bogachev_GaussianMeasures_1998}, which considers the isotropic, scalar-valued case with \( \lambda_1 = \dots = \lambda_{d} = 1 \) and \( d' = 1 \). The proof holds for the anisotropic case as well, and the vector-valued case follows by applying the scalar-valued result to each component function.
\end{proof}


\section{Optimal approximation of Gaussian Sobolev operators}\label{app: lower bounds}

We show that the worst-case approximation error for \( H_{\mu}^k(\cX; \cY) \)-operators based on \( s \) potentially adaptively chosen linear samples is bounded from below by the quantity \( v_{\bm{\tau(s + 1)}}^k \). To this end, we first recall the notion of the adaptive \(s\)-width and then adapt arguments from~\cite{adcock_SampleComplexityLearningLipschitz_2026}.

\begin{definition}[Adaptive \(s\)-width]
    Let \((\cV,\nm{\cdot}_{\cV})\) be a normed vector subspace of \(L_{\mu}^2(\cX;\cY)\) and let \(\cK \subset \cV\) be a subset.
    The adaptive \(s\)-width of \(\cK\) in \(\cV\) is given by
    \begin{align*}
    \begin{split}
        &\Theta_s(\cK;\cV,L_{\mu}^2(\cX; \cY)) \\
        &:= \inf\left\{
        \sup_{F \in \cK} \nm{F - \cT(\cL(F))}_{L_{\mu}^2(\cX; \cY)} : \cL: \cV \to \cY^s \textup{ adaptive}, \cT: \cY^s \to L_{\mu}^2(\cX; \cY)
        \right\},
    \end{split}
    \end{align*}
    where the infimum is taken over all adaptive sampling operators \( \cL \) and all (arbitrary) reconstruction operators \( \cT \). A mapping \( \cL = ( \cL_i)_{i = 1}^s : \cV \to \cY^s \) is called an adaptive (Hilbert-valued) sampling operator if  \(\cL_1: \cV \to \cY \) is bounded and mapping, if \(\cL_i: \cV \times \cY^{i-1} \to \cY \) is bounded and linear in the first component for \(i = 2, \dots, s\), and if the following condition holds:
    There exist \(Y, \wtY \in \cY \setminus \{0\}\), a normed vector space \(\widetilde{\cV} \subset L_{\mu}^2(\cX)\), and a scalar-valued adaptive sampling operator \(\widetilde{\cL}: \widetilde{\cV} \to \bbR^m\) such that, if \(Y F \in \cV\) for some \(F \in L_{\mu}^2(\cX)\), then \(F \in \widetilde{\cV}\) and  \(\cL(Y F) = \wtY \widetilde{\cL}(F)\).
    Further details can be found in~\cite[Sec~5.1]{adcock_SampleComplexityLearningLipschitz_2026}. 
\end{definition}

\begin{remark}[Bounded operators for pointwise sampling]
\label{rmk: boundedness of F}
    If \( \cV = C_b(\cX; \cY) \), then an example of an adaptive sampling operator is given by pointwise sampling \( \cL(F) = ( F(X_i) )_{i = 1}^s \), where the sample points \( X_i \in \cX \) may be chosen adaptively. In this case, the operators in the set \( \cK \subset \cV \) are necessarily bounded.
\end{remark}

The adaptive \(s\)-width of a set \(\cK\) is the smallest worst-case error achievable when every operator in \( \cK \) is reconstructed from \(s\) (possibly adaptively) chosen Hilbert-valued samples.
We are interested in a specific choice for \( \cK \), namely the unit ball of \( \cF C_b^k\) with respect to the \( H^k_{\mu} \)-norm
\[
\cF B_{\mu}^k (\cX; \cY)
:= \left\{ F \in \cF C_b^k(\cX; \cY) : \nm{F}_{H_{\mu}^k(\cX; \cY)} \leq 1 \right\}.
\]

For simplicity, we only consider the case \( \dim(\cX) = \infty \) in the remainder of this section. All results analogously hold in the case \( \dim(\cX) < \infty \) as well.
Recall from~\eqref{eq: weight formula} the weights \( v_{\bm{\gamma}, (k), d} \), which for \( d = \infty \) have the form
\[
v_{\bm{\gamma}, (k), \infty} := \left( 1 + \sum_{j = 1}^k \sum_{i_1, \dots, i_j = 1}^{\infty} \prod_{r = 1}^j \frac{\gamma_{i_r}^{(i_1, \dots, i_{r-1})}}{\lambda_{i_r}} \right)^{-1/2}, \quad \forall k \in \bbN, \bm{\gamma} \in \Gamma,
\]
together with a nonincreasing rearrangement \( \tau_{(k), \infty} : \bbN \to \Gamma \). Recall also the notational conventions introduced at the beginning of Section~\ref{subsec: upper bounds for the weights}. In particular, we write \( v_{\bm{\gamma}, (k)} = v_{\bm{\gamma}, (k), \infty} \) and \( \tau_{(k)} = \tau_{(k), \infty} \) as well as \( v_{\bm{\gamma}} = v_{\bm{\gamma}, (1)} \) and \( \tau = \tau_{(1)} \).

\begin{theorem}[Lower bound for the adaptive \( s \)-width]
\label{thm: lower width bound}
    For every subspace \( \cV \subset L^2_{\mu}(\cX; \cY) \) (possibly with norm stronger than \( L^2_{\mu} \)) which contains \( \cF B_{\mu}^k (\cX; \cY) \), we have
    \[
    \Theta_s(\cF B_{\mu}^k (\cX; \cY); \cV, L_{\mu}^2(\cX; \cY)) \geq v_{\bm{\tau_{(k)}(s + 1)}}
    \]
    for every \( s \in \bbN \).
\end{theorem}
\begin{proof}
    We provide a sketch of the argument and leave the details to the reader. The main idea is to construct a suitable \( \cF C_b^k(\cX; \cY) \)-operator from Hermite polynomials. 
    To this end, recall from~\eqref{eq: Hermite polynomials 1D} the one-dimensional Hermite polynomials \( H_n \), \( n \in \bbN_0 \).
    Similarly as in~\eqref{eq: Hermite polynomials multi-D}, we define for \( \bm{\gamma} \in \Gamma \) and \( \bm{\lambda} = (\lambda_i)_{i = 1}^{\infty} \) the Hermite polynomials on \( \cX \) as
    \[
    H_{\bm{\gamma},\bm{\lambda},\cX} : \cX \to \bbR, \quad H_{\bm{\gamma},\bm{\lambda},\cX}(X) := \prod_{i = 1}^{\infty} H_{\gamma_i} \left( \frac{\ip{X}{\phi_i}_{\cX}}{\sqrt{\lambda_i}} \right).
    \]
    The family \( \{ H_{\bm{\gamma}, \bm{\lambda}, \cX} \}_{\bm{\gamma} \in \Gamma} \) is an orthonormal basis of \( L^2_{\mu}(\cX) \), see~\cite[Thm. 9.7]{daprato_IntroductionInfiniteDimensionalAnalysis_2006}. Following the proof of~\cite[Lem. C.6]{adcock_SampleComplexityLearningLipschitz_2026}, we find that \( H_{\bm{\gamma}, \bm{\lambda}, \cX} \in H^k_{\mu}(\cX) \) and we can derive explicit expressions for the weak Gaussian derivatives \( D^{\bm{\alpha}} H_{\bm{\gamma}, \bm{\lambda}, \cX} \), \( \bm{\alpha} \in \Gamma \). \\ \indent 
    Next, let \( \eta \in C_c^{\infty}(\bbR) \) be a smooth bump function with \( \eta = 1 \) on \( (-1,1) \) and \( \supp(\eta) \subset (-2,2) \) such that \( \nm{\eta}_{L^{\infty}(\bbR)} \leq 1 \). 
    Let \( R \geq 1 \) and define \( \eta_R(x) := \eta( x / R ) \), \( x \in \bbR \). We define the `damped' Hermite polynomials in one dimension
    \[
    \wtH_{n,R} := \eta_R H_n \in C_c^{\infty}(\bbR),
    \]
    and in infinite dimensions
    \[
    \wtH_{\bm{\gamma}, R, \bm{\lambda}, \cX} : \cX \to \bbR, \quad
    \wtH_{\bm{\gamma}, R, \bm{\lambda}, \cX}(X) := \prod_{i \in \supp(\bm{\gamma})} \wtH_{\gamma_i, R}\left( \frac{\ip{X}{\phi_i}_{\cX}}{\sqrt{\lambda_i}} \right).
    \]
    By construction, we have \( \wtH_{\bm{\gamma}, R, \bm{\lambda}, \cX} \in \cF C_b^{k}(\cX) \). It now follows by standard arguments based on the dominated convergence theorem that
    \[
    \lim_{R \to \infty} \int_{\cX} \abs{ D^{\bm{\alpha}} \wtH_{\bm{\gamma}, R, \bm{\lambda}, \cX}(X) - D^{\bm{\alpha}} H_{\bm{\gamma}, \bm{\lambda}, \cX}(X) }^2 \D \mu(X) = 0
    \]
    for every \( \bm{\alpha} \in \Gamma \) with \( \nm{\bm{\alpha}}_1 \leq k \).
    Consequently,
    \begin{equation}
    \label{leq: convergence in H^k}
        \lim_{R \to \infty} \wtH_{\bm{\gamma}, R, \bm{\lambda}, \cX} = H_{\bm{\gamma}, \bm{\lambda}, \cX}, \quad \textup{in } H^k_{\mu}(\cX).
    \end{equation}
    Similarly as in~\cite[Lem. 4.12]{adcock_SampleComplexityLearningLipschitz_2026} (again invoking the dominated convergence theorem), we can deduce that
    for every finite subset \( I \subset \Gamma \) and every \( \varepsilon > 0 \), there exists \( \bar{R} > 0 \) such that for every \( R \geq \bar{R} \) we have
    \begin{equation}
    \label{eq: Riesz basis condition}
        (1-\varepsilon) \nm{\bm{x}}_{\ell^2(I)}^2 
        \leq \nm{\sum_{\bm{\gamma} \in I} x_{\bm{\gamma}} \wtH_{\bm{\gamma}, R, \bm{\lambda}, \cX}}_{L_{\mu}^2(\cX)}^2 \leq (1 + \varepsilon) \nm{\bm{x}}_{\ell^2(I)}^2,
        \quad \forall \bm{x} = (x_{\bm{\gamma}})_{\bm{\gamma} \in I} \in \bbR^I.
    \end{equation}
    Using~\eqref{leq: convergence in H^k} and~\eqref{eq: Riesz basis condition}, we can now argue analogously to~\cite[Sec. 4.3.2]{adcock_SampleComplexityLearningLipschitz_2026} to conclude the claim. More specifically, we take a suitable linear combination of the damped Hermite polynomials \( \wtH_{\bm{\gamma}, R, \bm{\lambda}, \cX} \) and multiply it with a suitable vector \( Y \in \cY \). This yields an operator in \( B^k_{\mu}(\cX; \cY) \) which allows us to reduce the recovery problem for \( B^k_{\mu}(\cX; \cY) \) to a finite-dimensional recovery problem for weighted \( \ell^2 \)-balls. The error in the latter can then be analyzed explicitly via the Kolmogorov width.
\end{proof}

The next result relates the weight \( v_{\bm{\tau_{(k)}(s)}} \) to the weight \( v_{\bm{\tau(s)}} \).  

\begin{proposition}[Lower weight bound]
\label{prop: lower weight bound}
    For every \(k \in \bbN \), we have
    \[
    v_{\bm{\tau_{(k)}(s)}} \geq v_{\bm{\tau(s)}}^k, \quad \forall s \in \bbN.
    \]
\end{proposition}

\begin{proof}
    For \( \bm{\gamma} \in \Gamma \), we compute
    \begin{align*}
        v_{\bm{\gamma}, (k)}^{-2} &= 1 + \sum_{j = 1}^k \sum_{i_1, \dots, i_j = 1}^{\infty} \prod_{r = 1}^j \frac{\gamma_{i_r}^{(i_1, \dots, i_{r-1})}}{\lambda_{i_r}} 
        \leq 1 + \sum_{j = 1}^k \sum_{i_1, \dots, i_j = 1}^{\infty} \prod_{r = 1}^j \frac{\gamma_{i_r}}{\lambda_{i_r}} \\
        &= 1 + \sum_{j = 1}^k \bigg( \sum_{i = 1}^{\infty} \frac{\gamma_i}{\lambda_{i}} \bigg)^j \leq \bigg( 1 + \sum_{i = 1}^{\infty} \frac{\gamma_i}{\lambda_{i}} \bigg)^k
        = v_{\bm{\gamma}, (1)}^{-2k}.
    \end{align*}
    Applying Lemma~\ref{lem: weight order invariance under rearrangements} yields the claim. 
\end{proof}

In summary, Theorem~\ref{thm: lower width bound} and Proposition~\ref{prop: lower weight bound} constitute a lower bound for the smallest worst-case error in the approximation of bounded, continuous Gaussian Sobolev operators. In the absence of PCA projection and noise errors, it matches the high probability upper bound in Theorem~\ref{thm: Hermite-PCA Sobolev error bound} up to constants that are independent of~\( s \).


\section{\texorpdfstring{\( C^k \)-operators with admissible growth at infinity}{Ck-operators with admissible growth at infinity}}
\label{app: regularity}

We introduce a class of operators \( F \) for which (with high probability) the latent space function \( f := \hEY \circ F \circ \hDX \) belongs to \( H_{\hat{\varrho}}^{k}(\bbR^{d_{\cX}}; \bbR^{d_{\cY}}) \) and the norm can be bounded independently of the empirical PCA eigenvalues~\( \hat{\lambda}_i \), provided sufficiently many PCA samples are drawn. This can be achieved if \( F \) and its derivatives do not grow too fast at infinity. We then discuss relevant classes of such operators, most importantly the class of \(C^k\)-operators whose derivatives are Lipschitz continuous.

\subsection{Definitions and properties}

\begin{definition}[Admissible growth function]  
    Let \( c > 1 \). We call a measurable function \( \omega : [0, \infty) \to [0, \infty) \) a~\emph{function of \((\mu, c)\)-admissible growth at infinity} (or just \emph{\((\mu, c)\)-admissible}) if \(\omega\) is nondecreasing and~\( \cI_c := \int_{\cX} \omega(\sqrt{c} \nm{X}_{\cX})^2 \D \mu(X) < \infty \).
\end{definition}

\begin{example}
\label{ex: admissible growth functions}
    Monomials \(\omega(t) = t^p\), \(p \geq 0\), and linear combinations thereof with nonnegative coefficients are \((\mu, c)\)-admissible for any \( c > 1 \). By Fernique's theorem, the function \( \omega(t) = \exp(\alpha t^2) \) is \((\mu, c)\)-admissible if \( 0 \leq \alpha < 1 / (4 c \lambda_1) \), see~\cite[Prop. 4.2.6]{lunardi_InfiniteDimensionalAnalysis_2015}.
\end{example}

\begin{definition}[\(C^k\)-operators with admissible growth]
    We define the space of \emph{\(C^k\)-operators with \((\mu, c)\)-admissible growth at infinity} by 
    \begin{align*}
        C_{(\mu,c)\textup{-adm}}^k(\cX; \cY) 
        := \bigg\{ F \in C^k(\cX; \cY) : \exists \omega \ & (\mu, c) \textup{-admissible such that for all } 0 \leq j \leq k : \\
        & \hspace{10ex} \nmd{D^j F(X)}_{\infty} \leq \omega(\nm{X}_{\cX}) \textup{ for all } X \in \cX \bigg\}.
    \end{align*}
\end{definition}

Let \( F \in C_{(\mu,c)\textup{-adm}}^k(\cX; \cY) \) with corresponding \((\mu,c)\)-admissible growth function \( \omega \). Let us now consider the event \( \cA_c := \{ \hat{\lambda}_i \leq c \lambda_i, \forall i \in [d_{\cX}] \} \) and fix corresponding realizations of \( \hDX \) and \( \hEY \). By Proposition~\ref{prop: bound for operator norm of Delta_mu}, the event \( \cA_c \) holds with high probability if sufficiently many PCA samples are drawn, i.e., if \( \cN_{\cX} \) is sufficiently large. As \(\hEY\), \(\hDX\) are linear, we find by the chain rule
\begin{equation*}
    D^j f(\bm{x})(\bm{z_1}, \dots, \bm{z_j})
    = \hEY \left( D^j F(\hDX(\bm{x})) (\hDX(\bm{z_1}), \dots, \hDX(\bm{z_j})) \right),
    \quad \forall 0 \leq j \leq k, \forall \bm{x}, \bm{z_1}, \dots, \bm{z_j} \in \bbR^{d_{\cX}}.
\end{equation*}
Since \( \hEY \) is a contraction and \( \hDX \) is an isometry, we deduce that
\begin{equation}
\label{eq: operator norm bound on derivative of hat-f}
    \nm{D^j f(\bm{x})}_{\infty} 
    \leq \nm{\hEY}_{\infty} \nm{D^j F(\hDX(\bm{x}))}_{\infty}
    \leq \omega(\nmd{\hDX(\bm{x})}_{\cX})
    = \omega(\nm{\bm{x}}_2).
\end{equation}
For \( d \in \bbN \cup \{\infty\} \), let \( \cN_d := \bigotimes_{i = 1}^{d} \cN(0,1) \) denote the standard normal product measure on \( \bbR^d \). A~change of variables and monotonicity of \( \omega \) yields on \( \cA_c \)
\begin{align*}
    \int_{\bbR^{d_{\cX}}} \omega(\nm{\bm{x}}_2)^2 \D \hat{\varrho}(\bm{x}) 
    &= \int_{\bbR^{d_{\cX}}} \omega\bigg( \sqrt{\sum_{i = 1}^{d_{\cX}} \hat{\lambda}_i x_i^2} \bigg)^2 \D \cN_{d_{\cX}}(\bm{x}) \\
    &\leq \int_{\bbR^{\infty}} \omega\bigg( \sqrt{c} \sqrt{\sum_{i = 1}^{\dim(\cX)} \lambda_i x_i^2} \bigg)^2 \D \cN_{\dim(\cX)}(\bm{x})
    = \cI_c.
\end{align*}
In combination with~\eqref{eq: operator norm bound on derivative of hat-f}, it then follows from Proposition~\ref{prop: classical vs. Gaussian weak derivatives} that \( f \in H_{\hat{\varrho}}^{k}(\bbR^{d_{\cX}}; \bbR^{d_{\cY}}) \) and the corresponding strong and weak Gaussian derivatives coincide almost everywhere. Collecting derivatives and using the bound \( \nm{\cdot}_{\HS_j} \leq d_{\cX}^{j/2} \nm{\cdot}_{\infty}\) on \( \bbR^{d_{\cX}} \), we eventually find
\[
\nmd{ f }_{H_{\hat{\varrho}}^{k}(\bbR^{d_{\cX}}; \bbR^{d_{\cY}})}^2 
\leq \bigg( 1 + \sum_{j = 1}^k d_{\cX}^j \bigg) \, \cI_c \quad \textup{on } \cA_c.
\]
Example~\ref{ex: admissible growth functions} shows that we can have very fast (exponential) growth at infinity and still maintain control \( \nmd{ f }_{H_{\hat{\varrho}}^{k}} \) on \( \cA_c \) independently of the~\( \hat{\lambda}_i \).

\subsection{Examples}

We discuss two important examples of classes of admissible operators.

\emph{(i) \(C^k\)-operators with bounded derivatives.} If \( F \in C_b^k(\cX; \cY) \), then \( F \in C_{(\mu,c)\textup{-adm}}^k(\cX; \cY) \) for every \( c > 1 \) with constant admissible growth function \( \omega := \nm{F}_{C_b^k(\cX; \cY)} \). In this case, \( \cI_c = \nm{F}_{C_b^k(\cX; \cY)}^2 \), and therefore  \( f \in H_{\hat{\varrho}}^{k}(\bbR^{d_{\cX}}; \bbR^{d_{\cY}}) \) with
\[
\nmd{ f }_{H_{\hat{\varrho}}^{k}(\bbR^{d_{\cX}}; \bbR^{d_{\cY}})} \leq \bigg( 1 + \sum_{j = 1}^k d_{\cX}^j \bigg)^{1/2} \nm{F}_{C_b^k(\cX; \cY)}.
\]
This holds everywhere, not only on \( \cA_c \).

\emph{(ii) \(C^k\)-operators with Lipschitz derivatives.} For \( j \in \bbN\), let \( L_j(\cX; \cY) \) denote the space of all bounded \(j\)-linear operators \( \cX^j \to \cY \), equipped with the operator norm. We set \( L_0(\cX; \cY) = \cY \).
\begin{definition}[\(C^k\)-operators with Lipschitz derivatives]
    Let \( k \in \bbN_0 \). We define the space of all \(C^k\)-operators with Lipschitz derivatives by
    \[
    C_{\Lip}^k(\cX; \cY) := \{ F \in C^{k}(\cX; \cY) : \cX \ni X \mapsto D^j F(X) \in L_j(\cX;\cY)  \textup{ is Lipschitz for all } 0 \leq j \leq k \}.
    \]
    For \( F \in C_{\Lip}^k(\cX; \cY) \), we denote by \( [F]_{C_{\Lip}^k (\cX; \cY)} \) the maximum of the smallest Lipschitz constants of \( D^j F \), \( 0 \leq j \leq k \). We further set \( [F]_0 := \max \{\nm{D^j F(0)}_{\infty} : 0 \leq j \leq k \} \) and equip \( C_{\Lip}^k(\cX; \cY) \) with the norm
    \(
    \nm{ F }_{C_{\Lip}^k(\cX; \cY)} :=  [F]_0 + [F]_{C_{\Lip}^k(\cX; \cY)}.
    \)
\end{definition}

If \( F \in C_{\Lip}^k(\cX; \cY) \), then \( F \in C_{(\mu,c)\textup{-adm}}^k(\cX; \cY) \) for every \( c > 1 \) with admissible growth function 
\[
\omega(t) := [F]_0 + [F]_{C_{\Lip}^k(\cX; \cY)} t, \quad t \geq 0.
\]
In this case,
\[
\cI_c = \bigg( a + b \sqrt{c} \sqrt{\sum_{i = 1}^{\dim(\cX)} \lambda_i} \bigg)^2 = ( a + b \sqrt{c} )^2,
\]
where we used in the last step that \( \sum_{i=1}^{\dim(\cX)} \lambda_i = 1 \) by Assumption~\ref{ass: input distribution}. 
It follows that \( f \in H_{\hat{\varrho}}^{k}(\bbR^{d_{\cX}}; \bbR^{d_{\cY}}) \) on \( \cA_c \) with
\[
\nmd{ f }_{H_{\hat{\varrho}}^{k}(\bbR^{d_{\cX}}; \bbR^{d_{\cY}})}
\leq \bigg( 1 + \sum_{j = 1}^k d_{\cX}^j \bigg)^{1/2} \left( [F]_0 + \sqrt{c} \ [F]_{C_{\Lip}^k(\cX; \cY)} \right)
\leq \sqrt{c} \, \bigg( 1 + \sum_{j = 1}^k d_{\cX}^j \bigg)^{1/2} \nm{F}_{C_{\Lip}^k(\cX; \cY)}.
\]

\end{document}